\documentclass[10pt]{article}

\usepackage{amsmath,amsfonts,amsthm,amssymb,mathrsfs}
\usepackage[T1]{fontenc}
\usepackage[utf8]{inputenc}
\usepackage{natbib} 
\usepackage[colorlinks,backref]{hyperref}
\usepackage{xcolor}
\usepackage{fullpage}
\usepackage{mathtools}
\usepackage{subcaption}

\usepackage{authblk}
\RequirePackage{tikz}
\usetikzlibrary{fit,positioning,arrows,automata,calc}
\usetikzlibrary{decorations.pathreplacing, calc}
\usepackage{multirow}

\tikzset{
   main/.style={circle, minimum size = 10mm, thick, 
        draw =black!80, node distance = 10mm},
   box/.style={rectangle, draw=black!100}
}
\usetikzlibrary{positioning}
\usepackage{enumitem}
\setlist{  
  listparindent=\parindent,
  parsep=0pt,
  leftmargin = 5pt,
}
\usepackage{thmtools}
\usepackage{thm-restate}
\usepackage[capitalize]{cleveref}
\crefname{assumption}{Assumption}{Assumptions}
\usepackage{titlesec}
\usepackage{algorithmicx}
\usepackage{algorithm}
\usepackage{algcompatible}
\algnewcommand\algorithmicinput{\textbf{Input:}}
\algnewcommand\Input{\item[\algorithmicinput]}
\algnewcommand\algorithmiciterate{\textbf{Iterate:}}
\algnewcommand\Iterate{\item[\algorithmiciterate]}
\algnewcommand\algorithmicinitialize{\textbf{Initialize:}}
\algnewcommand\Initialize{\item[\algorithmicinitialize]}
\algnewcommand\algorithmicoutput{\textbf{Output:}}
\algnewcommand\Output{\item[\algorithmicoutput]}
\algnewcommand\RETURN{\State \algorithmicreturn}%

\titleformat*{\section}{\large\bfseries}

\titleformat*{\subsection}{\normalfont\bfseries}

 \newcommand\numberthis{\addtocounter{equation}{1}\tag{\theequation}}
\renewcommand{\O}{\mathcal{O}}
\DeclareMathOperator*{\argmin}{arg\,min}

\DeclareMathOperator*{\argsup}{arg\,sup}

 \newcommand{\assertiontwo}{\sigma^2 nr \sqrt{L}}
\newcommand{\rhatinv}{\big(\bm{\mathcal{\hat R}} \bm{\mathcal{\hat R}}\t \big)\inv}

\newcommand{\eps}{\ensuremath{\varepsilon}}
 \newcommand{\Z}{\mathbb{Z}}  \newcommand{\R}{\mathbb{R}}
\newcommand{\s}{^*}
\renewcommand{\t}{^{\top}}

\renewcommand{\l}{^{(l)}}
\newcommand{\inv}{^{-1}}

\newcommand{\p}{\mathbb{P}}

\newcommand{\per}{\U_{\perp}\U_{\perp}\t}

\renewcommand{\hat}{\widehat}

\renewcommand{\tilde}{\widetilde}
\newcommand{\err}{ \frac{\sigma \kappa \sqrt{nr}}{\lambda \sqrt{L}}}

\newtheorem{theorem}{Theorem}

\newtheorem{lemma}{Lemma}
\newtheorem{conjecture}{Conjecture}
\theoremstyle{definition}

\newtheorem{remark}{Remark}

\setcitestyle{round}
\hypersetup{%
  colorlinks=true,
  linkcolor=blue,
  citecolor = blue,
  urlcolor = blue
}

\usepackage{mathtools}
\usepackage{bm}

\newcommand{\tr}{{ \sf Tr}}

\newcommand{\al}{\bm{A}^{(l)}}

\newcommand{\U}{\bm{U}}
\newcommand{\uhat}{\bm{\hat U}}

  \newcommand{\one}{^{(1)}}
  \newcommand{\two}{^{(2)}}

\renewcommand{\sl}{\bm{S}^{(l)}}

\renewcommand{\per}{\bm{U}_{\perp}}
\renewcommand{\O}{\mathcal{O}} 
\newcommand{\rl}{\bm{R}^{(l)}}

\newcommand{\nl}{\bm{N}^{(l)}}
\renewcommand{\R}{(\R)\inv}
\renewcommand{\uhat}{\bm{\hat U}}
\newcommand{\rhatl}{\bm{\hat R}^{(l)}}
\renewcommand{\R}{(\bm{\mathcal{R}\mathcal{R}}\t)\inv}
\newcommand{\pert}{\bm{U}_{\perp}\t}
\newcommand{\maxterm}{\max\{ \sqrt{nLr},nr\}}
\newcommand{\X}{\bm{\Xi}}
\newcommand{\W}{\bm{W}}
\usepackage{multicol,caption}
\usepackage{natbib} 
\bibpunct{(}{)}{;}{a}{,}{,}
\title{Statistically and Computationally Optimal Estimation and  Inference of Common Subspaces}
\author{Joshua Agterberg}
\date{\today}
\begin{document}

\maketitle

\begin{abstract}
Given multiple data matrices, many problems in statistics and data science rely on estimating a common subspace that captures certain structure shared by all the data matrices. 
In this paper we investigate the statistical and computational limits for the common subspace model in which one observes a collection of symmetric low-rank matrices perturbed by noise, where each low-rank matrix shares the same common subspace.  Our main results identify several regimes of the signal-to-noise ratio (SNR) such that estimation and inference are statistically or computationally  optimal, and we refer to these regimes as weak SNR, moderate SNR, strong estimation SNR, and strong inference SNR.  First, we propose an estimator based on projected gradient descent initialized via spectral sum of squares and show that it achieves the optimal $\sin\Theta$ error rate under strong estimation SNR.  These results are complemented by both statistical and computational lower bounds identifying the weak and moderate estimation SNR regimes.   Next, we turn to statistical inference for the $\sin\Theta$ distance itself, and we show that our estimator has an asymptotically Gaussian distribution in the strong inference SNR regime.  Based on this limiting result we propose confidence intervals and show that they are adaptively minimax optimal in the strong inference SNR regime, where adaptivity is measured in terms of the SNR. Finally, we show that adaptive confidence intervals are information-theoretically impossible below the strong inference SNR regime. Consequently, our results unveil a novel phenomenon: despite the SNR being ``above'' the computational limit for estimation, adaptive statistical inference may still be information-theoretically impossible.  
\end{abstract}

\tableofcontents

\section{Introduction}
In many statistical problems of interest, one is provided with multiple  datasets that are posited to share some common underlying structure, and the goal is to perform estimation and inference on this shared structure. Examples include neuroscience \citep{semedo_cortical_2019}, single-cell RNA sequencing \citep{ma_optimal_2026},  and multilayer network analysis \citep{paul_spectral_2020,macdonald_latent_2022,loyal_eigenmodel_2023,lyu_latent_2023}. In many such settings, this shared structure manifests through linear-algebraic relationships between datasets, often in the form of subspaces or matrix factorizations. For instance, in statistical network analysis, it is natural to posit shared latent structure, such as common community memberships associated with the same set of vertices, across multiple observed networks \citep{paul_spectral_2020,lei_bias-adjusted_2023}. Motivated by these considerations, in this paper we study how to optimally extract such common structure from multiple symmetric matrices.

To be more precise, in this paper we consider the \emph{common subspace model}, first introduced in \citet{arroyo_inference_2021} for network data as the \emph{Common Subspace Independent Edge} (COSIE) model. Suppose one observes $L$ matrices $\{\bm{A}^{(l)}\}_{l=1}^{L}$ with each $\bm{A}^{(l)} \in \mathbb{R}^{n\times n}$ such that
\begin{align}
    \bm{A}^{(l)} &= \bm{S}^{(l)} + \bm{N}^{(l)}, \label{gosie}
\end{align}
where $\bm{S}^{(l)}$ (the signal matrix) is a symmetric rank $r$ matrix, and $\bm{N}^{(l)}$ is a symmetric noise matrix. We assume that $\bm{N}^{(l)}$ has independent subgaussian entries, with variance $\sigma^2$ on the diagonal and $\sigma^2/2$ on the off-diagonal; that is, each $\bm{N}^{(l)}$ is a Wigner matrix.  In the case that each entry is Gaussian, we say it is a ${\sf GOE}$ matrix (i.e., a member of the \emph{Gaussian orthogonal ensemble}) with variance $\sigma^2$, otherwise we say it is a Wigner matrix with variance  $\sigma^2$.

We further assume that  each $\bm{S}^{(l)}$ can be factorized via
\begin{align*}
    \bm{S}^{(l)} &= \bm{U} \bm{R}^{(l)} \bm{U}\t,
\end{align*}
where $\bm{U} \in \mathbb{R}^{n\times r}$ is an orthonormal matrix common to all signal matrices (the \emph{common subspace}), and $\bm{R}^{(l)} \in \mathbb{R}^{r \times r}$ is a symmetric full-rank matrix associated to each $\bm{S}^{(l)}$.  The matrices $\bm{R}^{(l)}$ need not be diagonal. In general, $\bm{U}$ and ${\bm{R}^{(l)}}$ are identifiable only up to a simultaneous $r\times r$ orthogonal transformation. For simplicity, we assume throughout that the rank $r$ is known.   The matrix $\bm{U} \in \mathbb{R}^{n\times r}$ can be viewed as a matrix whose rows are $r$-dimensional latent Euclidean vectors associated to all the matrices \emph{simultaneously}, and thus is an object of central importance when seeking to aggregate information from multiple datasets.

The common subspace model encompasses several widely studied network models as special cases, including the multilayer stochastic blockmodel, the multilayer mixed-membership blockmodel, and the multilayer degree-corrected mixed-membership blockmodel with common degree corrections. We refer the reader to \citet{arroyo_inference_2021} for more details on this model and its relation to other network models.   However, the common subspace model need not be restricted to network data.  Indeed, our formulation differs slightly from that work in that we assume independent subgaussian Wigner noise with common variance $\sigma^2$, which implicitly allows the signal matrices $\bm{S}^{(l)}$ to be arbitrary symmetric low-rank matrices. By contrast, when $\bm{N}^{(l)}$ consists of mean-zero Bernoulli noise, as is standard in network data applications, the entries of $\bm{S}^{(l)}$ must lie in $[0,1]$, and hence the signal matrices are necessarily constrained.

The subgaussian noise assumption serves as a stylized statistical model for studying common subspace estimation in many settings. In the special case $L=1$, the model reduces to the classical matrix denoising problem. It is well known that the information-theoretical limits of estimation in such problems are governed by the signal strength of the underlying low-rank matrix, typically characterized by the magnitude of its smallest nonzero eigenvalue. In this paper, we investigate the information-theoretical limits of estimation and inference in the common subspace model in the regime where $L \gg 1$.

\subsection{Main Contributions}
Our main contributions in this paper are multifold.  From a pragmatic perspective, one primary contribution of this paper is to develop and analyze an end-to-end procedure, projected gradient descent with spectral initialization, that aims to recover the shared subspace $\U$.  While variants of our algorithm have been analyzed previously in the literature \citep{paul_spectral_2020}, a more comprehensive statistical analysis under the common subspace model is still lacking. Furthermore, existing results often require additional assumptions on either the initialization or the structure of the signal matrices $\bm{S}^{(l)}$.  In this work we only impose minimal signal-strength assumptions, and our results are complemented by the requisite lower bounds.

In this work we study the error $\| \sin\Theta(\bm{\hat U}_t,\bm{U})\|_F$, where $\bm{\hat U}_t$ is the output of projected gradient descent after $t$ iterations;  this error is defined in \cref{sec:notation}.  This loss function computes errors modulo orthogonal transformation, thus eliminating the intrinsic nonidentifiability inherent in the model.  First, we show that $\|\sin\Theta(\bm{\hat U}_t,\bm{U})\|_F$ converges to the optimal statistical error after logarithmically many iterations under a certain signal-to-noise (SNR) condition.  Next, turning to statistical inference, we provide a novel limit theorem and propose confidence intervals for the error $\|\sin\Theta(\uhat_t,\bm{U})\|_F^2$ under a slightly stronger SNR condition.  Remarkably, we demonstrate that we are able to provide confidence intervals for the loss  without knowledge of $\bm{U}$, akin to the out-of-sample error in linear regression. 

The other main contribution of this paper is to provide a comprehensive picture of the statistical and computational limits of estimation and inference.  To the best of our knowledge, a minimax study for statistical \emph{inference} of subspaces has not been studied previously.   Our analysis reveals a novel phenomenon, which we can summarize in the following theorem.    For simplicity of presentation we assume that $L \lesssim n$, that $r = O(1)$, and that each $\bm{N}^{(l)}$ is Gaussian noise.  In order to state our result, we define the following signal-strength parameter $\lambda$ defined through the equation
\begin{align*}
    \lambda^2 &= \frac{1}{L} \lambda_{\min}\bigg( \sum_{l=1}^{L}  (\rl )^2 \bigg).
\end{align*}
Evidently, when $L = 1$, the parameter $\lambda$ is the magnitude of the smallest nonzero eigenvalue of the low-rank matrix.  
We also define the \emph{generalized condition number} via
 $  \kappa := \max_{l} \frac{\|\rl \|}{\lambda}.$
We are now prepared to state our main informal result.

\begin{theorem}[Informal Statement of Main Results] \label{thm:informal}
Suppose we observe $\{\bm{A}^{(l)}\}_{l=1}^{L}$ as in \eqref{gosie}, and assume that $\kappa$ and $r$ are bounded, and each $\nl$ is a ${\sf GOE}$ matrix with variance $\sigma^2$. Then we have the following:
\begin{itemize}
    \item (Weak Estimation SNR) When $\lambda/\sigma \ll \sqrt{\frac{n}{L}}$, estimation of $\bm{U}$ is information-theoretically impossible.
    \item (Moderate Estimation SNR) When $\frac{\sqrt{n}}{L^{1/4}} \gg \lambda/\sigma \gtrsim \sqrt{\frac{n}{L}}$, no polynomial-time estimator of $\U$ exists based on predictions from the low-degree likelihood ratio.
    \item (Strong Estimation SNR) When $\lambda/\sigma \gtrsim \frac{\sqrt{n}}{L^{1/4}}$, projected gradient descent with spectral initialization achieves the optimal error rate after logarithmically many iterations.
    \item (Weak Inference SNR) When $\frac{n}{\sqrt{L}} \gg \lambda/\sigma \gg \sqrt{\frac{n}{L}}$, minimax-optimal adaptive inference  is information-theoretically impossible.  
    \item (Strong Inference SNR) When $\lambda/\sigma \gg \frac{n}{\sqrt{L}}$, adaptive statistical inference is information-theoretically possible, and our proposed confidence intervals adaptively achieve the optimal rate.
\end{itemize}

\end{theorem}

One immediate takeaway from this result is that  the computational barrier to estimation occurs at a smaller SNR value than the information-theoretical barrier to adaptive statistical inference.  Consequently, even if optimal \emph{estimation} is feasible in polynomial time, adaptive \emph{inference} may not be possible.  However, whenever  estimation or adaptive inference is both statistically and computationally feasible, our estimators achieve the optimal rates.   These phase transition phenomena can be displayed concisely in the following diagram.

\begin{tikzpicture}[x=13cm, y=1.5cm]

  \def\xzero{0}
  \def\xone{0.15}
  \def\xtwo{0.55}  
  \def\xthree{0.75}
  \def\xend{1.0}
  \def\y{0}

  \draw[->] (0, \y) -- (1.05, \y) node[right] {SNR $=\lambda/\sigma$};

  \draw[thick] (\xzero, \y-0.05) -- (\xzero, \y+0.05);
  \draw[thick] (\xone, \y-0.05) -- (\xone, \y+0.05);
  \draw[thick] (\xtwo, \y-0.05) -- (\xtwo, \y+0.05);
  \draw[thick] (\xthree, \y-0.05) -- (\xthree, \y+0.05);

  \node[below] at (\xzero, \y) {$0$};
  \node[below] at (\xone, \y) {$\frac{\sqrt{n}}{\sqrt{L}}$};
  \node[below] at (\xtwo, \y) {$\frac{\sqrt{n}}{L^{1/4}}$};
  \node[below] at (\xthree, \y) {$\frac{n}{\sqrt{L}}$};

  \draw[decorate,decoration={brace, amplitude=6pt}] 
    ($( \xzero, \y ) + (0, 0.5)$) -- 
    ($( \xone, \y ) + (0, 0.5)$) 
    node[midway, yshift=10pt] {\scriptsize Weak Estimation SNR};

  \draw[decorate,decoration={brace, amplitude=6pt}] 
    ($( \xone, \y ) + (0, 0.5)$) -- 
    ($( \xtwo, \y ) + (0, 0.5)$) 
    node[midway, yshift=10pt] {\scriptsize Moderate Estimation SNR};

  \draw[decorate,decoration={brace, amplitude=6pt}] 
    ($( \xtwo, \y ) + (0, 0.5)$) -- 
    ($( \xend, \y ) + (0, 0.5)$) 
    node[midway, yshift=10pt] {\scriptsize Strong Estimation SNR};

  \draw[decorate,decoration={brace, amplitude=6pt, mirror}] 
    ($( \xone, \y ) + (0, -0.5)$) -- 
    ($( \xthree, \y ) + (0, -0.5)$) 
    node[midway, yshift=-10pt] {\scriptsize Weak Inference SNR};

  \draw[decorate,decoration={brace, amplitude=6pt, mirror}] 
    ($( \xthree, \y ) + (0, -0.5)$) -- 
    ($( \xend, \y ) + (0, -0.5)$) 
    node[midway, yshift=-10pt] {\scriptsize Strong Inference SNR};

\end{tikzpicture}

To the best of our knowledge, this is the first time that a problem of this form has exhibited both a ``statistical-computational'' estimation gap as well as a ``computational-statistical'' inferential gap (i.e., where the computational threshold for estimation lies below the statistical threshold for adaptive inference).  In contrast, many existing models often have the property that being ``above'' the computational threshold for estimation results in statistical inference ``for free,'' in the sense that optimal inference comes with no additional SNR requirements, such as in tensor data analysis \citep{xia_inference_2022}.

\subsection{Paper Organization}
The rest of this paper is organized as follows.  In the subsequent section we introduce our main methodology.  In \cref{sec:estimation} we study the estimation error of our procedure, including an upper bound for each iteration and minimax and computational lower bounds.  In \cref{sec:inference} we study the information-theoretical limits for (adaptive) statistical inference, and we provide minimax-optimal confidence intervals as well as information-theoretical lower bounds.  In \cref{sec:relatedwork} we discuss related work, and in \cref{sec:numerical} we consider applications to real and simulated data. We finish with a discussion.  Proofs of main results can be found in the appendices.

\subsection{Notation}
\label{sec:notation}
For two sequences $a_n$ and $b_n$, we write $a_n \ll b_n$ if $a_n/b_n \to 0$ as $n \to \infty$, and we write $a_n \lesssim b_n$ if there is a constant $C > 0$ such that $a_n \leq C b_n$.  We will also write $a_n = o(b_n)$ if $a_n \ll b_n$ and $a_n = O(b_n)$ if $a_n \lesssim b_n$. We write $a_n \asymp b_n$ if both $a_n \lesssim b_n$ and $b_n \lesssim a_n$.  For a function $f: \mathbb{R}^{n \times r} \to \mathbb{R}$, we let $\nabla f$ denote its gradient viewed as an $n \times r$ matrix, and we let $\nabla^2 f$ denote its Hessian, viewed as either an operator on $n \times r$ matrices or a $nr \times nr$ matrix, where the distinction is clear from context.  

For a general matrix $\bm{M}$ we let $\|\bm{M}\|$ and $\| \bm{M} \|_F$ denote its spectral and Frobenius norms respectively, and we let ${\sf Tr}(\bm{M})$ denote its trace.  We let ${\sf SVD}_r(\bm{M})$ denote the leading $r$ left singular vectors of the matrix $\bm{M}$.  For two orthonormal matrices $\bm{U}_1$ and $\bm{U}_2$ of dimension $n \times r$ with $r \leq n$, we let $\sin\Theta( \bm{U}_1, \bm{U}_2)$ denote the matrix of singular values of $(\bm{U}_1)_{\perp}\t \bm{U}_2$, where $(\bm{U}_1)_{\perp}$ denotes any matrix with orthonormal columns such that $(\bm{U}_1)_{\perp}\t\bm{U}_1 = 0$. In particular, we have that
\begin{align}
    \| \sin\Theta(\bm{U}_1,\bm{U}_2) \|_F = \| (\bm{U}_1)_{\perp}\t \bm{U}_2 \|_F.
\end{align}
We let $\mathcal{O}_{\bm{U}_1,\bm{U}_2}$ denote the Frobenius-optimal orthogonal matrix aligning $\bm{U}_1$ and $\bm{U}_2$; that is,
\begin{align*}
    \mathcal{O}_{\bm{U}_1,\bm{U}_2} := \argmin_{\mathcal{O} : \mathcal{OO}\t = \bm{I}_r} \| \bm{U}_1 \mathcal{O} - \bm{U}_2 \|_F.
\end{align*}
For orthonormal $\bm{U}_1$ and $\bm{U}_2$, it is well-known that $\mathcal{O}_{\bm{U}_1,\bm{U}_2}$ can be computed from the product of the left and right singular vectors of $\bm{U}_1\t \bm{U}_2$.  

For a random variable $X$, we let $\|X\|_{\psi_\alpha}$ denote its Orlicz $\psi_{\alpha}$ norm.  See \citet{vershynin_high-dimensional_2018} for more details.  For a sequence of random variables $\{X_n\}_{n=1}^{\infty}$, we say $X_n \xrightarrow{d} X$ if $X_n$ converges to $X$ in distribution.  Similarly, we write $X_n \xrightarrow{p} X$ if $X_n$ converges in probability to $X$.  For two random variables $X$ and $Y$, we write $\mathbb{E} X$ as the expected value of $X$ and $\mathbb{E} X| Y$ as the conditional expectation of $X$ given $Y$. 

\section{Projected Gradient Descent with Spectral Initialization}
In this section we introduce our main algorithm to estimate the matrix $\bm{U}$.  Define the loss function
\begin{align*}
    h\big( \bm{\hat U}; \{ \bm{\hat R}^{(l)}\}_{l=1}^{L} \big) := \frac{1}{4L} \sum_{l=1}^{L} \| \bm{A}^{(l)} - \bm{\hat U} \bm{\hat R}^{(l)} \bm{\hat U}\t \|_F^2.
\end{align*}
Up to possible rescaling, $h$ can be viewed as the objective function whose minimizer is the maximum likelihood estimator under Gaussian noise.  However, due to the constraint that $\bm{U}\t \U = \bm{I}_r$, globally optimizing this function can be NP-hard in general.  

We therefore propose to optimize the loss function $h$ by projected gradient descent which constrains the current iterate $\bm{\hat U}_t$ to be orthonormal at every iteration.  Given an initialization $\bm{\hat U}_0$ and stepsize $\eta$, we update $\bm{\hat U}_t$ and $\bm{\hat R}_t$ via
\begin{align*}
\bm{\hat R}_{t}^{(l)} &:= \bm{\hat U}_{t}\t \bm{A}^{(l)} \bm{\hat U}_t;\qquad l = 1 ,\dots, L; \\
    \bm{\hat U}_{t+.5} &= \bm{\hat U}_t - \eta \nabla_{\bm{\hat U}_t}    h\big( \bm{\hat U}_t, \{ \bm{\hat R}^{(l)}_t\}_{l=1}^{L} \big); \\
    \bm{\hat U}_{t+1} &= {\sf SVD}_r( \bm{\hat U}_{t+.5} ).
\end{align*}
The gradient of $h$ keeping $\bm{\hat R}_t^{(l)}$ fixed is given by 
\begin{align}
    \nabla_{\bm{\hat U}_t} h(\bm{\hat U}_t,\{\bm{\hat R}^{(l)}_t\} ) &= \frac{1}{L} \sum_{l=1}^{L} \big( \bm{\hat U}_t \bm{\hat R}_t^{(l)} \bm{\hat U}_t\t - \bm{A}^{(l)} \big) \bm{\hat U}_t \bm{\hat R}_t^{(l)}.
\end{align}
For a given orthonormal $\bm{\hat U}_t$, the update $\bm{\hat R}^{(l)}_t$ is the closed-form minimizer of the objective function.  Thus, the algorithm can be viewed as updating $\bm{\hat U}_t$ via projected gradient descent, and then directly minimizing the error for each $\bm{\hat R}_t^{(l)}$. Alternatively, this algorithm can be viewed as a Riemannian gradient descent algorithm on $\bm{\hat U}$ directly, where the SVD step acts as the retraction.

\subsection{Warm Initialization via Spectral Sum of Squares}
In order to initialize the algorithm, we propose using the sum of squares.  Define
\begin{align*}
    \bm{\hat U}_{0} := {\sf SVD}_r\big( \sum_{l=1}^{L}  (\bm{A}^{(l)})^2 \big).
\end{align*}
To understand this initialization, we note that squaring is in general necessary when the matrices $\bm{R}^{(l)}$ are permitted to have negative eigenvalues.  For example, in the simple case that
\begin{align*}
    \bm{R}^{(1)} &= \begin{pmatrix} \lambda & 0 \\ 0 & \lambda \end{pmatrix}; \qquad \bm{R}^{(2)} = \begin{pmatrix}
        -\lambda & 0 \\ 0 & - \lambda
    \end{pmatrix},
\end{align*}
it is straightforward to observe that simply averaging $\bm{A}^{(l)}$ will result in a ``canceling out'' effect, resulting in pure noise, which is uninformative. More generally, if one has
\begin{align*}
    \bm{R}^{(l)} &= \begin{pmatrix}
        a^{(l)} & b^{(l)} & b^{(l)} \\ b^{(l)} & a^{(l)} & b^{(l)} \\ b^{(l)} & b^{(l)} & a^{(l)} 
    \end{pmatrix}
\end{align*}
with some $a^{(l)} > b^{(l)}$ and some $a^{(l)} < b^{(l)}$, then the average will exhibit a similar ``canceling out'' effect.  Similar examples have appeared in previous works; notably \citet{lei_bias-adjusted_2023} explicitly consider such settings in the context of multilayer stochastic blockmodels, which is a special case of the common subspace model.  

When using the sum of squares, we have
\begin{align*}
    \mathbb{E}\bigg( \sum_l \big(\bm{A}^{(l)} \big)^2 \bigg)&= \sum_l \big( \bm{S}^{(l)} \big)^2 + \sigma^2 L \frac{n+1}{2} \bm{I}_n,
\end{align*}
which has the property that its leading eigenspace is equal to $\bm{U}$ (up to orthogonal transformation), since the addition of $\sigma^2L \frac{n+1}{2} \bm{I}_n$ does not affect the eigenspace.  The full algorithm is summarized in \cref{alg}.

	\begin{algorithm} [t]
 		\caption{Projected Gradient Descent Initialized via Spectral Sum of Squares}
 		\begin{algorithmic} 
 			\Input Collection of matrices $\{\bm{A}^{(l)} \}_{l=1}^{L}$; rank $r$, stepsize $\eta$, number of iterations $T$.
 			\begin{enumerate}
 				\item Define
                \begin{align}
                    \bm{\hat U}_0 := {\sf SVD}_r\bigg( \sum_{l=1}^{L} (\bm{A}^{(l)})^2 \bigg).
                \end{align}
                \item While $t < T$:
                \begin{enumerate}
                    \item Set $\bm{\hat R}^{(l)}_t := \bm{\hat U}_t\t \bm{A}^{(l)} \bm{\hat U}_t$.
                    \item Gradient step:
                    \begin{align}
                        \bm{\hat U}_{t+.5} := \bm{\hat U}_t - \frac{\eta}{L} \sum_{l=1}^{L} \big( \bm{\hat U}_{t} \bm{\hat R}^{(l)}_t \bm{\hat U}_{t}\t - \bm{A}^{(l)} \big) \bm{\hat U}_{t} \bm{\hat R}_t^{(l)}.
                    \end{align}
                    \item Projection step:
                    \begin{align}
                        \bm{\hat U}_{t+1} := {\sf SVD}_r(\bm{\hat U}_{t+.5}).
                    \end{align}
                \end{enumerate}
                                
                        \end{enumerate}
 			\Output $\bm{\hat U}_T$.
 		\end{algorithmic}
 		\label{alg}
 	\end{algorithm}



\section{Estimation Upper and Lower Bounds} \label{sec:estimation}
In this section we consider the statistical and computational limits for estimation. 
We first have the following result.

\begin{theorem}[Upper Bounds for Iterates] \label{thm:mainthm}
Suppose that there exists some sufficiently large constant $C_0$ such that
    \begin{align}
          \lambda/\sigma \geq  C_0 \max \bigg\{  \kappa \sqrt{r} \sqrt{\frac{nr}{L}}, \kappa^3  \sqrt{\frac{nr}{L}}, \kappa r^{1/4} \frac{\sqrt{n}}{L^{1/4}}, r \frac{\sqrt{n}}{L^{1/4}}\bigg\}. \label{maincondition}
    \end{align}
    Suppose each matrix $\bm{N}^{(l)}$ is a subgaussian Wigner matrix with variance $\sigma^2$, and suppose also that $\max_{i,j,l} \| \bm{N}^{(l)}_{ij} \|_{\psi_2} \lesssim \sigma$ and that $r \leq \min\{L,n\}$ with $\log(L) \leq c n r$ for some sufficiently small constant $c$, and suppose that $L \leq n^2$.  Let $\bm{\hat U}_t$ denote the output of \cref{alg} after $t$ iterations, and suppose the stepsize $\eta$ satisfies $\eta \in (\frac{c_{\eta}'}{\lambda_{\max}^2}, \frac{c_{\eta}}{\lambda_{\max}^2})$ where $c_{\eta}' < c_{\eta} < 1$ are any fixed constants.  Then with probability at least $1 -  \exp( - c n)$, for all iterations $t$ it holds that
    \begin{align*}
         \|\sin\Theta(\bm{\hat U}_{t},\bm U)\|_F
    \leq
    C_1\frac{\sigma\kappa\sqrt{nr}}{\lambda\sqrt L}
    +
    \left(1-\frac{c_\eta'}{8\kappa^2}\right)^t
    C_2\frac{\sigma^2 n\sqrt r}{\sqrt L\lambda^2}.
    \end{align*}
    Consequently, after $O \bigg( \log( \frac{\lambda/\sigma \sqrt{L}}{C_1 \kappa\sqrt{nr}}) \bigg)$ iterations, it holds that
    \begin{align*}
                \| \sin\Theta(\bm{\hat U}_t,\bm{U}) \|_F &\leq \frac{2C_1 \sigma \kappa \sqrt{nr}}{\lambda \sqrt{L}}.
    \end{align*}
\end{theorem}
\begin{proof}
    See \cref{sec:mainproof}.
\end{proof}

The key assumption in \cref{thm:mainthm} is \eqref{maincondition}. When $\kappa, r \asymp 1$, its dominant requirement is that $\lambda/\sigma \gtrsim \sqrt{n}/L^{1/4}$. In \cref{sec:minimaxlowerbound}, we argue that this condition is unavoidable for all polynomial-time estimators of $\bm{U}$.
A particularly interesting regime is when $\lambda/\sigma \ll \sqrt{n}$ but $\lambda/\sigma \gtrsim \sqrt{n}/L^{1/4}$, since $\lambda/\sigma \ll \sqrt{n}$ implies that individual subspace estimation cannot be consistent \citep{cai_rate-optimal_2018}. In this setting, the sum of squares spectral initialization (i.e., $t=0$) incurs a “quadratic” error of order $\tfrac{\sigma^2 n}{\lambda^2 \sqrt{L}}$. This behavior mirrors subspace perturbation bounds for low-rank rectangular matrices of dimension $p_1 \times p_2$ with $p_2 \gg p_1$, where a quadratic term dominates in the high-noise regime \citep{cai_rate-optimal_2018,cai_subspace_2021}.
This analogy is reinforced by the observation that $\bm{\hat U}_0$ can be viewed as the matrix of left singular vectors of $[\bm{A}^{(1)},\ldots,\bm{A}^{(L)}]$, a rectangular matrix with $p_1=n$ and $p_2=Ln$. However, unlike these other settings, the additional structure of the common subspace model allows projected gradient descent to reduce the effect of this quadratic error.

Several recent works have studied variants of the sum of squares spectral approach \citep{lei_bias-adjusted_2023,xie_bias-corrected_2024,zhou_deflated_2025,zhang_heteroskedastic_2022,agterberg_entrywise_2022}. These papers focus on the heteroskedastic setting, where debiasing is necessary for improved estimation. In contrast, our model assumes homoskedastic noise, so no additional debiasing is required. Extending our analysis to heteroskedastic noise would be of both practical and theoretical interest, and we leave this direction for future work.

\subsection{Statistical and Computational Lower Bounds} \label{sec:minimaxlowerbound}
The purpose of this section is to demonstrate that the error and assumptions in \cref{thm:mainthm} are optimal.  First, we provide a minimax lower bound for estimation with respect to the $\sin\Theta$ loss.  Define the parameter space
\begin{align}
    \mathcal{P}(\lambda, n,L,\sigma) := \bigg\{ \{\bm{S}^{(l)}\}_{l=1}^{L}: \sl = \bm{U} \rl \U\t; \U \in \mathbb{R}^{n\times r}, \U\t\U = \bm{I}_r, \lambda^2 \leq  \lambda_{\min} \bigg( \frac{1}{L} \sum_{l} (\rl)^2 \bigg) \bigg\}. \label{parameterspace}
\end{align}
We have the following minimax lower bound.

\begin{theorem}[Minimax Lower Bound] \label{thm:minimax}
Assume that $r \leq c_0 \sqrt{n}$ for some sufficiently small constant $c_0$.  Then it holds that 
    \begin{align*}
        \inf_{\bm{\hat U}} \sup_{\mathcal{P}(\lambda,n,L,\sigma)} \mathbb{E} \|\sin\Theta(\bm{\hat U}, \bm{U}) \|_F^2\gtrsim \frac{\sigma^2 nr}{\lambda^2 L} \wedge r, 
    \end{align*}
    where the infimum is taken over all estimators $\bm{\hat U}$ on the basis of the observation $\{\bm{A}^{(l)}\}_{l=1}^{L}$ under {\sf GOE} noise.
\end{theorem}

\begin{proof}
    See \cref{sec:minimaxlowerboundproof}.  
\end{proof}
 \cref{thm:minimax} shows that the error bound attained by \cref{alg} in \cref{thm:mainthm} is minimax rate-optimal after logarithmically many iterations.  However, \cref{thm:minimax} only proves that a necessary condition for consistency is that $\lambda/\sigma \gg \sqrt{\frac{nr}{L}}$.  In contrast, \cref{thm:mainthm} requires an additional condition, which, when $r, \kappa \asymp 1$ is that $\lambda/\sigma \gtrsim \sqrt{n}/L^{1/4}$.  We will provide evidence that this assumption is necessary for a polynomial-time estimator to exist using the low-degree likelihood ratio approach, which has provided evidence for a number of problems related to low-rank matrix estimation \citep{lyu_optimal_2023,lei_computational_2024,luo_computational_2024}.   See \citet{kunisky_notes_2022} for a survey.

Explicitly, suppose one is given $L$ i.i.d. observations $X_1,\dots, X_L$, and consider testing the null hypothesis $H_0: X_i \sim \mathbb{Q}_n$ versus the alternative $X_i \sim \mathbb{P}_n$. As is well-known, the likelihood ratio $L_n(\bm{X}) = \frac{d\mathbb{P}_n(\bm{X})}{d \mathbb{Q}_n(\bm{X})}$ is uniformly most powerful, and, furthermore, when $\mathbb{P}_n$ and $\mathbb{Q}_n$ are statistically indistinguishable (in the sense that the sum of Type I and Type II errors do not tend to zero) it holds that $\|L_n(\bm{X})\|^2 = O(1)$ as $n \to \infty$, where $\|\cdot\|^2 = \mathbb{E}_{\mathbb{Q}_n} ( \cdot )^2$.  
Building on this idea, in \citet{kunisky_notes_2022} the authors propose studying computational distinguishability by projecting $L_n(\bm{X})$ onto polynomials of degree at most $D$.  Define $L^{\leq D}_n(\bm{X})$ as the orthogonal projection of $L_n(\bm{X})$ onto polynomials of degree at most $D$.  \citet{kunisky_notes_2022} conjecture that when $\| L^{\leq D}_n(\bm{X}) \|^2 = O(1)$, then the hypothesis test is \emph{computationally} hard.  

To apply the conjecture we consider the following two hypotheses:
\begin{align*}
H_0:& \ \bm{A}^{(l)} \sim  \bm{N}^{(l)} \text{ for all $l$}; \numberthis \label{null} \\  
\begin{split}
     H_a :& \ \bm{A}^{(l)} \sim \bm{S}^{(l)} + \bm{N}^{(l)} \text{ for all $l$, where } \\
&\bm{S}^{(l)} = \eps_l \lambda \bm{u u}\t; \\
&\eps_l \sim  \mathrm{IID \ Rademacher}.
\end{split} \numberthis \label{alt}
\end{align*}
For simplicity we assume that $\bm{u}$ has entries in $\pm n^{-1/2}$ uniformly at random.  This simple model is a rank-one version of the model we consider herein, where $\bm{u}$ represents the shared subspace.  The hypothesis test above tests pure noise against the alternative that there is a shared subspace $\bm{u}$. Our conjecture for this test is based on a variant from \citet{lyu_optimal_2023}.  
\begin{conjecture}[Conjecture 1 of \citet{lyu_optimal_2023}] \label{conj:lowdegree}
Let $\mathbb{Q}_n$ denote the null hypothesis \eqref{null} and $\mathbb{P}_n$ denote the alternative \eqref{alt}.  
Denote $\bm{X} = \{ \bm{A}^{(l)}\}_{l=1}^{L}$.  If there exists $\eps \geq 0$ and $D \geq \log(nL)^{1+\eps}$ such that $\| L^{\leq D}_n(\bm{X}) \|^2 = 1 + o(1)$ then there is no polynomial-time test $\phi_n: \mathbb{R}^{L \times n \times n} \mapsto \{0,1\}$ such that the sum of Type I error and Type II error probabilities satisfy 
    \begin{align*}
        \mathbb{E}_{\mathbb{Q}_n}[ \phi_n(\bm{X})] + \mathbb{E}_{\mathbb{P}_n}[ 1 - \phi_n(\bm{X})] \to 0
    \end{align*}
    as $n \to \infty$.
\end{conjecture}

 Therefore, if $\|L_n^{\leq D}\|^2 = 1 +o(1)$ for this model, \cref{conj:lowdegree} suggests that there is no polynomial-time algorithm that can detect the existence of a rank one shared subspace.  We have the following result. 


\begin{theorem}[Computational Lower Bound] \label{thm:computational}
When $\lambda/\sigma = o\big( \frac{\sqrt{n}}{L^{1/4}}\big)$ it holds that $\| L_n^{\leq D} \|^2 = 1 + o(1)$ as $n \to \infty$.     
\end{theorem}

\begin{proof}
    See \cref{sec:compgapproof}.
\end{proof}

Combining \cref{thm:minimax} and \cref{thm:computational}, when $r,\kappa \asymp 1$, our proposed estimator attains minimax-optimal estimation error after logarithmically many iterations in the minimal regime such that a polynomial-time estimator exists.  In particular, these results complete the statement of \cref{thm:informal} for the different Estimation SNR regimes.



\section{Minimax-Optimal and Adaptive Inference} \label{sec:inference}
In this section we investigate the information-theoretical limits for statistical inference for the error $\|\sin\Theta(\bm{\hat U}, \bm{U})\|_F^2$, where we denote $\bm{\hat U} = \uhat_t$ for $t$ iterations. Define the matrix $\bm{\mathcal{R}}$ via
\begin{align*}
    \bm{\mathcal{R}} := \big[ \bm{R}^{(1)}, \cdots, \bm{R}^{(L)} \big] \in \mathbb{R}^{r \times Lr}.
\end{align*}
Our first result in this direction  is a novel limit theorem in the Strong Inference SNR Regime.  
\begin{theorem}[Asymptotic Normality] \label{thm:normality}
Suppose that the conditions of \cref{thm:mainthm} hold, and suppose $t \geq  C_{\eta} \log \big( \frac{\kappa^2 \lambda^2 L}{\sigma^2 nr } \big)$, where $C_{\eta}$ depends on the step size $\eta$.  Suppose further that $\lambda/\sigma \gg  \frac{n r^{2} \kappa^4 }{\sqrt{L}}$ and $L \lesssim n$ with $L \to \infty$. Suppose also that $\log(r) \lesssim \sqrt{L}$.
    Assume that 
    \begin{align}
        \max\bigg\{ \frac{\kappa^4}{L}, \kappa^4\|\U\|_{2,\infty}^2, \kappa^2 \sqrt{nr} \| \U\|_{2,\infty}^2, \frac{\kappa^2 r^{3/2} \log(n)}{\sqrt{n}} \bigg\} \ll  1. \label{variancelowerbound}
    \end{align}
Then it holds that
    \begin{align*}
        \frac{\| \sin\Theta(\bm{\hat U},\bm{U}) \|_F^2 - \frac{\sigma^2 n}{2} {\sf Tr}( \R) }{ \sigma^2 \sqrt{n/2} \|\R \|_F } \xrightarrow{d} \mathcal{N}(0,1).
    \end{align*}
\end{theorem}
\begin{proof}
    See \cref{sec:limitheorem}. 
\end{proof}
We focus on the asymptotic regime where $L$ grows with $n$, but still satisfies $L \lesssim n$. When $\kappa,r\asymp 1$, the condition $\|\bm{U}\|_{2,\infty}^2 \ll 1/\sqrt{nr}$ is quite mild, since $\|\bm{U}\|_{2,\infty}^2$ can be as small as $\frac{r}{n}$. Therefore, the main restriction in \cref{thm:normality} is the SNR condition $\lambda/\sigma \gg \tfrac{n r^2 \kappa^4}{\sqrt{L}},$
which is much stronger than the condition $\lambda/\sigma \gtrsim \sqrt{n}/L^{1/4}$ from \cref{thm:mainthm} when $\kappa,r\asymp 1$ and $L\lesssim n$. Though this assumption may seem conservative, we will show in \cref{sec:adaptivity} that it is actually unavoidable for adaptive inference.

\cref{thm:normality} is closely related to several asymptotic normality results in the literature \citep{xia_inference_2022,xia_normal_2021,bao_singular_2021}. In particular, \citet{xia_normal_2021} proves asymptotic normality for the $\sin\Theta$ distance in the matrix denoising model under Gaussian noise. Their theorem has a similar structure, except that the centering term depends not only on the singular values of the signal matrix, but also on the matrix dimensions $d_1$ and $d_2$. This additional dependence comes from the dimension mismatch between $d_1$ and $d_2$: when $d_1=d_2$, the centering term depends only on the singular values. In our setting, the singular values of $\bm{\mathcal{RR}}^\top$ play an analogous role. A related result can also be found in \citet{bao_singular_2021} in the regime where the singular values are of order $\sqrt{n}$ (after appropriate rescaling). Unlike \citet{xia_normal_2021} and \citet{bao_singular_2021}, we additionally provide estimators for the centering and scaling terms. Finally, \citet{xia_inference_2022} establishes a qualitatively similar result for tensor data, though our setting and proofs are different.


\begin{remark}[Highlights of the Proof]
The proof of this result follows from \cref{thm:asymptotics}, which establishes the first-order asymptotic expansion for $\|\pert\uhat_t\|_F^2$.  
The major first step of the proof is to show that 
\begin{align}
    \bm{U}_{\perp}\t \bm{\hat U}_{t} &\approx \sum_l  \bm{U}_{\perp}\t \bm{N}^{(l)} \bm{U} \bm{R}^{(l)} \R,
\end{align}
where the approximation holds up to small-order terms at least quadratic in the noise $\nl$, together with a deterministic optimization error. To prove this approximation, we rely on decoupling arguments to control the quadratic error terms, as well as a novel argument showing that the loss function is geodesically convex (with respect to the Grassmannian manifold geometry) once all iterates are within the statistical neighborhood guaranteed by \cref{thm:mainthm}.  After establishing this result, we then expand out the leading-order term to yield
\begin{align*}
    \| \pert \uhat_t \|_F^2 &= {\sf Tr}\bigg( \bm{U}_{\perp} \bm{U}_{\perp}\t \bigg[ \sum_l \bm{N}^{(l)} \bm{U} \bm{R}^{(l)} \R \bigg] \bigg[ \sum_l \R  \bm{R}^{(l)}\bm{U} \bm{N}^{(l)} \bigg] \bigg) + {\sf Res},
\end{align*}
where ${\sf Res}$ contains all the smaller-order terms and their interactions with the leading-order term.  To handle this final residual term we separate it into terms of order three in $\nl$ and terms of order four and higher together with the deterministic optimization error. To handle the third-order terms, we appeal to decoupling concentration arguments for third-order polynomials of random variables together with $\eps$-net arguments to handle the dependence.  With this residual term controlled, our final result follows by showing that this leading-order term  can be written as a martingale with respect to an appropriate filtration and applying a version of the Martingale Central Limit Theorem.
\end{remark}

\subsection{Rate Optimal Adaptive Confidence  Intervals} \label{sec:confidenceintervals}
\cref{thm:normality} suggests that it is possible to derive confidence intervals via plug-in estimation.  To determine whether these confidence intervals are minimax rate optimal, we  first provide a minimax lower bound for the expected length of all $1-\alpha$ honest confidence intervals.  We will follow a similar framework as in \citet{cai_confidence_2017}.  To lay the groundwork,  we will consider the parameter space $\mathcal{P}(\lambda,n,L,\sigma)$ defined in \eqref{parameterspace}.  Since we are interested in the dependence on the SNR, we will use the shorthand $\mathcal{P}(\lambda) =\mathcal{P}(\lambda,n,L,\sigma)$, since changing $\sigma$ is equivalent to changing $\lambda$ by a rescaling argument.  
Observe that $\mathcal{P}(\lambda_{1}) \subset \mathcal{P}(\lambda_2)$ for any $\lambda_{1} \geq \lambda_2$.  Thus, $\mathcal{P}(\lambda)$ forms a nested set of parameter spaces.

Given a parameter $0 < \alpha <1$, and an estimate $\bm{\hat U}$ of $\bm{U}$ and a parameter space $\mathcal{P}$, let $\mathcal{I}_{\alpha}( \mathcal{P})$ denote the set of all  honest $1 - \alpha$ level confidence intervals for $\|\sin\Theta(\bm{\hat U},\bm{U})\|_F^2$ over $ \mathcal{P}$; that is,
\begin{align*}
    \mathcal{I}_{\alpha}( \mathcal{P}) :&= \Bigg\{ {\sf CI}_{\alpha}\big(\{ \bm{A}^{(l)}\}_{l=1}^{L} \big) = [l, u]:  \inf_{\theta \in  \mathcal{P}} \p_{\theta} \bigg( \| \sin\Theta(\bm{\hat U},\bm{U}) \|_F^2 \in {\sf CI}_{\alpha}\big(\{ \bm{A}^{(l)}\}_{l=1}^{L} \big)\bigg)  \geq 1 - \alpha \Bigg\}.
\end{align*}
For simplicity we suppress the dependence of ${\sf CI}_{\alpha}$ on the observations $\{\bm{A}^{(l)}\}_{l=1}^{L}$.  
The length of the confidence interval ${\sf L}({\sf CI}_{\alpha})$ is defined as the difference $u - l$.  The maximum expected length over the parameter space $\mathcal{P}$ is denoted as
\begin{align*}
    \mathcal{L}_{{\sf CI}_{\alpha}}(  \mathcal{P} ):= \sup_{\theta\in \mathcal{P}} \mathbb{E}_{\theta} {\sf L}\big({\sf CI}_{\alpha} \big).  
\end{align*}
Our first result gives a lower bound on the minimax expected length over $\mathcal{P}(\lambda)$.
\begin{theorem} \label{thm:ci_minimaxlowerbound}
Suppose that $\lambda/\sigma \geq C_0  \frac{nr}{\sqrt{L}}$ 
for some sufficiently large constant $C_0 > 0$, and that $r \geq 2$.   
Then there exists a constant $c' > 0$ such that for all $n$ sufficiently large,
\begin{align}
    \inf_{{\sf CI}_{\alpha} \in \mathcal{I}_{\alpha}(\mathcal{P}(\lambda))}   \mathcal{L}_{{\sf CI}_{\alpha}}\big( \mathcal{P}(\lambda) \big) &\geq c'  \frac{\sigma^2 nr}{\lambda^2 L} \frac{1}{\sqrt{nr}}.\label{cilb}
\end{align}
\end{theorem}
\begin{proof}
    See \cref{sec:ci_minimaxlowerboundproof}.
\end{proof}
When $\lambda/\sigma \geq C_0 \frac{nr}{\sqrt{L}}$, we are in the regime where \cref{thm:normality} applies up to factors of $r$ and $\kappa$.  It is also straightforward to show that if $\R$ is known, then the resulting confidence intervals will have the width given above.  Thus, if it is possible to estimate the centering and scaling terms to appropriate fidelity, the resulting confidence intervals will have minimax optimal length.  


We now discuss the construction of the optimal confidence interval that achieves the rate in \cref{thm:ci_minimaxlowerbound}. Define
\begin{align*}
    \bm{\mathcal{\hat R}} := \begin{pmatrix} \bm{\hat U}\t \bm{A}^{(1)} \bm{\hat U} , \cdots , \uhat\t \bm{A}^{(L)} \uhat \end{pmatrix} \in \mathbb{R}^{r \times rL}.
\end{align*}
Let $z_{\alpha/2}$ denote the $1 - \alpha/2$ quantile of the standard Gaussian distribution.  
We define
\begin{align*}
    \hat{{\sf CI}_{\alpha}} :=\frac{  \sigma^2n}{2}{\sf Tr}\big[ \rhatinv \big] \pm z_{\alpha/2} \sigma^2 \sqrt{n/2} \| \rhatinv \|_F.
\end{align*}
This quantity can be understood as simple plug-in estimation of the centering and scaling terms in \cref{thm:normality}.  Nonetheless, the following result shows that this confidence interval is optimal when $\bm{\hat U}$ is computed via \cref{alg}.  

\begin{theorem}[Confidence Interval Upper Bound] \label{thm:minimaxoptimalconfidenceinterval}
Suppose the conditions of \cref{thm:normality} hold.  
Then it holds that 
    \begin{align*}
       \liminf_{n\to\infty} \inf_{\theta \in \mathcal{P}(\lambda)} \mathbb{P} \bigg( \| \sin\Theta(\uhat,\U) \|_F^2 \in \hat{{\sf CI}_{\alpha}} \bigg) \geq 1 - \alpha,
    \end{align*}
    and for all $n$ sufficiently large, with probability at least $1 - \exp(- c n)$,
    \begin{align*}
      {\sf L} (\hat{{\sf CI}_{\alpha}}) \lesssim \frac{\sigma^2  nr}{\lambda^2 L}  \frac{1}{\sqrt{nr}}.
    \end{align*}
\end{theorem}
\begin{proof}
    See \cref{sec:confidenceintervalproof}.
\end{proof}
Note that \cref{thm:minimaxoptimalconfidenceinterval} does not require sample splitting to achieve optimal inference, unlike other similar results in the literature on linear regression \citep{cai_accuracy_2018,cai_confidence_2017} and matrix completion \citep{carpentier_adaptive_2018}.  Furthermore, \cref{thm:minimaxoptimalconfidenceinterval} demonstrates that when $\lambda/\sigma \gg \tfrac{n}{\sqrt{L}}$ and $r,\kappa = O(1)$, our proposed confidence intervals are minimax-optimal.  Since the construction of $\hat{ {\sf CI}_{\alpha}}$ does not rely on knowledge of $\lambda$ \emph{a priori}, the resulting confidence intervals are  adaptive to signal strength.  

\subsection{Adaptivity Lower Bounds} \label{sec:adaptivity}
In this section we determine whether adaptive confidence intervals can be constructed in the regime $\lambda/\sigma \lesssim \frac{n}{\sqrt{L}}$.  
To study adaptivity, we consider a similar framework as in \citet{cai_confidence_2017}. Given two nested parameter spaces $\mathcal{P}_1 \subset \mathcal{P}_2$, we define the adaptation benchmark
\begin{align*}
    \mathcal{L}^*_{\alpha}( \mathcal{P}_1,\mathcal{P}_2) := \inf_{{\sf CI}_{\alpha} \in \mathcal{I}_{\alpha}(\mathcal{P}_2)} \mathcal{L}_{{\sf CI}_{\alpha}}( \mathcal{P}_1 ),
\end{align*}
which is the infimum of the maximum expected length over $\mathcal{P}_1$  among all $1 - \alpha$ confidence intervals over $\mathcal{P}_2$.  Note that \cref{thm:ci_minimaxlowerbound} gives a lower bound on $\mathcal{L}^{*}_{\alpha}(\mathcal{P}(\lambda),\mathcal{P}(\lambda))$, and, moreover, the confidence intervals in \cref{thm:minimaxoptimalconfidenceinterval} attain this rate.  

If a confidence interval is rate optimally adaptive over $\mathcal{P}_1$ and $\mathcal{P}_2$, it should have optimal expected length simultaneously over both parameter spaces while still maintaining coverage over $\mathcal{P}_2$.  Then this confidence interval must satisfy $\mathcal{L}_{{\sf CI}_{\alpha}}(\mathcal{P}_1) \geq  \mathcal{L}^*_{\alpha}(\mathcal{P}_1,\mathcal{P}_2)$.  Hence, if $\mathcal{L}^*_{\alpha}(\mathcal{P}_1,\mathcal{P}_2)$ is significantly larger than the minimax rate over $\mathcal{P}_1$, then rate-optimal adaptation is impossible.

For technical reasons in the proof, we use the following slightly modified parameter space $\mathcal{\tilde P}(\lambda)$ defined via
\begin{align*}
    \mathcal{\tilde P}(\lambda) := \bigg\{ \{\bm{S}^{(l)}\}_{l=1}^{L} : \sl = \U \rl \U\t; \U \in \mathbb{R}^{n\times r}, \U\t\U = \bm{I}_r, r \lambda^2 \leq \frac{1}{L} \sum_{l=1}^{L} \| \rl \|_F^2 \bigg\}.
\end{align*}
If each $\rl$ is well-conditioned with condition number $O(1)$, then $\mathcal{\tilde P}(\lambda)$ coincides with $\mathcal{P}(\lambda)$ up to constants.  The following result shows the (lack of) adaptivity of confidence intervals over the two parameter spaces $\mathcal{\tilde P}(\lambda_1), \mathcal{\tilde P}(\lambda_2)$. 
\begin{theorem} \label{thm:ci_minimax2} Take $\lambda_{1} \geq \lambda_2$, where $\lambda_1/\sigma \geq \frac{nr}{\sqrt{L}}$.  Suppose that $r \geq 2$.  Then there exists some constant $c' > 0$ such that for all $n$ sufficiently large it holds that 
   \begin{align*}
      \mathcal{L}^{*}_{\alpha}\big( \mathcal{\tilde P}(\lambda_1), \mathcal{\tilde P}(\lambda_2) \big) \geq c' \frac{\sigma^2 nr}{\lambda_{1}^2 L} \bigg( \frac{\sigma \sqrt{nr}}{\lambda_2 \sqrt{L}} + \frac{1}{\sqrt{nr}} \bigg).
   \end{align*}
\end{theorem}
\begin{proof}
    See \cref{sec:ci_minimaxlowerboundproof}.
\end{proof}

Observe that when $r,\kappa \asymp 1$, and $\lambda_1 \gg \lambda_2$, with $\lambda_2/\sigma \ll \frac{nr}{\sqrt{L}}$, then rate-optimal adaptation is impossible. 
However, if $\lambda_1,\lambda_2 \gg \sigma \frac{n r}{\sqrt{L}}$, rate-optimal adaptation is possible, and this threshold matches the assumption in  \cref{thm:minimaxoptimalconfidenceinterval} up to the factor of $r$.
Thus, considering \cref{thm:ci_minimaxlowerbound,thm:ci_minimax2,thm:minimaxoptimalconfidenceinterval}, we have now established the existence of the Weak and Strong Inference SNR Regimes, which completes the proof of \cref{thm:informal}.

\section{Related Work} \label{sec:relatedwork}
Our work is closely connected to the literature on multilayer network analysis.  A number of authors have considered estimation in the multilayer stochastic blockmodel with arbitrary connection probabilities, which can be viewed as a special case of the common subspace model \citep{lei_bias-adjusted_2023,lei_consistent_2020,paul_spectral_2020,jing_community_2021,agterberg_joint_2025,chen_global_2022,huang_spectral_2023}.   Our proposed estimation procedure has previously been considered in \citet{paul_spectral_2020} as ``orthogonal-linked matrix factorization,'' though statistical guarantees are only provided for global minima under a multilayer stochastic blockmodel.  Our procedure does not actually require us to find a global minimum, and we show that our algorithm (with our initialization) is both statistically and computationally optimal.

The work \citet{lei_computational_2024} gives a computational lower bound for estimation in the multilayer stochastic blockmodel, but their results are sharp only up to logarithmic factors.  In contrast, in this work we are able to explicitly characterize the computational limits in terms of the SNR. The works   \citet{lyu_optimal_2023,lyu_optimal_2023-1} are also closely related to this work, where they study the statistical and computational limits for a closely related model for mixtures of low-rank matrices.  Our model is significantly more general than theirs, though their model can be viewed as a special case of the model we consider herein.  In addition, none of the aforementioned works consider inference.

Our work is also related to the work \citet{arroyo_inference_2021} which proposes the common subspace model.  They consider estimating $\bm{U}$ by first setting $\bm{\hat U}^{(l)} := {\sf SVD}_r(\bm{A}^{(l)})$ and $\bm{\hat U} = {\sf SVD}_r( [\bm{\hat U}^{(1)}, \bm{\hat U}^{(2)}, \cdots ,\bm{\hat U}^{(L)} ])$.  This procedure has been further studied and extended in several follow-on works \citep{agterberg_joint_2025,zheng_limit_2024,paul_spectral_2020}.  However, a major deficiency of this procedure and its variants is that the first step requires each matrix $\bm{A}^{(l)}$ to have sufficient signal strength.  Under our statistical model, this is equivalent to requiring that $\lambda/\sigma \gg \sqrt{n}$, which does not benefit from the ``aggregation effect'' of having larger $L$. 

In order to leverage the ``aggregation'' from  multiple networks, the work \citet{lei_bias-adjusted_2023} proposed to estimate $\bm{U}$ (under a multilayer stochastic blockmodel) via a bias-adjusted sum of squares.  Their procedure can essentially be viewed as a modified version of our initialization $\bm{\hat U}_0$, where the modification consists of diagonal deletion due to the heteroskedasticity.  
They prove an upper bound under a certain signal-strength condition (for Bernoulli noise), and this signal-strength condition was later shown to be computationally optimal in \citet{lei_computational_2024}. Subsequently, \citet{xie_bias-corrected_2024} proposed to further modify the procedure of \citet{lei_bias-adjusted_2023} to ``impute'' the diagonal.  The algorithm considered therein can be viewed as a Bernoulli-noise specific extension of the ${\sf HeteroPCA}$ algorithm proposed in \citet{zhang_heteroskedastic_2022} and further studied in both \citet{agterberg_entrywise_2022,yan_inference_2024}.  However, in our work the focus is on optimal error rates and statistical inference under homoskedastic noise where bias adjustment is not necessary, and therefore the focus is significantly different.

Our work is also related to the literature on tensor data analysis, which exhibits similar statistical-computational gaps \citep{zhang_tensor_2018,luo_tensor_2022,auddy_tensors_2025,luo_tensor--tensor_2024,arous_landscape_2019}, though  our model is significantly different.  
 Our model is also related to several models considered in the data integration literature (e.g., \citet{lock_joint_2013}).  The works \citet{ma_optimal_2026,yang_estimating_2025,li_heterojive_2025} are perhaps most similar, though only \citet{ma_optimal_2026} consider the ``weak-signal'' regime we consider herein, though their focus is difrerent.  

Turning to statistical inference, there have been a number of works on distributional theory and statistical inference for single networks or matrices \citep{agterberg_statistical_2023,chen_asymmetry_2021,chen_spectral_2021,yan_inference_2024,fan_asymptotic_2022,xie_entrywise_2024,xie_higher-order_2025,cheng_tackling_2021,pu_asymptotic_2026,liu_asymptotic_2025,bao_singular_2021,ding_high_2020}; see \citet{agterberg_overview_2026} for a survey in the stochastic blockmodel setting.  However, inferential techniques for multilayer networks or collections of matrices are comparatively lacking, though there have been some efforts.  
Our work is most closely related to the works \citet{zheng_limit_2024,xie_bias-corrected_2024,agterberg_statistical_nodate,xia_inference_2022}.  
The work \citet{zheng_limit_2024}
considers the model herein and derives CLTs for the \emph{rows} of the estimated subspace as well as the estimated score matrices.  Similarly, in the work \citet{xie_bias-corrected_2024} the author develops entrywise limit theorems for the rows of the matrix $\bm{\hat U}$ obtained via bias-adjusted spectral sum of squares.  Finally, the works \citet{agterberg_statistical_nodate,xia_inference_2022} consider statistical inference for tensors.  Our analyses are different from these prior works, and we primarily focus on the $\sin\Theta$ distance.  Of these, only \citet{xia_inference_2022,xia_normal_2021} consider this distance, but neither of these papers consider the model herein, nor do they undertake a minimax analysis.  

Our study on adaptive inference is inspired by a rich literature on adaptation theory in various settings.  For example, our minimax result in \cref{thm:ci_minimax2} is reminiscent of several results in high-dimensional sparse linear regression \citep{cai_accuracy_2018,cai_confidence_2017,cai_optimal_2021-1,cai_statistical_2023,guo_optimal_2019,nickl_confidence_2013}  and functional data analysis \citep{cai_adaptation_2004,cai_adaptive_2013,cai_adaptive_2005,robins_adaptive_2006,hoffmann_adaptive_2011,cai_optimal_2006,lepskii_asymptotically_1992,lepskii_asymptotically_1993,cai_adaptive_2014}.  In these prior works the authors also consider adaptive confidence sets or confidence intervals, and they consider adaptation to an unknown parameter.  However, the only result we know of considering adaptive inference in a matrix context is \citet{carpentier_adaptive_2018}, but the focus of that work is in adaptation with respect to the rank $r$, and they do not consider estimating subspaces.  To the best of our knowledge, our study is the first to consider adaptive inference, where adaptation is considered with respect to the SNR.  

\section{Numerical Applications} \label{sec:numerical}
In this section we apply \cref{alg} to both real and simulated data.  In the subsequent section we consider simulated data, and in \cref{sec:flightdata} we consider an application to global trade data.  

\subsection{Simulations}
\begin{figure}[ht]
    \centering
    \begin{minipage}{0.5\textwidth}
        \centering
        \includegraphics[width=\textwidth]{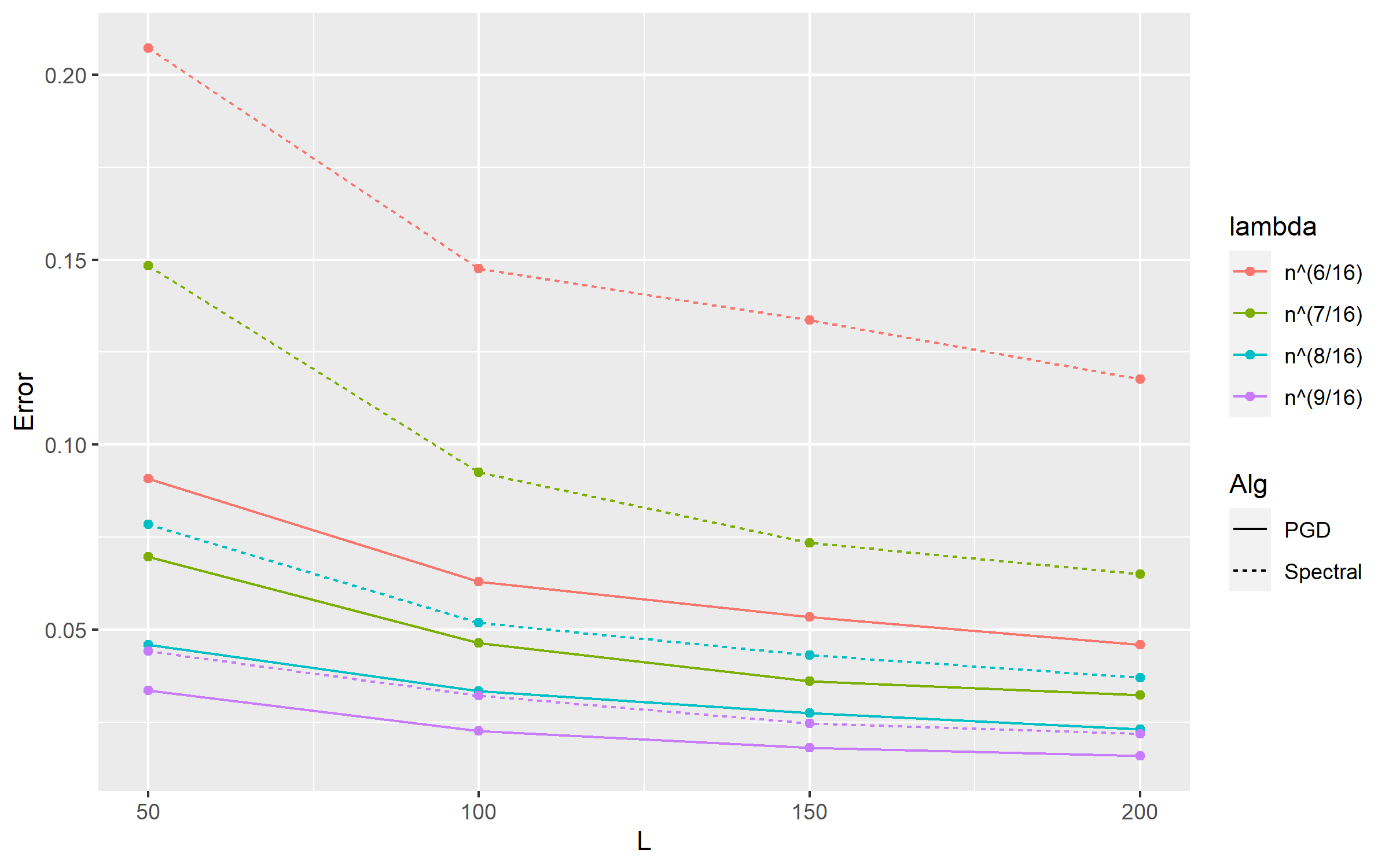} 
        \caption*{$n = 200$}
    \end{minipage}\hfill
    \begin{minipage}{0.5\textwidth}
        \centering
        \includegraphics[width=\textwidth]{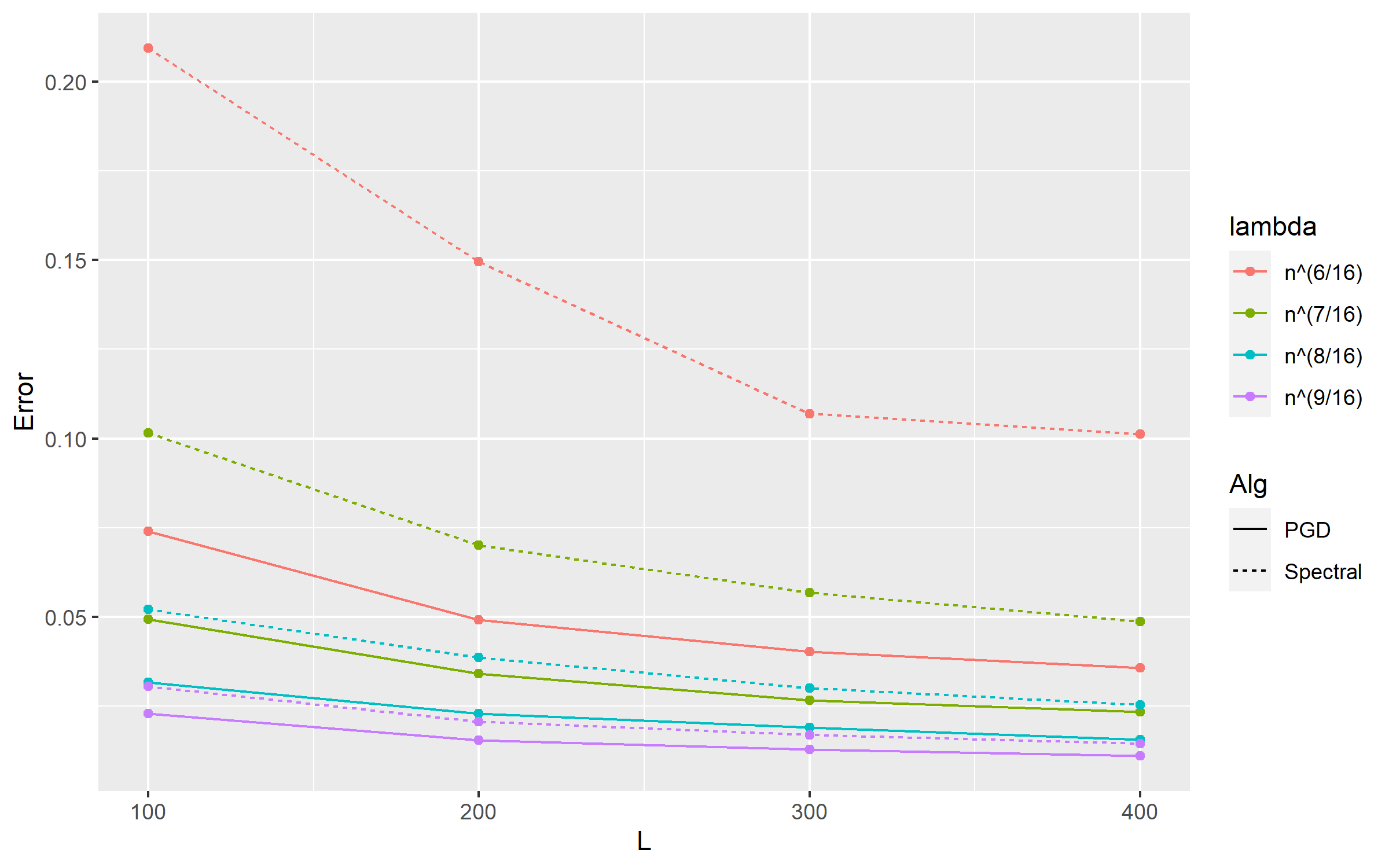} 
        \caption*{$n=400$}
    \end{minipage}
    \caption{Simulation of proposed projected gradient descent algorithm (solid lines) versus the spectral initialization (dotted lines) with $\sigma = 1$. Each color is a different value of $\lambda = n^{c/16}$ with $c \in \{6,7,8,9\}$.  The $x$-axis denotes the number of layers, and the $y$-axis denotes the error averaged over 10 Monte Carlo repetitions.} \label{fig:errorsim}
\end{figure}

We first consider the following setup.  For each figure, we consider $n \in \{200,400\}$. For a given $L,\lambda$, and $n$, we first generate a subspace $\U$ randomly by drawing a Gaussian random matrix of dimension $n \times r$, and then computing its left singular vectors. We then generate $\rl$ by drawing a random symmetric matrix of dimension $r \times r$ with standard Gaussian entries, and then normalizing it so its smallest singular value is $\lambda$.  Each $\bm{A}^{(l)}$ is then drawn with $\sigma = 1$ according to the model \eqref{gosie}.  
In \cref{fig:errorsim} we plot the errors for $L \in \{n/4,n/2,3n/4,n\}$ and $\lambda = n^{c/16}$ for $c \in \{6,7,8,9\}$.  Each different color is associated to a different value of $\lambda$.  The dotted line corresponds to the spectral initialization (akin to the algorithm proposed in prior works), and the solid line denotes the error after 100 iterations of projected gradient descent.  Notably, the error for a given value of $\lambda$ is always smaller for projected gradient descent than the spectral initialization, but for sufficiently large $\lambda$, the two errors are much closer, and get smaller for larger $L$, which mirrors our theory directly.

\begin{table}[ht]
\centering

\begin{subtable}[t]{0.48\textwidth}
\centering
\begin{tabular}{c|cccc}
\hline
 & \multicolumn{4}{c}{$L$} \\
\cline{2-5}
$\lambda$ & 50 & 100 & 150 & 200 \\
\hline
$n^{8/16}$  & 0.44 & 0.68 & 0.71 & 0.71 \\
$n^{9/16}$  & 0.75 & 0.85 & 0.95 & 0.92 \\
$n^{10/16}$ & 0.95 & 0.94 & 0.97 & 0.99 \\
$n^{11/16}$ & 0.93 & 0.90 & 0.93 & 0.97 \\
\hline
\end{tabular}
\caption*{$n = 200$}
\end{subtable}
\hfill
\begin{subtable}[t]{0.48\textwidth}
\centering
\begin{tabular}{c|cccc}
\hline
 & \multicolumn{4}{c}{$L$} \\
\cline{2-5}
$\lambda$ & 100 & 200 & 300 & 400 \\
\hline
$n^{8/16}$  & 0.68 & 0.80 & 0.83 & 0.82 \\
$n^{9/16}$  & 0.89 & 0.90 & 0.93 & 0.94 \\
$n^{10/16}$ & 0.92 & 0.94 & 0.96 & 0.95 \\
$n^{11/16}$ & 0.94 & 0.96 & 0.96 & 0.96 \\
\hline
\end{tabular}
\caption*{$n = 400$}
\end{subtable}

\caption{Empirical probability of $\|\sin\Theta(\bm{\hat U},\U)\|_F^2 \in \hat{{\sf CI}_{\alpha}}$ ($\alpha = .05)$ averaged over 100 Monte Carlo repetitions.}
\label{table:cis}
\end{table}

Next, we consider confidence intervals for the error $\|\sin\Theta(\bm{\hat U}_t,\U)\|_F^2$.  We consider a similar setup as before, only now we consider $\lambda = n^{c/16}$ with $c \in \{8,9,10,11\}$, since \cref{thm:minimaxoptimalconfidenceinterval} requires stronger signal strength.  Table \ref{table:cis} shows the empirical probability that the true $\|\sin\Theta(\bm{\hat U},\U)\|_F^2$ lies in the computed confidence interval with $\alpha = .05$, averaged over 100 samples for each $L$, with the left table associated to $n = 200$ and the right for $n = 400$.  According to our theory, as we increase $\lambda$ and increase $L$, we expect the confidence intervals to be valid.  Indeed, at $\lambda = n^{8/16}$ we do not see validity for either $n$ value, but for $\lambda = n^{9/16}$ we see that the empirical probability is closer to .95 for larger $L$.  For the row associated to $\lambda = n^{10/16}$, we again see that our resulting confidence intervals are approximately valid for all $L$.  This phenomenon is reflected by both $n = 200$ and $n = 400$.

\subsection{Application to Trade Data}

\label{sec:flightdata}

\begin{figure}[ht]
    \centering
    \begin{minipage}{0.45\textwidth}
        \centering
        \includegraphics[width=.8\textwidth]{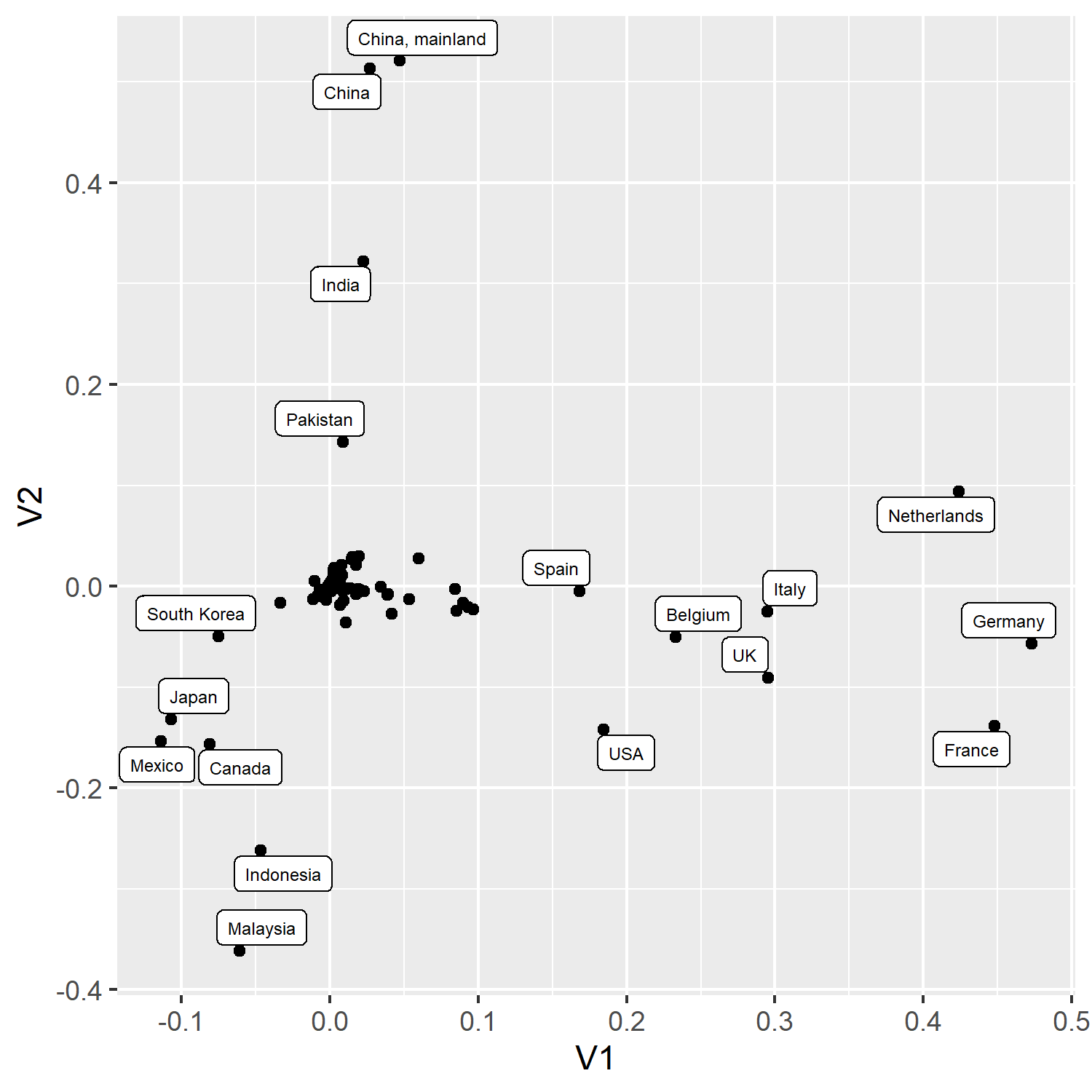} 
        \caption*{Dimensions 1 and 2.}
    \end{minipage}\hfill
    \begin{minipage}{0.45\textwidth}
        \centering
        \includegraphics[width=.8\textwidth]{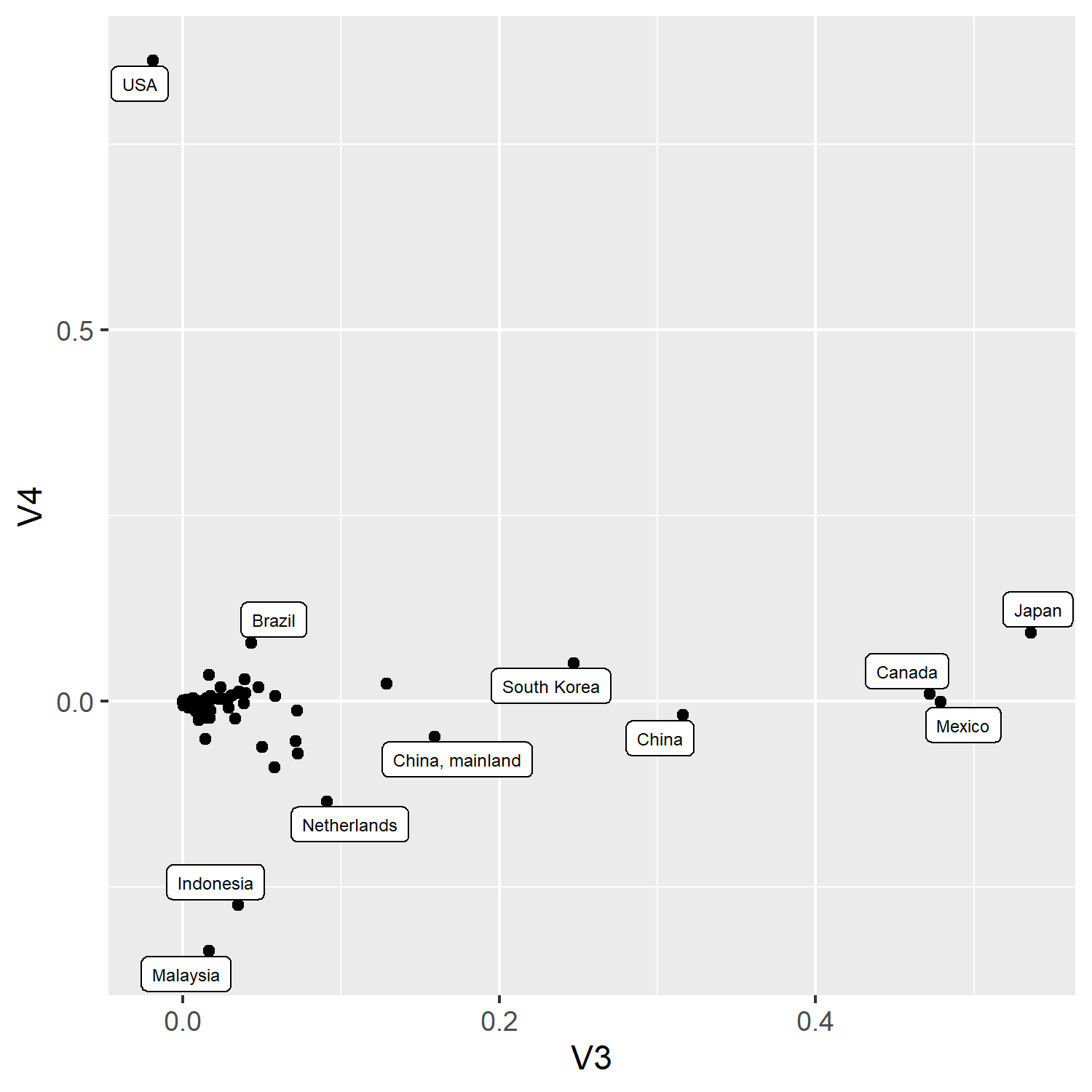} 
        \caption*{Dimensions 3 and 4.}
    \end{minipage}
    \caption{Resulting embedding of trade data} \label{fig:tradedata}
\end{figure}

In this section we apply projected gradient descent with the spectral initialization to the trade data analyzed in several previous works \citep{jing_community_2021,lyu_optimal_2023,agterberg_estimating_2025} and collected by \citet{de_domenico_structural_2015}.  We preprocess the data as in \citet{agterberg_estimating_2025}, resulting in $L = 59$ matrices of dimension $214 \times 214$, where each matrix $\bm{A}^{(l)}$ is a measure of the trading volume of commodity $l$ between different countries.  We run projected gradient descent with $r = 5$ as in \citet{agterberg_estimating_2025}.  

In \cref{fig:tradedata} we plot dimensions 1 and 2 (left) and dimensions 3 and 4 (right) of $\bm{\hat U}_t$.  Observe that dimensions 1 and 2 have a clear distinction between Europe and Asia, and dimensions 3 and 4 have a clear separation of the USA from other countries.  To further study the resulting embeddings, in \cref{fig:five_side_by_side} we plot the percentile of each dimension, with largest values in red and smallest values in blue.  In dimension 2 we see that most of Eastern Europe, Asia, and Africa are warmer colors, whereas Western Europe and North and South America are largely cooler colors.  Dimension 4 contains similar behavior, except with the roles of Europe and the Americas reversed.  This figure suggests that \cref{alg} is uncovering underlying shared structure.

\begin{figure}[h!]
\centering

\begin{subfigure}{0.3\textwidth}
\centering
\includegraphics[width=\linewidth]{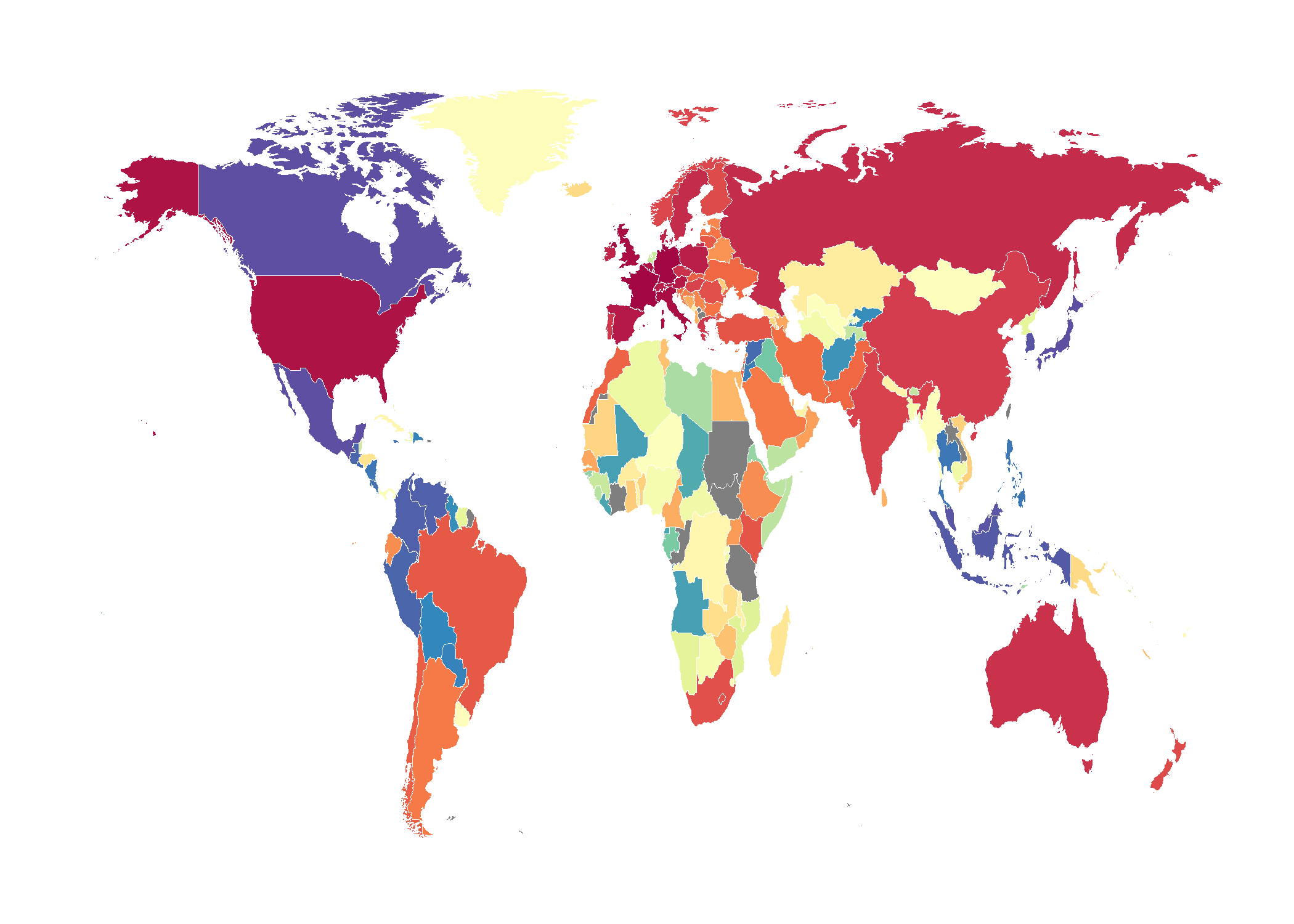}
\caption*{Dimension 1}
\end{subfigure}
\hspace{0.03\textwidth}
\begin{subfigure}{0.3\textwidth}
\centering
\includegraphics[width=\linewidth]{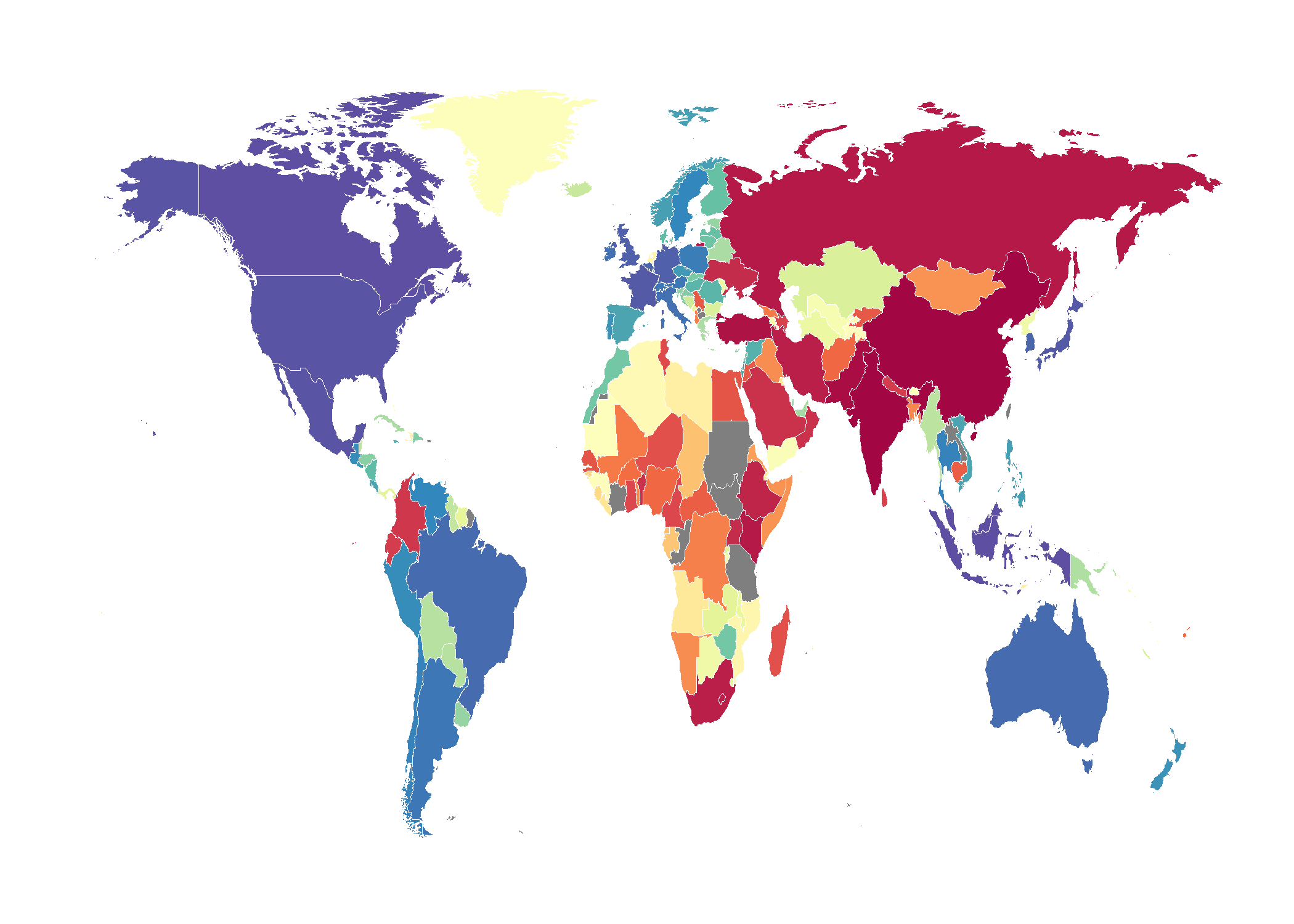}
\caption*{Dimension 2}
\end{subfigure}
\hspace{0.03\textwidth}
\begin{subfigure}{0.3\textwidth}
\centering
\includegraphics[width=\linewidth]{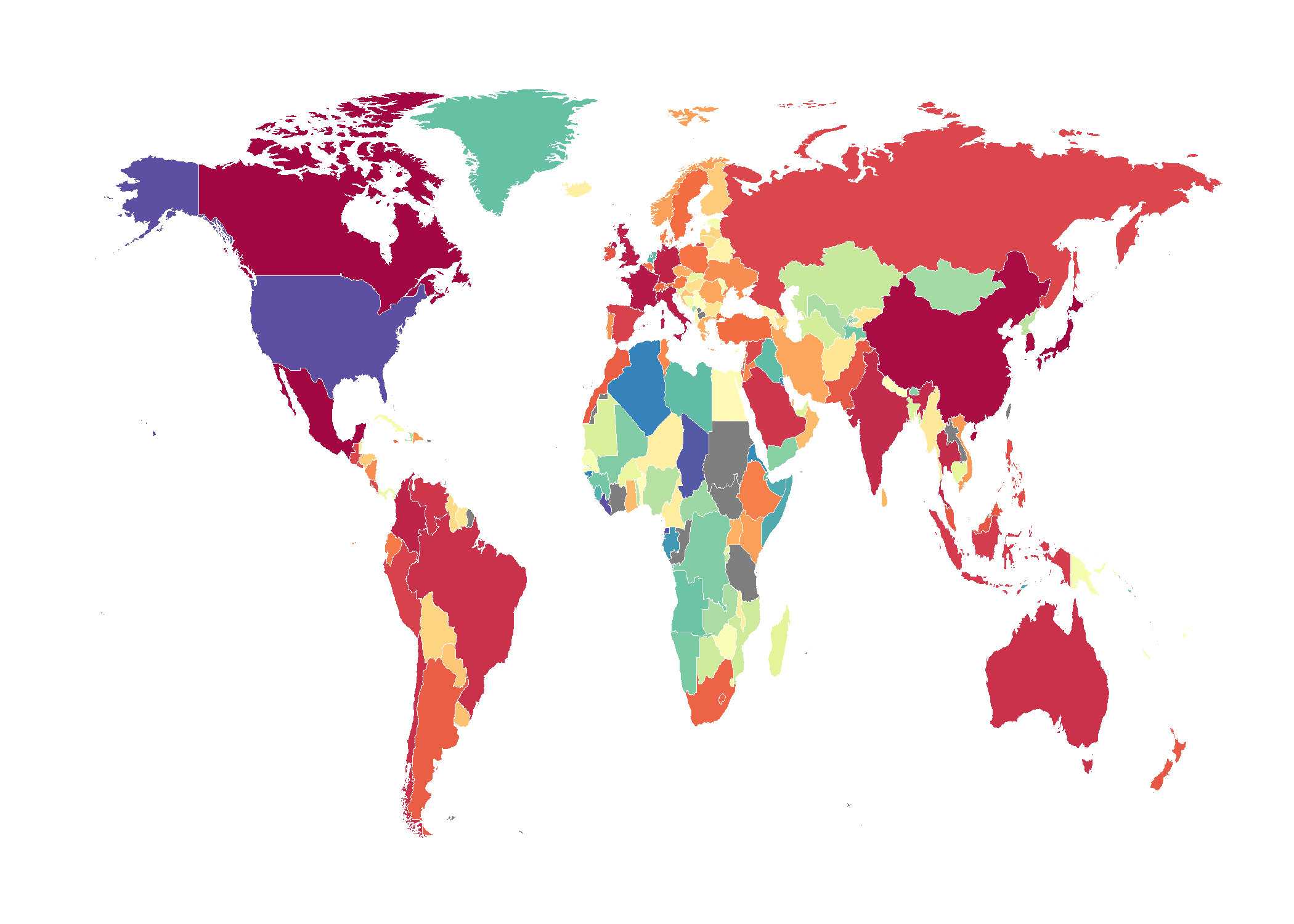}
\caption*{Dimension 3}
\end{subfigure}
\vspace{0.6em}
\begin{subfigure}{0.3\textwidth}
\centering
\includegraphics[width=\linewidth]{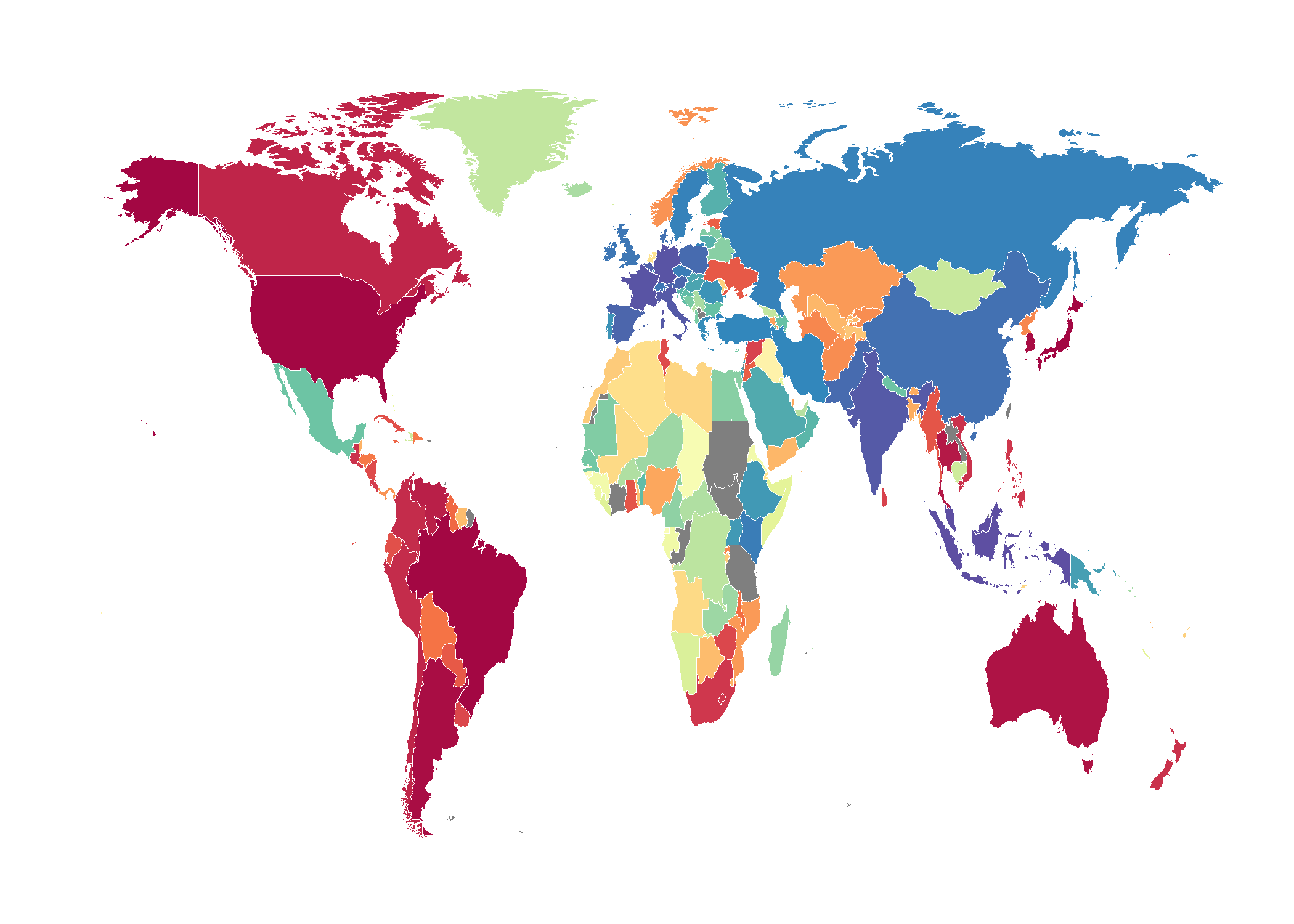}
\caption*{Dimension 4}

\end{subfigure}
\hspace{0.03\textwidth}
\begin{subfigure}{0.3\textwidth}
\centering
\includegraphics[width=\linewidth]{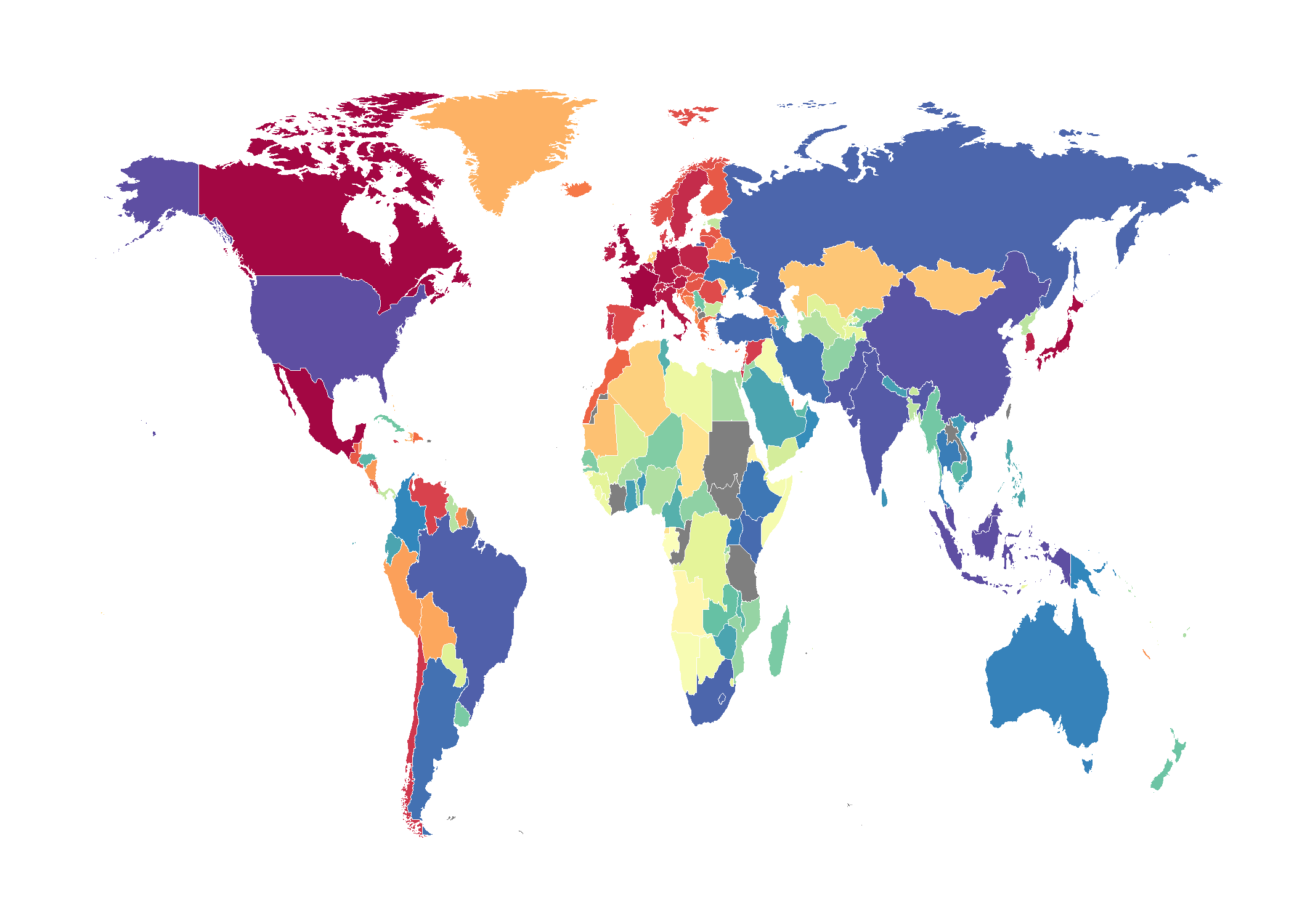}
\caption*{Dimension 5}

\end{subfigure} 

\caption{Percentile of each column of $\bm{\hat U}_t$ for each country, with warmer colors associated to larger values and cooler colors associated to smaller values.}
\label{fig:five_side_by_side}
\end{figure}

\section{Discussion}
We have considered the statistical and computational limits for estimation of $\bm{U}$ with respect to the Frobenius $\sin\Theta$ distance, and the statistical limits for adaptive inference for the error $\|\sin\Theta(\bm{\hat U},\bm{U})\|_F^2$.  In practice, these results imply that statistical inference may be difficult without \emph{a priori} knowledge of the signal strength, even when optimal estimation is feasible.  There are several possible extensions, which we detail below.
\begin{itemize}
    \item \textbf{Optimal dependence on $r$ and $\kappa$.} In this work we have focused on the statistical and computational limits without optimizing for $r$ and $\kappa$, and our results are somewhat loose with regard to these quantities.  It would be of interest to determine the precise dependence on these quantities, and whether there are algorithms that are robust to overspecifying $r$, or that perform well with large values of $\kappa$.
    \item \textbf{Extension to heteroskedastic noise}.  Our main results require, at the very least, that each $\bm{N}^{(l)}$ is a Wigner matrix.  However, in practice, the noise matrices may have heteroskedasticity, as is the case for binary network data.  Extending our analysis to this setting may require additional modifications that explicitly account for heteroskedasticity.
    \item \textbf{Sharp phase transitions.} Our results demonstrate different phase transitions based on the order-wise behavior of $\lambda/\sigma$.  It is of interest to determine the sharp constants for the error in all possible regimes of interest.
\end{itemize}

\section*{Acknowledgments}
Part of this research was performed while the author was visiting the Institute for Mathematical and Statistical Innovation (IMSI), which is supported by the National Science Foundation (Grant No. DMS-1929348). 

\appendix

\section{Proof of \cref{thm:mainthm}} \label{sec:mainproof}
In this section we prove \cref{thm:mainthm}.  Throughout our proofs we note that if an event holds with probability at least $1 - C \exp ( - c n)$ then it holds with probability $1 - \exp( - c n)$ for some other constant $c$, and we do this replacement subsequently without further justification.  

We will prove by induction.  Our first result controls the spectral norm of the initialization error. Throughout our proofs, we let $\lambda_{\max} := \max_l \| \rl \|$.  

\begin{lemma}
    \label{lem:initialconcentration}
    Under the conditions of \cref{thm:mainthm}, with probability at least $1 -  \exp ( - c n) $ it holds that
\begin{align*}
        \bigg\| \sum_{l} \bigg( (\bm{A}^{(l)})^2 - (\bm{S}^{(l)})^2 - \frac{\sigma^2}{2}  (n+1) \bm{I}_n \bigg) \bigg\| \lesssim  \lambda_{\max} \sigma \sqrt{nL} + \sigma^2 n \sqrt{L}.
\end{align*}
\end{lemma}
\begin{proof}
    See \cref{sec:initiallemproof}.
\end{proof}

\begin{proof}[Proof of \cref{thm:mainthm}: Initialization]
    We now provide a bound for the spectral initialization.   We will simply apply the Davis-Kahan Theorem.  Observe that 
\begin{align*}
    \sum_{l=1}^{L} \mathbb{E} (\bm{A}^{(l)})^2 &= \sum_l (\bm{S}^{(l)})^2 + \frac{\sigma^2}{2} L (n+1)  \bm{I}.
\end{align*}
This shows that
\begin{align*}
    \lambda_{r}\bigg( \sum_l (\bm{S}^{(l)})^2 + \frac{\sigma^2}{2} (n+1) L  \bm{I} \bigg) \geq L \lambda^2 + \frac{\sigma^2}{2} n L.
\end{align*}
Observe that by \cref{lem:initialconcentration}, with probability at least $1 -  \exp( - c n)$,
\begin{align*}
    \bigg\| \sum_l \bigg( \big( \bm{A}^{(l)} \big)^2 - \big( \bm{S}^{(l)} \big)^2 - \frac{\sigma^2}{2} (n+1) \bm{I}_n \bigg) \bigg\| &\lesssim \lambda_{\max} \sigma \sqrt{nL} +  \sigma^2 n \sqrt{L} \\
    &\leq \frac{L \lambda^2}{4}
\end{align*}
where the final inequality holds whenever $\lambda/\sigma \geq C \max\big\{ \kappa \sqrt{\frac{n}{L}}, \frac{\sqrt{n}}{L^{1/4}}\big\}$ which holds when \eqref{maincondition} is satisfied.  Therefore, by the Davis-Kahan Theorem, we have that with this same probability,
\begin{align*}
    \| \sin\Theta(\bm{\hat U}_0,\bm{U}) \| &\lesssim \frac{\lambda_{\max} \sigma \sqrt{nL} +  \sigma^2 n \sqrt{L}}{L \lambda^2 } \\
    &\lesssim \frac{\kappa \sigma \sqrt{n}}{\sqrt{L}\lambda} + \frac{\sigma^2 n}{\lambda^2 \sqrt{L}}.
\end{align*}   
Using the  deterministic inequality $\| \bm{A} \|_F \leq \sqrt{r} \| \bm{A} \|$ for any matrix $\bm{A}$ of rank at most $r$, we obtain that
\begin{align*}
      \| \sin\Theta(\bm{\hat U}_0,\bm{U}) \|_F &\leq \frac{C\kappa \sigma \sqrt{nr}}{\sqrt{L}\lambda} + \frac{C \sigma^2 n \sqrt{r}}{\lambda^2 \sqrt{L}}
\end{align*}
whenever the event in \cref{lem:initialconcentration} holds.
\end{proof}

We prove \cref{thm:mainthm}  by induction.    For simplicity we only consider the regime where $\sqrt{n} \gtrsim \lambda/\sigma \gtrsim \sqrt{n}/L^{1/4}$ since otherwise no additional iterations are required and the result holds for $t = 0 $.  We also do not consider the updates for $\bm{\hat R}^{(l)}_t$ since we can encode these updates directly into the algorithm, since $\bm{\hat R}^{(l)}_t = \bm{\hat U}_t\t \bm{A}^{(l)} \bm{\hat U}_t$.  Therefore, the gradient updates are equivalently given by
\begin{align*}
    \bm{\hat U}_{t+1/2} &= \bm{\hat U}_t - \frac{\eta}{L} \sum_{l=1}^{L} \big( \bm{\hat U}_t \bm{\hat R}^{(l)}_t \bm{\hat U}_t\t - \bm{A}^{(l)} \big) \bm{\hat U}_t \bm{\hat R}_t^{(l)} \\
    &= \bm{\hat U}_t - \frac{\eta}{L} \sum_{l=1}^{L} \big( \bm{\hat U}_t \bm{\hat U}_t\t \bm{A}^{(l)} \bm{\hat U}_t \bm{\hat U}_t\t - \bm{A}^{(l)} \big) \bm{\hat U}_t \bm{\hat U}_t\t \bm{A}^{(l)} \bm{\hat U}_t; \\
    \bm{\hat U}_{t+1} &= {\sf SVD}_r( \bm{\hat U}_{t+1/2} ).
\end{align*}
This algorithm does not include any updates for $\bm{\hat R}^{(l)}_t$, and it is equivalent to \cref{alg}.

 Define  the event
\begin{subequations}
 \begin{align}
    \mathcal{E}_{{\sf good}} &:= \bigg\{ \bigg\| \sum_l \bm{N}^{(l)} \bm{Q} \bm{S}^{(l)} \bigg\| \leq C \sigma \sqrt{rnL} \lambda_{\max} \| \bm{Q} \|_F \text{ for all matrices $\bm{Q}$ of rank at most $2r$} \bigg\} \label{assertion1}\\
 \begin{split}   &\bigcap   \bigg\{ \bigg\| \sum_{l} \bm{N}^{(l)} \bm{Q} \bm{N}^{(l)} - \frac{\sigma^2}{2} L \big(\bm{Q}\t + {\sf Tr}(\bm{Q}) \bm{I} \big)  \bigg\| \leq C \sigma^2  n r\sqrt{L} 
 \| \bm{Q} \|_F \\
    &\qquad\qquad \qquad \qquad \qquad \qquad\qquad \qquad \text{  for all  matrices $\bm{Q}$ of rank at most $2r$} \bigg\} 
\end{split} \label{assertion2}  \\ 
    &\bigcap   \bigg\{  \bigg\| \sum_l  \bm{U}_{\perp}\t \bm{N}^{(l)} \bm{UU}\t \bm{N}^{(l)} \bm{U}\bigg\| \leq C \sigma^2 r \max\{ \sqrt{nL}, n \} \bigg\} 
    \label{assertion3} \\
    &\bigcap \bigg\{ \bigg\| \sum_l \bm{N}^{(l)} \bm{S}^{(l)} \bigg\| \leq C \sigma \sqrt{nL} \lambda_{\max} \bigg\} \label{assertion4}
        \end{align}
    \end{subequations}
We will prove the result by induction on $t$ on the event $\mathcal{E}_{{\sf good}}$.  The following lemma shows that $\mathcal{E}_{{\sf good}}$ holds with high probability.
\begin{lemma} \label{lem:goodevent}
Suppose that each $\bm{N}^{(l)}$ is a subgaussian Wigner matrix with variance $\sigma^2$ and $\psi_2$ norm bounded by $\sigma$.  Suppose that $\max_l \| \bm{S}^{(l)} \| \leq \lambda_{\max}$ and $\log(L) \leq c n r$ for some sufficiently small constant $c$.  Then the event $\mathcal{E}_{{\sf good}}$ holds with probability at least $1- \exp( -c n)$.
\end{lemma}
\begin{proof}
  See \cref{sec:goodeventproof}.
\end{proof}

To give the proof we will require some additional bounds on the event $\mathcal{E}_{{\sf good}}$.  Define
\begin{align}
\begin{split}
    {\sf L}_t &:= - \frac{1}{L} \big(  \bm{\hat U}_t \bm{\hat U}_t\t - \bm{UU}\t \big) \sum_l \bm{N}^{(l)} \bm{\hat U}_t \bm{\hat U}_t\t \bm{S}^{(l)} \bm{\hat U}_t + \frac{1}{L} \bm{U}_{\perp} \bm{U}_{\perp}\t  \sum_l \bm{N}^{(l)} \bm{\hat U}_t \bm{\hat U}_t\t \bm{S}^{(l)} \bm{\hat U}_t \\
            &\quad -  \bm{\hat U}_t \bm{\hat U}_t\t \frac{1}{L} \sum_{l} \bm{S}^{(l)} \bm{\hat U}_t \bm{\hat U}_t\t \bm{N}^{(l)} \bm{\hat U}_t + \frac{1}{L} \sum_l \bm{S}^{(l)} \bm{\hat U}_t \bm{\hat U}_t\t \bm{N}^{(l)} \bm{\hat U}_t 
\end{split}
    \label{lt} \\
            {\sf Q}_t &:= - \frac{1}{L}  \bm{\hat U}_t \bm{\hat U}_t\t \sum_l \bm{N}^{(l)} \bm{\hat U}_t \bm{\hat U}_t\t \bm{N}^{(l)} \bm{\hat U}_t + \frac{1}{L}  \sum_l \bm{N}^{(l)} \bm{\hat U}_t \bm{\hat U}_t\t \bm{N}^{(l)} \bm{\hat U}_t, \label{qt}
            \end{align}
which are the terms linear and quadratic in $\bm{N}^{(l)}$ respectively. 
The following lemma controls the linear error. 
\begin{lemma}\label{lem:Lt}
Suppose the conditions of \cref{thm:mainthm} hold. 
Suppose that 
\begin{align*}
    \|\pert \uhat_t \|_F \leq \frac{1}{2}.
\end{align*}
Then on the event $\mathcal{E}_{{\sf good}}$ it holds that
    \begin{align*}
        \| \bm{U}_{\perp} \bm{U}_{\perp}\t {\sf L}_t \|_F &\leq C \| \bm{U}_{\perp}\t \bm{\hat U}_t \|_F \frac{\sigma\lambda_{\max} \sqrt{nr}}{\sqrt{L}} + C \frac{\sigma \lambda_{\max} \sqrt{nr}}{\sqrt{L}}; \\
          \|  {\sf L}_t \|_F & \leq  C\sqrt{r} \| \pert \uhat_t \|_F  \frac{\sigma \lambda_{\max} \sqrt{nr}}{\sqrt{L}} + C\frac{\sigma \lambda_{\max} \sqrt{nr}}{\sqrt{L}}.
    \end{align*}
\end{lemma}
\begin{proof}
    See \cref{sec:linearproof}.
\end{proof}
The following lemma bounds the term ${\sf Q}_t$.
\begin{lemma}\label{lem:Qt} 
Suppose the conditions of \cref{thm:mainthm} hold. 
Suppose that 
\begin{align*}
    \|\pert \uhat_t \|_F \leq \frac{1}{2}.
\end{align*}
Then on the event $\mathcal{E}_{{\sf good}}$ it holds that
\begin{align*}
  \| \bm{U}_{\perp} \bm{U}_{\perp}\t   {\sf Q}_t \|_F &\leq \| \pert \uhat_t\|_F \frac{C \sigma^2 n r^{3/2}}{\sqrt{L}} + \frac{C \sigma^2 r^{3/2} \max\{ \sqrt{nL},n\}}{L} \\
 \| {\sf Q}_t \|_F &\leq  \frac{C \sigma^2 n r^2}{\sqrt{L}} 
\end{align*}
    
\end{lemma}
\begin{proof}
    See \cref{sec:quadraticproof}.
\end{proof}
Finally, the following technical lemma is needed to ensure that the iterates contract. 
\begin{lemma}\label{lem:contractionlemma} Suppose the conditions of \cref{thm:mainthm} hold. 
Suppose that 
\begin{align*}
    \|\pert \uhat_t \|_F \leq \frac{1}{2}.
\end{align*}
Then on the event $\mathcal{E}_{{\sf good}}$ it holds that
\begin{align*}
    \lambda_r \bigg( \sum_l \bm{\hat U}_t\t \bm{S}^{(l)} \bm{\hat U}_t \bm{\hat U}_t\t \bm{S}^{(l)} \bm{\hat U}_t \bigg) &\geq \frac{L \lambda^2}{4} ; \\
    \| \sum_l \bm{S}^{(l)} \bm{\hat U}_t \bm{\hat U}_t\t \bm{S}^{(l)} \bm{\hat U}_t \| &\leq  L \lambda_{\max}^2.
\end{align*}
\end{lemma}

\begin{proof}
    See \cref{sec:contractionproof}.
\end{proof}

With these lemmas in place we now have the proof of \cref{thm:mainthm} for $t \geq 1$.

\begin{proof}[Proof of \cref{thm:mainthm}: Induction Step]
Define the event $\mathcal{E}_{\text{\cref{lem:initialconcentration}}}$ as the event that
\begin{align*}
    \| \sin\Theta(\bm{\hat U}^{(0)}, \bm{U} ) \|_F \leq \underbrace{C_1 \frac{\sigma \kappa \sqrt{nr}}{\lambda \sqrt{L}}}_{=: A} + \underbrace{C_2 \frac{\sigma^2 n \sqrt{r}}{\sqrt{L} \lambda^2}}_{=:B}.
\end{align*}
  Throughout we condition on the event $\mathcal{E}_{{\sf good}} \cap \mathcal{E}_{\text{\cref{lem:initialconcentration}}}$. 
By \cref{lem:initialconcentration}, $\mathcal{E}_{\text{\cref{lem:initialconcentration}}}$ holds with probability at least $1 - \exp( - c n)$, and by \cref{lem:goodevent}, $\mathcal{E}_{{\sf good}}$ holds with this same probability. 
Under \eqref{maincondition} and the event $\mathcal{E}_{{\sf good}}\cap \mathcal{E}_{\text{\cref{lem:initialconcentration}}}$, it holds that $\|\pert\uhat_t\|_F\leq \frac12$
provided $C_0$ is sufficiently large.  Therefore, suppose that at time $t$ it holds that 
\begin{align*}
    \|\sin\Theta(\bm{\hat U}_t,\bm U)\|_F \leq A+\big(1-\frac{\eta\lambda^2}{8}\big)^t B.
\end{align*}
Clearly this holds at time $t = 0$.  
The gradient update gives 
\begin{align*}
     \bm{\hat U}_{t+1/2} &= \bm{\hat U}_t -\frac{\eta}{L} \sum_{l=1}^{L} \bigg( \bm{\hat U}_t \uhat_t\t \bm{A}\l \uhat_t \bm{\hat U}_t\t -\bm{A}^{(l)} \bigg) \bm{\hat U}_{t} \bm{\hat U}_t\t \bm{A}^{(l)} \bm{\hat U}_t \\ 
           &= \bm{\hat U}_t - \frac{\eta}{L} \big( \bm{\hat U}_t \bm{\hat U}_t\t - \bm{UU}\t \big) \sum_{l=1}^{L} \bm{S}^{(l)} \bm{\hat U}_t \bm{\hat U}_t\t \big( \bm{S}^{(l)}  \big) \bm{\hat U}_t  \\
         &\qquad - \frac{\eta}{L} \sum_{l=1}^{L} \bigg( \big( \bm{\hat U}_t \bm{\hat U}_t\t - \bm{U} \bm{U}\t \big) \bm{N}^{(l)}  - \bm{U}_{\perp} \bm{U}_{\perp}\t \bm{N}^{(l)} \bigg) \bm{\hat U}_t \bm{\hat U}_t\t \big( \bm{S}^{(l)}  \big) \bm{\hat U}_t  \\
        &\qquad - \frac{\eta}{L} \sum_{l=1}^{L}  \bm{\hat U}_t \bm{\hat U}_t\t \big( \bm{S}^{(l)} + \bm{N}^{(l)} \big) \bm{\hat U}_t \bm{\hat U}_t\t \bm{N}^{(l)} \bm{\hat U}_t \\
        &\qquad + \frac{\eta}{L} \sum_{l=1}^{L} \big( \bm{S}^{(l)} + \bm{N}^{(l)} \big) \bm{\hat U}_t \bm{\hat U}_t\t \big(  \bm{N}^{(l)} \big) \bm{\hat U}_t .
           \end{align*}
           From the definitions of ${\sf L}_t$ and ${\sf Q}_t$ in \cref{lt} and \cref{qt} it holds that 
    \begin{align*}
            \bm{U}_{\perp} \bm{U}_{\perp}\t \bm{\hat U}_{t+1/2} 
            &= \bm{U}_{\perp} \bm{U}_{\perp}\t \bm{\hat U}_t - \frac{\eta}{L} \bm{U}_{\perp} \bm{U}_{\perp}\t \bm{\hat U}_t \bm{\hat U}_t\t \sum_{l} \bm{S}^{(l)} \bm{\hat U}_t \bm{\hat U}_t\t \bm{S}^{(l)} \bm{\hat U}_t 
        + \eta \bm{U}_{\perp} \bm{U}_{\perp}\t \bigg( {\sf L}_t +  {\sf Q}_t \bigg).
    \end{align*}
From \cref{lem:Lt} and \cref{lem:Qt} we have that
\begin{align*}
         \| \bm{U}_{\perp} \bm{U}_{\perp}\t {\sf L}_t \|_F &\leq C \| \bm{U}_{\perp}\t \bm{\hat U}_t \|_F \frac{\sigma \sqrt{nr}\lambda_{\max}}{\sqrt{L}} + C \frac{\sigma \lambda_{\max} \sqrt{nr}}{\sqrt{L}};  \\
   \| \bm{U}_{\perp} \bm{U}_{\perp}\t   {\sf Q}_t \|_F &\leq \| \pert \uhat_t \|_F \frac{C \sigma^2 n r^{3/2}}{\sqrt{L}} + \frac{C\sigma^2 r^{3/2} \max\{\sqrt{nL},n\}}{L}. 
\end{align*}
By \cref{lem:contractionlemma},
\begin{align*}
    \lambda_r\left(\frac{\eta}{L}\sum_l\bm{\hat U}_t^\top\bm S^{(l)}\bm{\hat U}_t\bm{\hat U}_t^\top\bm S^{(l)}\bm{\hat U}_t\right)\geq \frac{\eta\lambda^2}{4}.
\end{align*}
Since $\eta\lambda^2\leq 1$, this yields
\begin{align*}
    \|\bm U_\perp^\top\bm{\hat U}_{t+1/2}\|_F
    &\leq
    \|\bm U_\perp^\top\bm{\hat U}_t\|_F
    \left(1-\frac{\eta\lambda^2}{4}+\eta\frac{C\sigma\sqrt{nr}\lambda_{\max}}{\sqrt L}+\eta\frac{C\sigma^2nr^{3/2}}{\sqrt L}\right)
    +\eta\frac{C\sigma\lambda_{\max}\sqrt{nr}}{\sqrt L}
    +\eta\frac{C\sigma^2r^{3/2}\max\{\sqrt{nL},n\}}{L}.
\end{align*}
Under \eqref{maincondition}, since $C_1$ is fixed, it holds that
\begin{align*}
    \frac{C\sigma\lambda_{\max}\sqrt{nr}}{\sqrt L}+\frac{C\sigma^2r^{3/2}\max\{\sqrt{nL},n\}}{L}\leq \frac{\lambda^2}{64}A.
\end{align*}
Therefore, from the induction hypothesis,
\begin{align*}
    \|\bm U_\perp^\top\bm{\hat U}_{t+1/2}\|_F
    &\leq
    \bigg(1-\frac{7\eta\lambda^2}{32}\bigg)\|\bm U_\perp^\top\bm{\hat U}_t\|_F+\frac{\eta\lambda^2}{64}A\\
    &\leq
    \bigg(1-\frac{7\eta\lambda^2}{32}\bigg)\bigg(A+\big(1-\frac{\eta\lambda^2}{8}\big)^tB\bigg)+\frac{\eta\lambda^2}{64}A\\
    &=
    A\left(1-\frac{13\eta\lambda^2}{64}\right)
    +\left(1-\frac{\eta\lambda^2}{8}\right)^tB\left(1-\frac{7\eta\lambda^2}{32}\right).
\end{align*}
We next bound the projection step.  Write
\begin{align*}
    \bm{\hat U}_{t+1/2}=\bm{\hat U}_{t+1}\bm{\hat\Sigma}_t\bm{\hat V}_t^\top.
\end{align*}
Then
\begin{align*}
    \|\sin\Theta(\bm{\hat U}_{t+1},\bm U)\|_F \leq \|\bm{\hat\Sigma}_t^{-1}\|\,\|\bm U_\perp^\top\bm{\hat U}_{t+1/2}\|_F.
\end{align*}
We claim that
\begin{align*}
    \|\bm{\hat\Sigma}_t^{-1}\|\leq 1+\frac{\eta\lambda^2}{32}.
\end{align*}
Indeed, write
\begin{align*}
    \bm{\hat U}_{t+1/2}=\bm{\hat U}_t+\mathcal I_t,
\end{align*}
where
\begin{align*}
    \mathcal I_t:=
    -\frac{\eta}{L}\big(\bm{\hat U}_t\bm{\hat U}_t^\top-\bm U\bm U^\top\big)\sum_l\bm S^{(l)}\bm{\hat U}_t\bm{\hat U}_t^\top\bm S^{(l)}\bm{\hat U}_t
    +\eta{\sf L}_t+\eta{\sf Q}_t.
\end{align*}
Using $\eta\lambda_{\max}^2\leq c_\eta$ and the induction hypothesis,
\begin{align*}
    \left\|
    \frac{\eta}{L}\big(\bm{\hat U}_t\bm{\hat U}_t^\top-\bm U\bm U^\top\big)\sum_l\bm S^{(l)}\bm{\hat U}_t\bm{\hat U}_t^\top\bm S^{(l)}\bm{\hat U}_t
    \right\|
    &\leq 2\eta\lambda_{\max}^2\|\pert\uhat_t\|_F\\
    &\leq 2c_\eta\left[A+\left(1-\frac{\eta\lambda^2}{8}\right)^t B\right].
\end{align*}
Since $\eta\lambda^2\geq \frac{c_\eta'}{\kappa^2},$
\eqref{maincondition} implies that $A+B \leq \frac{c_\eta'}{1024c_\eta}\frac{1}{\kappa^2}$ and hence 
\begin{align*}
  \big\|
    \frac{\eta}{L}\big(\bm{\hat U}_t\bm{\hat U}_t^\top-\bm U\bm U^\top\big)\sum_l\bm S^{(l)}\bm{\hat U}_t\bm{\hat U}_t^\top\bm S^{(l)}\bm{\hat U}_t
    \big\| &\leq 2 c_{\eta} ( A + B) \\
    &\leq 2 c_{\eta} \frac{c_\eta'}{1024c_\eta}\frac{1}{\kappa^2} \\
    &\leq \frac{\eta\lambda^2}{512}.
    \end{align*}
Similarly, by \cref{lem:Lt,lem:Qt} and \eqref{maincondition},
\begin{align*}
    \|\eta{\sf L}_t\|_F+\|\eta{\sf Q}_t\|_F
    &\leq
    \eta\|\pert\uhat_t\|_F\frac{C\sigma\sqrt n r\lambda_{\max}}{\sqrt L}
    +\eta\frac{C\sigma\lambda_{\max}\sqrt{nr}}{\sqrt L}
    +\eta\frac{C\sigma^2nr^2}{\sqrt L}\\
    &\leq \frac{\eta\lambda^2}{512}.
\end{align*}
Hence, it holds that $\|\mathcal I_t\|\leq \frac{\eta\lambda^2}{256}.$
Therefore,
\begin{align*}
    \left\|\bm{\hat U}_{t+1/2}^\top\bm{\hat U}_{t+1/2}-\bm I_r\right\|\leq 2\|\mathcal I_t\|+\|\mathcal I_t\|^2\leq \frac{\eta\lambda^2}{128}+\frac{\eta^2\lambda^4}{256^2}\leq \frac{\eta\lambda^2}{64}.
\end{align*}
By Weyl's inequality, $\lambda_r^2(\bm{\hat\Sigma}_t)\geq 1-\frac{\eta\lambda^2}{64}.$ 
Thus, since $\eta\lambda^2\leq 1$,
\begin{align*}
    \|\bm{\hat\Sigma}_t^{-1}\|\leq \frac{1}{\sqrt{1-\eta\lambda^2/64}}\leq 1+\frac{\eta\lambda^2}{32}.
\end{align*}
Therefore, we have shown that 
\begin{align*}
    \|\sin\Theta(\bm{\hat U}_{t+1},\bm U)\|_F
    &\leq
    \left(1+\frac{\eta\lambda^2}{32}\right)
    \left[
        A\left(1-\frac{13\eta\lambda^2}{64}\right)
        +\left(1-\frac{\eta\lambda^2}{8}\right)^tB\left(1-\frac{7\eta\lambda^2}{32}\right)
    \right]\\
    &\leq
    A+\left(1-\frac{\eta\lambda^2}{8}\right)^{t+1}B.
\end{align*}
Hence, for every $t\geq0$, on the event $\mathcal{E}_{{\sf good}}$,
\begin{align*}
    \|\sin\Theta(\bm{\hat U}_{t},\bm U)\|_F
    \leq
    C_1\frac{\sigma\kappa\sqrt{nr}}{\lambda\sqrt L}
    +
    \left(1-\frac{\eta\lambda^2}{8}\right)^t
    C_2\frac{\sigma^2 n\sqrt r}{\sqrt L\lambda^2}.
\end{align*}
Since $\eta\lambda^2>c_\eta'/\kappa^2$, we also have
\begin{align*}
    \|\sin\Theta(\bm{\hat U}_{t},\bm U)\|_F
    \leq
    C_1\frac{\sigma\kappa\sqrt{nr}}{\lambda\sqrt L}
    +
    \left(1-\frac{c_\eta'}{8\kappa^2}\right)^t
    C_2\frac{\sigma^2 n\sqrt r}{\sqrt L\lambda^2}.
\end{align*}
This completes the proof.
\end{proof}


\subsection{Proof of \cref{lem:initialconcentration}} \label{sec:initiallemproof}
\begin{proof}
Without loss of generality we let $\sigma^2 = 1$. 
    We define 
    \begin{align*}
    S &:= \bigg\| \sum_{l} \bigg( (\bm{A}^{(l)})^2 - (\bm{S}^{(l)})^2 - \frac{1}{2} (n + 1) \bm{I}_n \bigg) \bigg\|.
    \end{align*}
    It holds that $S \leq S_L + S_Q$ where $S_L$ and $S_Q$ are defined via
    \begin{align*}
        S_L :&= \bigg\| \sum_{l} \bigg( \bm{N}^{(l)} \bm{S}^{(l)} + \bm{S}^{(l)} \bm{N}^{(l)} \bigg) \bigg\|; \\
        S_Q :&=   \bigg\| \sum_{l} \bigg(  (\bm{N}^{(l)} )^2 - \mathbb{E}\big( \bm{N}^{(l)}\big)^2  \bigg) \bigg\|.
\end{align*}
We will bound each term separately.   
\\ \ \\ \noindent
\textbf{Bounding $S_L$:} First, we note that 
\begin{align*}
    \sum_l \bm{N}^{(l)} \bm{S}^{(l)} &= \sum_l \mathcal{P}_{{\sf upper-diag}}( \bm{N}^{(l)}) \bm{S}^{(l)} + \sum_l \mathcal{P}_{{\sf upper-diag}}^c (\bm{N}^{(l)}) \bm{S}^{(l)},
\end{align*}
where $\mathcal{P}_{{\sf upper-diag}}$ is the operator that sets its lower off-diagonal elements to zero.  Without loss of generality we bound the first term; the second is similar.

Let $x$ be any deterministic unit vector in the span of $\bm{UU}\t$, and let $y$ be any deterministic $n$-dimensional unit vector.  Then we can write
\begin{align*}
y\t \sum_l  \mathcal{P}_{{\sf upper-diag}} (\bm{N}^{(l)}) \bm{S}^{(l)} x &= \sum_{1\leq i \leq n, 1 \leq j \leq n} \sum_{k=1}^{n} \sum_l y_ i  \mathcal{P}_{{\sf upper-diag}} (\bm{N}^{(l)})_{ij} \bm{S}^{(l)}_{jk}  x_k.
\end{align*}
Note that for fixed $x$ and $y$ this is a sum of independent subgaussian random variables.  We will note that its $\psi_2$ norm can be bounded by
\begin{align*}
\sum_{1 \leq i \leq n,1\leq j\leq n}\sum_l y_i^2  \bigg( \sum_k \bm{S}^{(l)}_{jk }x_k \bigg)^2  &= \sum_l \sum_j \| e_j\t \bm{S}^{(l)} x \|^2 = \sum_l \| \bm{S}^{(l)} x \|^2 \leq L \lambda_{\max}^2,
\end{align*}
since $\sup_{\|x\|=1} \| \bm{S}^{(l)} x \|^2 \leq \lambda_{\max}^2$.  By taking a union bound via $\eps$-net, it holds that 
\begin{align*}
    \| \sum_l \bm{N}^{(l)} \bm{S}^{(l)} \| \leq C t
\end{align*}
with probability at least $1 - \exp\bigg( cn - \frac{t^2}{L \lambda_{\max}^2} \bigg).$ 
Consequently, letting $ \frac{t^2}{L \lambda_{\max}^2}= s^2 $, we obtain that
\begin{align*}
     \| \sum_l \bm{N}^{(l)} \bm{S}^{(l)} \|  \leq C s \sqrt{L} \lambda_{\max}
\end{align*}
with probability at least $1 - \exp\bigg( cn - s^2 \bigg).$ Therefore, we may let $s = C \sqrt{n} $ to obtain that
\begin{align*}
    \| \sum_l \bm{N}^{(l)} \bm{S}^{(l)} \|  \leq C \sqrt{nL} \lambda_{\max}
\end{align*}
with probability at least $1 - \exp(- c n).$
\\ \ \\ \noindent
\textbf{Bounding $S_Q$}: 
First, we have
\begin{align*}
S_Q &=     \bigg\| \sum_l (\bm{N}^{(l)})^2 - \mathbb{E} (\bm{N}^{(l)} )^2 \bigg\| \leq \bigg\| \sum_l \mathcal{H}(\bm{N}^{(l)})^2  \bigg\| + \bigg\| \sum_l\mathcal{D}( \bm{N}^{(l)})^2 - \mathbb{E} (\bm{N}^{(l)} )^2 \bigg\|,
\end{align*}
where $\mathcal{H}(\cdot)$ is the \emph{hollowing operator} that sets the diagonal to zero, and $\mathcal{D}$ is the \emph{diagonal operator} that sets the off-diagonal to zero. We bound each term in turn.
\begin{itemize}
    \item \textbf{The off-diagonal term.} 
Observe that the first term can be written as a U-statistic in the entries of $\bm{N}^{(l)}$, and we may apply Theorem 1 of \citet{pena_decoupling_1995} to obtain that 
\begin{align*}
    \p\bigg\{ \sum_l \mathcal{H}(\bm{N}^{(l)})^2 \bigg\| > t \bigg\} \leq  C \p \bigg\{ \bigg\| \sum_l \mathcal{H} \bigg(\bm{N}^{(l)} \bm{\tilde N}^{(l)} \bigg) \bigg\| > C t \bigg\},
\end{align*}
where $\bm{\tilde N}^{(l)}$ is an independent copy.  We then have that 
\begin{align*}
    \bigg\| \sum_l \mathcal{H}\bigg( \bm{N}^{(l)} \bm{\tilde N}^{(l)} \bigg) \bigg\| \leq \underbrace{\bigg\| \sum_l \bm{N}^{(l)} \bm{\tilde N}^{(l)} \bigg\|}_{=: T_1} + \underbrace{\bigg\| \sum_l \mathcal{D} \bigg( \bm{N}^{(l)} \bm{\tilde N}^{(l)} \bigg) \bigg\|}_{=: T_2}.
\end{align*}
We now bound each term $T_1$ and $T_2$ as follows.
\begin{itemize}
    \item \textbf{The term $T_2$}.   For the diagonal term above, we note that the $i$'th diagonal entry can be written as
\begin{align*}
    \sum_l \sum_j \bm{N}^{(l)}_{ij} \bm{\tilde N}^{(l)}_{ij}.
\end{align*}
Conditional on $\bm{\tilde N}^{(l)}$, the above is a subgaussian random variable with variance proxy $\sum_l \sum_j \big( \bm{\tilde N}^{(l)}_{ij} \big)^2$.  Therefore, it can be bounded via
\begin{align*}
      \bigg| \sum_l \sum_j \bm{N}^{(l)}_{ij} \bm{\tilde N}^{(l)}_{ij} \bigg| \leq C t \sqrt{\sum_l \sum_j \big( \bm{\tilde N}^{(l)}_{ij} \big)^2}
\end{align*}
with probability at least $1 - 2 \exp( - c t^2)$, conditional on $\bm{\tilde N}^{(l)}$.  Using concentration of the norm for subgaussian random variables, uniformly over $i$ it holds that
\begin{align}
    \sqrt{\sum_l \sum_j \big( \bm{\tilde N}^{(l)}_{ij} \big)^2} \leq C \sqrt{nL} \label{eventevent}
\end{align}
with probability at least $1 - 2 n \exp( - c nL)$.  Therefore, on this event, for all $i$, with probability at least $1 - 2 n \exp( - c t^2)$  it holds that
\begin{align*}
     \bigg| \sum_l \sum_j \bm{N}^{(l)}_{ij} \bm{\tilde N}^{(l)}_{ij} \bigg| \leq C t \sqrt{nL},
\end{align*}
whenever the event \eqref{eventevent} holds.  Combining it all we have that this quantity is bounded by $C n \sqrt{L}$ with probability at least $1 - 2 n \exp( -c n) - 2 n\exp( - c nL)$.  
\item \textbf{The term $T_1$.} We replicate our proof for $\| \sum_l \bm{N}^{(l)} \bm{S}^{(l)} \|$ only replacing $\bm{S}^{(l)}$ with $\bm{\tilde N}^{(l)}$.  Conditional on the event
\begin{align*}
    \sup_{ \|x\| = 1} \sum_l \| \bm{N}^{(l)} x \|^2 \leq C Ln,
\end{align*}
it holds that 
\begin{align*}
    \bigg\| \sum_l \bm{N}^{(l)} \bm{\tilde N}^{(l)} \bigg\| \leq C s \sqrt{Ln}
\end{align*}
with probability at least $1 - 2 \exp( c n - s^2)$.  Taking $s \asymp \sqrt{n}$ yields the same bound, which holds with probability at least $1 - 2 \exp( - c n) - 2 \exp( - c Ln)$.  
\end{itemize}
We have thus shown that
\begin{align*}
    \bigg\| \sum_l \mathcal{H}\bigg( (\bm{N}^{(l)})^2 \bigg) \bigg\| \leq C n \sqrt{L}
\end{align*}
with probability at least $1 - 8 n \exp( - c n)$. 
\item \textbf{The Diagonal Term}.   We observe that
\begin{align*}
    \bigg\| \mathcal{D} \bigg( \sum_l (\bm{N}^{(l)})^2 - \mathbb{E}( \bm{N}^{(l)})^2 \bigg) \bigg\| &= \max_i \bigg| \sum_{l,j} ( \bm{N}^{(l)}_{ij})^2 - \mathbb{E}(\bm{N}^{(l)}_{ij})^2 \bigg|.
\end{align*}
For fixed $i$, this is a subexponential random variable.  Bernstein's inequality implies that 
\begin{align*}
    \bigg| \sum_{l,j} ( \bm{N}^{(l)}_{ij})^2 - \mathbb{E}(\bm{N}^{(l)}_{ij})^2 \bigg| \leq t
\end{align*} 
with probability at least $1 - 2 \exp( - c \min\{ \frac{t^2}{Ln}, t \} )$.  Take $t =  C n \sqrt{L}$ and a union bound to show that 
\begin{align*}
     \bigg\| \mathcal{D} \bigg( \sum_l (\bm{N}^{(l)})^2 - \mathbb{E}( \bm{N}^{(l)})^2 \bigg) \bigg\| \lesssim n \sqrt{L}
\end{align*}
with probability at least $1 - 2 n \exp( - c n)$.  \end{itemize}
Combining all these inequalities and adjusting constants completes the proof.
\end{proof}

\subsection{Proof of \cref{lem:goodevent}} \label{sec:goodeventproof}

\begin{proof}[Proof of \cref{lem:goodevent}]
The proof for \eqref{assertion1} is similar to the proof of \cref{lem:initialconcentration}.  First, let $\bm{Q}$ be any fixed matrix of rank at most $2r$.  Slightly modifying the proof of \cref{lem:initialconcentration} shows that we have that for any fixed matrix $\bm{Q}$ it holds that 
\begin{align*}
    \| \sum_l \bm{N}^{(l)} \bm{Q} \bm{S}^{(l)} \| \lesssim C \sigma  s \sqrt{L} \lambda_{\max} \| \bm{Q} \|_F
\end{align*}
with probability at least $1- \exp( cn - s^2)$. Next, consider
\begin{align*}
    \bm{\tilde Q} := \argsup_{\bm{Q}: \| \bm{Q}\|_F \leq 1; \text{rank}(\bm{Q}) \leq 2r} \| \sum_l \bm{N}^{(l)} \bm{Q} \bm{S}^{(l)} \|.
\end{align*}
Let $\mathcal{N}_{\eps}$ be an $\eps$-net for matrices of rank at most $2r$ with Frobenius norm at most 1.  Then $|\mathcal{N}_{\eps}| \leq \exp( c n r)$ by Lemma 3.1 of \citet{candes_tight_2011}. Define the event
\begin{align}
   \sup_{\bm{Q} \in \mathcal{N}_{\eps}} \| \sum_l \bm{N}^{(l)} \bm{Q} \bm{S}^{(l)} \| \leq C \sigma  \sqrt{Lnr} \lambda_{\max}. \label{bestfriendsparty}
\end{align}
By a union bound \eqref{bestfriendsparty} holds with probability at least $1 - \exp( - c n r)$ for an appropriate choice of $C$ above.  Let $\bm{Q}_{\eps}$ be the matrix in $\mathcal{N}_{\eps}$ within $\eps$ of $\bm{\tilde Q}$.  Without loss of generality we can assume that $(\bm{\tilde Q} - \bm{Q}_{\eps}) \bm{U} \neq 0$ (since if it does the result is zero). In this case, note that 
\begin{align*}
    \| \big( \bm{\tilde Q} - \bm{Q}_{\eps} \big) \bm{UU}\t \|_F \leq \eps < 1,
\end{align*}
and hence $\big( \bm{\tilde Q} - \bm{Q}_{\eps} \big) \bm{UU}\t$ is a matrix of rank at most $2r$ with Frobenius norm at most 1.  Furthermore, when \eqref{bestfriendsparty} holds,
\begin{align*}
M &:= \| \sum_l \bm{N}^{(l)}\bm{\tilde Q}\bm{S}^{(l)} \| \\
&\leq \| \sum_l \bm{N}^{(l)} \big( \bm{\tilde Q}  - \bm{Q}_{\eps} \big) \bm{S}^{(l)} \| + \| \sum_l \bm{N}^{(l)} \big(  \bm{Q}_{\eps} \big) \bm{S}^{(l)} \| \\
    &\leq \| \big(\bm{\tilde Q} - \bm{Q}_{\eps}\big) \bm{UU}\t \|_F \| \sum_l \bm{N}^{(l)} \frac{\big( \bm{\tilde Q}  - \bm{Q}_{\eps} \big)}{ \| \big(\bm{\tilde Q} - \bm{Q}_{\eps}\big)\bm{UU}\t \|_F} \bm{S}^{(l)} \| + \| \sum_l \bm{N}^{(l)} \big(  \bm{Q}_{\eps} \big) \bm{S}^{(l)} \| \\
    &\leq \eps M + C  \sigma  \sqrt{Lnr} \lambda_{\max}.
\end{align*} 
 This implies that $M \lesssim \sqrt{Lnr} \lambda_{\max}$ with probability at least $1- \exp( - c n r)$, which is the assertion \eqref{assertion1}.

To prove \eqref{assertion2}, we similarly apply an $\eps$-net, but we also use a decoupling argument. First, fix any matrix $\bm{Q}$ of rank at most $2r$.  Observe that since each $\bm{N}^{(l)}$ is a Wigner matrix, it holds that 
\begin{align*}
    \mathbb{E} \bm{N}^{(l)} \bm{Q} \bm{N}^{(l)} &= \frac{\sigma^2}{2} \bigg( \bm{Q}\t + {\sf diag}\big( {\sf Tr}( \bm{Q})\big) \bigg).
\end{align*}
Consequently, we have that
\begin{align*}
\bigg\| \sum_{l} \bm{N}^{(l)} \bm{Q} \bm{N}^{(l)} - \frac{\sigma^2}{2} L \big( \bm{Q}\t+ {\sf Tr}(\bm{Q}) \bm{I} \big) \bigg\| &\leq \bigg\| \sum_l \mathcal{H} \bigg(  \bm{N}^{(l)} \bm{Q} \bm{N}^{(l)} - \frac{\sigma^2}{2}  \bm{Q}\t \bigg)  \bigg\| \\
&\quad + \bigg\| \sum_l \mathcal{D} \bigg(   \bm{N}^{(l)} \bm{Q} \bm{N}^{(l)} -\frac{\sigma^2}{2}  \big(  \bm{Q} + {\sf Tr}(\bm{Q}) \bm{I} \big) \bigg)  \bigg\|,
\end{align*}
where we have separated the off-diagonal from the diagonal. We will bound each term separately for fixed $\bm{Q}$.  
\begin{itemize}
    \item \textbf{The off-diagonal. } We have that the $i,i'$ entry is given by (for $i \neq i'$)
\begin{align*}
     \bigg( \bm{N}^{(l)} \bm{Q} \bm{N}^{(l)} \bigg)_{ii'} - \frac{\sigma^2}{2}\bm{Q}_{i'i} &= \sum_{j,k=1}^{n} \bm{N}^{(l)}_{ij} \bm{Q}_{jk} \bm{N}^{(l)}_{i' k} -\frac{\sigma^2}{2} \bm{Q}_{i'i}.
\end{align*} 
Let $\bm{x,y}$ be unit vectors.  Then we have that
\begin{align*}
  \sum_l  \bm{x}\t \mathcal{H} &\bigg( \bm{N}^{(l)} \bm{Q} \bm{N}^{(l)} - \frac{\sigma^2}{2} \bm{Q}\t \bigg) \bm{y} \\
    &=\sum_l\bigg(  \sum_{i \neq i'} \sum_{j,k} \bm{x}_i \bm{y}_{i'}  \bm{N}^{(l)}_{ij} \bm{Q}_{jk} \bm{N}^{(l)}_{i' k} - \frac{\sigma^2}{2} \bm{x}\t \mathcal{H} (\bm{Q}\t) \bm{y}\bigg) \\
 &= \sum_l \bigg(  \sum_{i\neq i'}  \bm{x}_i \bm{y}_{i'}  \bm{N}^{(l)}_{ii'} \bm{Q}_{i'i} \bm{N}^{(l)}_{i' i} + \sum_{i\neq i'} \sum_{j \neq i', k\neq i}  \bm{x}_i \bm{y}_{i'}  \bm{N}^{(l)}_{ij} \bm{Q}_{jk} \bm{N}^{(l)}_{i' k}  - \bm{x}\t \mathcal{H} (\bm{Q}\t) \bm{y} \bigg) \\
 &=  \sum_l \sum_{i< i'} \big[ \big( \bm{N}^{(l)}_{ii'} \big)^2 - \frac{\sigma^2}{2}  \big] \bigg( \bm{x}_i \bm{y}_{i'}  \bm{Q}_{i'i} +  \bm{x}_{i'} \bm{y}_{i}  \bm{Q}_{ii'} \bigg)  + \sum_l \sum_{i\neq i'} \sum_{j \neq i', k\neq i}  \bm{x}_i \bm{y}_{i'}  \bm{N}^{(l)}_{ij} \bm{Q}_{jk} \bm{N}^{(l)}_{i' k}.
\end{align*}
We will control each of these two terms separately.
\begin{itemize}
    \item \textbf{The first term}: Observe that the first term is a sum of independent mean-zero sub-exponential random variables.  We have that
    \begin{align*}
        \sum_l \sum_{i < i'} \bigg\| \big[ \big( \bm{N}^{(l)}_{ii'} \big)^2 - \frac{\sigma^2}{2} \big] \bigg( \bm{x}_i \bm{y}_{i'}  \bm{Q}_{ii'} +  \bm{x}_{i'} \bm{y}_{i}  \bm{Q}_{i'i} \bigg) \bigg\|_{\psi_1}^2 \leq C L \sigma^4 \sum_{i,i'} \big( \bm{x}_i \bm{y}_{i'}  \bm{Q}_{i'i} \big)^2 \lesssim \sigma^4 L \| \bm{Q} \|^2,
    \end{align*}
    since $\bm{x}$ and $\bm{y}$ are unit vectors.  Similarly, we have that
    \begin{align*}
        \max_{i,i'} \bigg\|  \big[ \big( \bm{N}^{(l)}_{ii'} \big)^2 - \frac{\sigma^2}{2}  \big] \bigg( \bm{x}_i \bm{y}_{i'}  \bm{Q}_{i'i} +  \bm{x}_{i'} \bm{y}_{i}  \bm{Q}_{ii'} \bigg) \bigg\|_{\psi_1} \lesssim \sigma^2 \| \bm{Q} \|.
    \end{align*}
    Therefore, by the generalized Bernstein inequality (Theorem 2.8.1 of \citet{vershynin_high-dimensional_2018}), it holds that 
    \begin{align*}
     \bigg|  \sum_l  \sum_{i< i'} \big[ \big( \bm{N}^{(l)}_{ii'} \big)^2 - 1 \big] \bigg( \bm{x}_i \bm{y}_{i'}  \bm{Q}_{i'i} +  \bm{x}_{i'} \bm{y}_{i}  \bm{Q}_{ii'} \bigg) \bigg| \lesssim t \sqrt{Ln} \sigma^2 \|\bm{Q}\|
    \end{align*}
    with probability at least $1 - \exp( - c \min( t^2n, t \sqrt{Ln}) )$.  Letting $t \asymp \sqrt{n} r$ shows that with probability at least $1 - \exp\big( - c \min\{ n^2 r^2, n^2 \sqrt{Lr} \big) \geq 1 - \exp( - c nr )$, it holds that 
    \begin{align*}
         \bigg|  \sum_l  \sum_{i< i'} \big[ \big( \bm{N}^{(l)}_{ii'} \big)^2 - 1 \big] \bigg( \bm{x}_i \bm{y}_{i'}  \bm{Q}_{i'i} +  \bm{x}_{i'} \bm{y}_{i}  \bm{Q}_{ii'} \bigg) \bigg| \lesssim nr \sqrt{L} \sigma^2 \|\bm{Q}\|_F.
    \end{align*}
 \item \textbf{The second term}.  We note that for $i \neq i'$, $j \neq i',k\neq i$, the second term is a U-statistic in $\bm{N}^{(l)}_{ij}$.  By the decoupling inequality (Theorem 3.4.1 of \citet{de_la_pena_decoupling_1999}), it holds that
\begin{align*}
    \p\bigg\{ \bigg| \sum_l \sum_{i\neq i'} \sum_{j \neq i', k\neq i}  \bm{x}_i \bm{y}_{i'}  \bm{N}^{(l)}_{ij} \bm{Q}_{jk} \bm{N}^{(l)}_{i' k} \bigg| > t \bigg\} \leq C \p\bigg\{ \bigg| \sum_l \sum_{i\neq i'} \sum_{j \neq i', k\neq i}  \bm{x}_i \bm{y}_{i'}  \bm{\tilde N}^{(l)}_{ij} \bm{Q}_{jk} \bm{N}^{(l)}_{i' k}  \bigg| > C t \bigg\},
\end{align*}
where $\bm{\tilde N}^{(l)}$ is an independent copy of $\bm{N}^{(l)}$.
Define the event
\begin{align*}
    \mathcal{E} := \max_l \| \bm{N}^{(l)} \| \leq C \sigma \sqrt{nr}.
\end{align*}
By a union bound, we note that $\mathcal{E}$ holds with probability at least $1 - L \exp( - c nr)$.  We observe that on the event $\mathcal{E}$,
\begin{align*}
    \bigg\| \sum_l \sum_{i} \sum_{j} \bm{\tilde N}_{ij}^{(l)} \bm{x}_i \bigg( \sum_{k \neq i} \sum_{i' \neq i,j} \bm{Q}_{jk} \nl_{i'k} \bm{y}_{i'}  \bigg) \bigg\|_{\psi_2}^2 &\lesssim \sigma^2 L \max_{l,i} \sum_j \bigg( \sum_k \sum_{i' \neq k,i} \bm{Q}_{jk} \bm{N}^{(l)}_{i'k} \bm{y}_{i'} \bigg)^2 \\
    &\lesssim \sigma^2 L \max_l \max_i \sum_j \bigg\| e_j\t \bm{Q} \mathcal{H}\big( \mathcal{P}^{-i} \big( \bm{N}^{(l)} \big) \big) \bm{y} \bigg\|^2 \\
    &\leq \sigma^2 L \| \bm{Q} \|_F^2 \max_l \max_i \bigg\|  \mathcal{H}\big( \mathcal{P}^{-i} \big( \bm{N}^{(l)} \big) \big) \bm{y} \bigg\|^2 \\
    &\leq \sigma^2 L \| \bm{Q} \|_F^2\max_l \| \bm{N}^{(l)} \|^2 \\
    &\leq \sigma^4 L n r \| \bm{Q} \|_F^2,
\end{align*} 
where $\mathcal{P}^{-i}$ removes the $i$'th row of $\bm{N}^{(l)}$ and $\mathcal{H}$ is the hollowing operator. In the final line we used the fact that the spectral norm of any submatrix is bounded by the spectral norm of the matrix itself. 
Consequently, we have that 
\begin{align*}
    \p\bigg\{ \bigg| \sum_l \sum_{i\neq i'} \sum_{j \neq i', k\neq i}  \bm{x}_i \bm{y}_{i'}  \bm{\tilde N}^{(l)}_{ij} \bm{Q}_{jk} \bm{N}^{(l)}_{i' k}  \bigg| > C t \sigma^2 \sqrt{Lnr} \|\bm{Q}\|_F \cap \mathcal{E} \bigg\} \leq 2 \exp( - c t^2).
\end{align*}
If $t \asymp \sqrt{nr}$, then we have that with probability at least $1 - \exp( - c nr)$
\begin{align*}
     \bigg| \sum_l \sum_{i\neq i'} \sum_{j \neq i', k\neq i}  \bm{x}_i \bm{y}_{i'}  \bm{\tilde N}^{(l)}_{ij} \bm{Q}_{jk} \bm{N}^{(l)}_{i' k}  \bigg| \leq  C  \sigma^2 n r \sqrt{L} \|\bm{Q}\|_F.
\end{align*}
\end{itemize}
Combining these bounds  and taking a union bound over $\bm{x}$ and $\bm{y}$ in an $\eps$-net and applying Lemma 4.4.1 of \citet{vershynin_high-dimensional_2018} (adjusting constants if necessary) shows that
    \begin{align*}
         \bigg\| \sum_l \mathcal{H} \bigg(  \bm{N}^{(l)} \bm{Q} \bm{N}^{(l)} - \frac{\sigma^2}{2} L  \bm{Q} \bigg)  \bigg\| \lesssim  n r \sqrt{L} \sigma^2 \| \bm{Q} \|_F
    \end{align*}
    with probability at least $1 - \exp\big( - c n r \big).$
\item \textbf{The diagonal}. We note that since the matrix in question is diagonal, we need only bound it for fixed $i$ (and take a union bound over $i$).   We have that
\begin{align*}
    \sum_{l} \sum_{j,k} \bm{N}_{ij}^{(l)} \bm{Q}_{jk} \bm{N}_{ki}^{(l)} &= \sum_l  \sum_{j\neq k} \bm{N}^{(l)}_{ij} \bm{Q}_{jk} \bm{N}^{(l)}_{ki} + \sum_l  \sum_j \big( \bm{N}^{(l)}_{ij} \big)^2 \bm{Q}_{jj}.
\end{align*}
By a similar argument to the previous part of the proof, by the decoupling inequality we can show that 
 \begin{align*}
  \bigg|  \sum_l  \sum_{j\neq k} \bm{N}^{(l)}_{ij} \bm{Q}_{jk} \bm{N}^{(l)}_{ki} \bigg| &\lesssim t \sigma^2 \sqrt{Ln} \| \bm{Q} \|_F 
 \end{align*}
 with probability at least $ 1 - \exp( - c t^2 n)$.  Similarly, 
 \begin{align*}
    \bigg| \sum_l \sum_j \bigg( \big( \bm{N}^{(l)}_{ij} \big)^2 -  \frac{\sigma^2}{2} - \frac{\sigma^2}{2} \mathbb{I}_{\{i = j\}} \bigg)\bm{Q}_{jj}\bigg| &\lesssim t \sqrt{Ln} \sigma^2 \| \bm{Q} \|_F
  \end{align*}
  with probability at least $1 - \exp( - c \min(t^2 n, t \sqrt{Ln} ))$.  The rest of the argument is exactly the same as the previous part, which yields the proof of \cref{assertion2}.  
\end{itemize}
To prove \eqref{assertion3}, we first let $\bm{x},\bm{y}$ be deterministic unit vectors of dimension $n-r$ and $r$ respectively.  Then we have that 
\begin{align*}
    \sum_l \bm{x}\t \pert \nl \U\U\t \nl \U \bm{y} &= \sum_l \langle \bm{z}^{(l)}_1 , \bm{z}^{(l)}_2 \rangle,
\end{align*}
where $\bm{z}^{(l)}_1$ and $\bm{z}^{(l)}_2$ are uncorrelated $r$-dimensional random vectors with covariances equal to $\sigma^2 \bm{I}_r$ respectively.   The $\psi_1$ norm is bounded by $\sigma^2 r$, and we can apply Bernstein's inequality for sums of subexponential random variables to yield that 
\begin{align*}
    \p\bigg\{ \bigg| \sum_l \langle \bm{z}^{(l)}_1 , \bm{z}^{(l)}_2 \rangle \bigg| > \sigma^2 t \bigg\} &\leq 2 \exp\bigg\{ - c \min\bigg( \frac{t^2}{L r^2}, \frac{t}{r} \bigg) \bigg\}.
\end{align*}
Therefore, taking $t = C \max\{ \sqrt{nL} r, nr\}$ it holds that 
\begin{align*}
\bigg|    \sum_l \bm{x}\t \pert \nl \U\U\t \nl \U \bm{y} \bigg| \leq C \sigma^2 r \max\{ \sqrt{nL}, n \}
\end{align*}
with probability at least $1 - 2\exp(- c n)$.  Taking a net over $\bm{x,y}$ after adjusting necessary constants completes the proof.  

The proof for \eqref{assertion4} follows from \cref{lem:initialconcentration}.
\end{proof}

\subsection{Proof of \cref{lem:Lt}} \label{sec:linearproof}
\begin{proof} First, recall that
\begin{align*}
\begin{split}
    {\sf L}_t &:= - \frac{1}{L} \big(  \bm{\hat U}_t \bm{\hat U}_t\t - \bm{UU}\t \big) \sum_l \bm{N}^{(l)} \bm{\hat U}_t \bm{\hat U}_t\t \bm{S}^{(l)} \bm{\hat U}_t + \frac{1}{L} \bm{U}_{\perp} \bm{U}_{\perp}\t  \sum_l \bm{N}^{(l)} \bm{\hat U}_t \bm{\hat U}_t\t \bm{S}^{(l)} \bm{\hat U}_t \\
            &\quad -  \bm{\hat U}_t \bm{\hat U}_t\t \frac{1}{L} \sum_{l} \bm{S}^{(l)} \bm{\hat U}_t \bm{\hat U}_t\t \bm{N}^{(l)} \bm{\hat U}_t + \frac{1}{L} \sum_l \bm{S}^{(l)} \bm{\hat U}_t \bm{\hat U}_t\t \bm{N}^{(l)} \bm{\hat U}_t 
\end{split}
\end{align*}
Therefore, 
\begin{align*}
    \bm{U}_{\perp} \bm{U}_{\perp}\t {\sf L}_t &= -  \frac{1}{L} \bm{U}_{\perp} \bm{U}_{\perp}\t \bm{\hat U}_t \bm{\hat U}_t\t \sum_{l} \bm{N}^{(l)} \bm{\hat U}_t \bm{\hat U}_t\t \bm{S}^{(l)} \bm{\hat U}_t + \frac{1}{L} \sum_{l} \bm{U}_{\perp} \bm{U}_{\perp}\t \bm{N}^{(l)} \bm{\hat U}_t\bm{\hat U}_t\t \bm{S}^{(l)} \bm{\hat U}_t \\
    &\quad -\frac{1}{L} \bm{U}_{\perp} \bm{U}_{\perp}\t \bm{\hat U}_t \bm{\hat U}_t\t \sum_{l} \bm{S}^{(l)} \bm{\hat U}_t \bm{\hat U}_t\t \bm{N}^{(l)} \bm{\hat U}_t \\
    &=: T_1 + T_2 + T_3.
\end{align*}
We bound each term in turn.
\begin{itemize}
    \item \textbf{The term $T_1$}.  The first term satisfies
    \begin{align*}
        \frac{1}{L} \| \bm{U}_{\perp} \bm{U}_{\perp}\t \bm{\hat U}_t \bm{\hat U}_t\t \sum_l \bm{N}^{(l)} \bm{\hat U}_t \bm{\hat U}_t\t \bm{S}^{(l)} \bm{\hat U}_t \|_F &\leq
       \frac{1}{L}  \| \bm{U}_{\perp} \bm{U}_{\perp}\t \bm{\hat U}_t \bm{\hat U}_t\t \sum_l \bm{N}^{(l)} \big( \bm{\hat U}_t \bm{\hat U}_t\t - \bm{UU}\t \big) \bm{S}^{(l)} \bm{\hat U}_t \|_F \\
        &\quad + \frac{1}{L}  \| \bm{U}_{\perp} \bm{U}_{\perp}\t \bm{\hat U}_t \bm{\hat U}_t\t \sum_l \bm{N}^{(l)} \bm{S}^{(l)} \bm{\hat U}_t \|_F \\
        &\leq \frac{1}{L} \| \bm{U}_{\perp} \t \bm{\hat U}_t \|_F \bigg\| \sum_{l} \bm{N}^{(l)} \big( \bm{\hat U}_t \bm{\hat U}_t\t - \bm{UU}\t \big) \bm{S}^{(l)} \bigg\| \\
        &\quad +  \frac{1}{L} \| \bm{U}_{\perp}\t \bm{\hat U}_t \|_F \| \sum_l \bm{N}^{(l)} \bm{S}^{(l)} \| \\
        &\leq \| \bm{U}_{\perp}\t \bm{\hat U}_t \|_F \bigg( \frac{C \sigma \sqrt{nr} \lambda_{\max}}{\sqrt{L}} \| \uhat_t \uhat_t\t - \U\U\t  \|_F + \frac{\sigma \sqrt{n}\lambda_{\max}}{\sqrt{L}} \bigg) \\
        &\leq \| \bm{U}_{\perp}\t \bm{\hat U}_t \|_F \bigg( \frac{C \sigma \sqrt{nr} \lambda_{\max}}{\sqrt{L}} \| \bm{U}_{\perp}\t \bm{\hat U}_t \|_F + \frac{\sigma \sqrt{n}\lambda_{\max}}{\sqrt{L}} \bigg)\numberthis \label{billybob2}
    \end{align*}
    on the event $\mathcal{E}_{{\sf good}}$ by  \eqref{assertion1} and \eqref{assertion4}, where in the final inequality we have applied Lemma 1 of \citet{cai_rate-optimal_2018}. 
    \item \textbf{The term $T_2$}.  We have that 
    \begin{align*}
        \| T_2 \|_F &= \frac{1}{L} \| \pert \sum_l \nl \uhat_t \uhat_t\t \bm{S}\l \uhat_T \|_F \\
        &\leq \frac{1}{L} \| \pert \sum_l \nl ( \uhat_t \uhat_t\t - \U \U\t ) \bm{S}\l \|_F + \frac{1}{L} \| \pert \sum_l \nl \bm{S}\l \|_F \\
        &\leq \frac{1}{L} \|  \sum_l \nl ( \uhat_t \uhat_t\t - \U \U\t ) \bm{S}\l \|_F + \frac{1}{L} \| \pert \sum_l \nl \bm{S}\l \|_F \\
        &\leq \| \uhat_t \uhat_t\t - \U \U\t \|_F \frac{C\sigma \lambda_{\max} \sqrt{nr}}{\sqrt{L}} + \frac{C\sigma \lambda_{\max} \sqrt{nr}}{\sqrt{L}} \\
      &\leq    \| \pert \uhat_t \|_F \frac{C\sigma \lambda_{\max} \sqrt{nr}}{\sqrt{L}} + \frac{C\sigma \lambda_{\max} \sqrt{nr}}{\sqrt{L}}, \numberthis \label{billybob1}
    \end{align*}
    where we have again used \eqref{assertion1} and \eqref{assertion4} and Lemma 1 of \citet{cai_rate-optimal_2018}. 
    \item \textbf{The term $T_3$}. By exactly the same argument as the term $T_1$, we have that
    \begin{align*}
        \| T_3\|_F \leq \| \pert \uhat_t \|_F \bigg( \frac{C \sigma \sqrt{nr} \lambda_{\max}}{\sqrt{L}} \|\pert \uhat_t\|_F + \frac{\sigma \sqrt{n}\lambda_{\max}}{\sqrt{L}} \bigg). \numberthis \label{billybob}
    \end{align*}
\end{itemize}
Combining \eqref{billybob2},\eqref{billybob1},  and \eqref{billybob} shows that
\begin{align*}
    \| \bm{U}_{\perp} \bm{U}_{\perp}\t {\sf L}_t \|_F &\leq \| \bm{U}_{\perp}\t \bm{\hat U}_t \|_F \bigg( \frac{\sigma \sqrt{nr}\lambda_{\max}}{\sqrt{L}} \| \bm{U}_{\perp}\t \bm{\hat U}_t\|_F + \frac{\sigma \sqrt{n}\lambda_{\max}}{\sqrt{L}} \bigg) + \frac{\sigma \sqrt{nr}\lambda_{\max}}{\sqrt{L}} + \frac{C\sigma \sqrt{nr}\lambda_{\max}}{\sqrt{L}} \| \bm{U}_{\perp}\t \bm{\hat U}_t \|_F \\
    &\leq \| \bm{U}_{\perp}\t \bm{\hat U}_t \|_F \frac{C \sigma \sqrt{nr}\lambda_{\max}}{\sqrt{L}} + \frac{C \sigma \sqrt{nr}\lambda_{\max}}{\sqrt{L}}
\end{align*}
where the final bound holds as long $\|\pert \uhat_t \|_F \leq \frac{1}{2}.$

We now bound $\| {\sf L}_t \|_F$.  Using the same arguments as the previous bound, we have
\begin{align*}
    \| {\sf L}_t \|_F &\leq \frac{1}{L} \bigg\| \big( \uhat_t \uhat_t\t - \U \U\t \big) \sum_l \nl \uhat_t\uhat_t\t \bm{S}^{(l)} \uhat_t \bigg\|_F \\
    &\quad + \frac{1}{L} \bigg\| \pert \sum_l \nl \uhat_t \uhat_t\t \sl \uhat_t \bigg\|_F \\
    &\quad + \frac{1}{L} \bigg\| \sum_l \sl \uhat_t \uhat_t\t \nl \uhat_t \bigg\|_F \\
    &\quad + \frac{1}{L} \bigg\| \sum_l \sl \uhat_t \uhat_t\t \nl \bigg\|_F \\
    &\leq\frac{1}{L} \bigg\|   \uhat_t \uhat_t\t - \U \U\t  \bigg\|_F \bigg(  \bigg\|\sum_l \nl (\uhat_t\uhat_t\t - \U \U\t ) \bm{S}^{(l)} \uhat_t \bigg\|_F +  \bigg\| \pert \sum_l \nl \sl \bigg\|_F \bigg)\\
    &\quad + \frac{1}{L} \bigg\| \pert \sum_l \nl \bigg( \uhat_t \uhat_t\t - \U \U\t \bigg) \sl  \bigg\|_F + \frac{1}{L} \bigg\| \pert \sum_l \nl \sl \bigg\|_F \\
    &\quad + \frac{1}{L} \bigg\| \sum_l \sl (\uhat_t \uhat_t\t - \U\U\t ) \nl  \bigg\|_F + \frac{1}{L} \bigg\| \pert \sum_l \nl \sl \bigg\|_F\\
    &\quad + \frac{1}{L} \bigg\| \sum_l \sl (\uhat_t \uhat_t\t - \U\U\t ) \nl \bigg\|_F + \frac{1}{L} \bigg\| \pert \sum_l \nl \sl \bigg\|_F \\
    &\leq \sqrt{2} \| \pert \uhat \|_F \bigg( \sqrt{r} \|\uhat_t\uhat_t\t - \U \U\t \|_F \frac{C\sigma \lambda_{\max} \sqrt{nr}}{\sqrt{L}} + \frac{C\sigma \lambda_{\max} \sqrt{nr}}{\sqrt{L}} \bigg) \\
    &\quad + \sqrt{r} \|\uhat_t\uhat_t\t - \U \U\t \|_F \frac{C\sigma \lambda_{\max} \sqrt{nr}}{\sqrt{L}} + \frac{C\sigma \lambda_{\max} \sqrt{nr}}{\sqrt{L}} \\
    &\leq  \| \pert \uhat \|_F\sqrt{r}  \frac{C \sigma \lambda_{\max} \sqrt{nr}}{\sqrt{L}} + \frac{C \sigma \lambda_{\max} \sqrt{nr}}{\sqrt{L}}
\end{align*}
which completes the proof. 
\end{proof}

\subsection{Proof of \cref{lem:Qt}} \label{sec:quadraticproof}
\begin{proof}
    We have that
    \begin{align}
            {\sf Q}_t &:= - \frac{1}{L}  \bm{\hat U}_t \bm{\hat U}_t\t \sum_l \bm{N}^{(l)} \bm{\hat U}_t \bm{\hat U}_t\t \bm{N}^{(l)} \bm{\hat U}_t + \frac{1}{L}  \sum_l \bm{N}^{(l)} \bm{\hat U}_t \bm{\hat U}_t\t \bm{N}^{(l)} \bm{\hat U}_t. 
    \end{align}
On the event $\mathcal{E}_{{\sf good}}$ by \eqref{assertion2} for any matrix $\bm{Q}$ it holds that
\begin{align*}
    \bigg\| \sum_l \nl \bm{Q} \nl - \frac{\sigma^2}{2} \bigg( \bm{Q} + {\sf Tr}(\bm{Q}) \bm{I} \bigg) \bigg\| \lesssim \sigma^2 nr \sqrt{L} \| \bm{Q} \|_F.
\end{align*}
By adding and subtracting appropriately, we obtain that 
\begin{align*}
\bm{U}_{\perp} \bm{U}_{\perp}\t {\sf Q}_t    &=  \underbrace{- \frac{1}{L} \bm{U}_{\perp} \bm{U}_{\perp}\t \bm{\hat U}_t\bm{\hat U}_t\t \sum_l \bm{N}^{(l)} (\bm{\hat U}_t \bm{\hat U}_t\t - \bm{UU}\t) \bm{N}^{(l)} \bm{\hat U}_t + \frac{\sigma^2}{2}  \bm{U}_{\perp} \bm{U}_{\perp}\t \bm{\hat U}_t \bm{\hat U}_t\t \bigg( \bm{\hat U}_t \bm{\hat U}_t\t - \bm{UU}\t \bigg) \bm{\hat U}_t}_{=:T_1} \\
    & \underbrace{- 
     \frac{1}{L} \bm{U}_{\perp} \bm{U}_{\perp}\t \bm{\hat U}_t\bm{\hat U}_t\t \sum_l \bm{N}^{(l)}  \bm{UU}\t \bm{N}^{(l)} \bm{\hat U}_t   +  \frac{\sigma^2}{2}   \bm{U}_{\perp} \bm{U}_{\perp}\t \bm{\hat U}_t \bm{\hat U}_t\t \bigg(\bm{UU}\t + r \bm{I}_r \bigg)  \bm{\hat U}_t }_{=:T_2}
    \\&+ \underbrace{\frac{1}{L} \sum_l \bm{U}_{\perp} \bm{U}_{\perp}\t \bm{N}^{(l)} (\bm{\hat U}_t \bm{\hat U}_t\t - \bm{UU}\t ) \bm{N}^{(l)} \bm{\hat U}_t -  \frac{\sigma^2}{2}   \bm{U}_{\perp} \bm{U}_{\perp}\t \bm{\hat U}_t}_{=: T_3} \\
    &+ \underbrace{\frac{1}{L}\sum_l \bm{U}_{\perp} \bm{U}_{\perp}\t \bm{N}^{(l)} \bm{UU}\t \bm{N}^{(l)} \bm{U}_{\perp} \bm{U}_{\perp}\t \bm{\hat U}_t - \frac{\sigma^2}{2}  r \bm{U}_{\perp} \bm{U}_{\perp}\t  \bm{\hat U}_t}_{=:T_4}  \\
    &+  \underbrace{\frac{1}{L}\sum_l \bm{U}_{\perp} \bm{U}_{\perp}\t \bm{N}^{(l)} \bm{UU}\t \bm{N}^{(l)} \bm{U}\bm{U}\t \bm{\hat U}_t}_{=:T_5}.
\end{align*}
We bound each term in turn. 
\begin{itemize}
    \item \textbf{The term $T_1$:} By \eqref{assertion2} since ${\sf Tr}(\uhat\uhat\t) = {\sf Tr}(\U\U\t) = r$,  it holds that 
    \begin{align*}
         \| T_1 \|_F &= \frac{1}{L} \bigg\| \bm{U}_{\perp} \bm{U}_{\perp}\t \bm{\hat U}_t\bm{\hat U}_t\t \bigg(  \sum_l \bm{N}^{(l)} (\bm{\hat U}_t \bm{\hat U}_t\t - \bm{UU}\t) \bm{N}^{(l)} - \frac{\sigma^2}{2}  \big( \bm{\hat U}_t \bm{\hat U}_t\t - \bm{UU}\t \big) \bigg) \bm{\hat U}_t  \bigg\|_F \\
        &\leq \frac{1}{L} \| \bm{U}_{\perp}\t \bm{\hat U}_t \|_F \bigg\| \sum_l \bm{N}^{(l)} ( \bm{\hat U}_t \bm{\hat U}_t\t - \bm{UU}\t ) \bm{N}^{(l)}  - \frac{\sigma^2}{2} ( \bm{\hat U}_t \bm{\hat U}_t\t - \bm{UU}\t ) \bigg\| \\
        &\leq \| \bm{U}_{\perp}\t \bm{\hat U}_t \|_F \| \bm{\hat U}_t \bm{\hat U}_t\t - \bm{UU}\t \|_F  \frac{C \sigma^2 nr \sqrt{L}}{L} \\
        &\leq  \| \bm{U}_{\perp}\t \bm{\hat U}_t \|_F^2 C \frac{\sigma^2 nr}{\sqrt{L}},
    \end{align*}
    where the final inequality holds by Lemma 1 of \citet{cai_rate-optimal_2018}.
    \item \textbf{The term $T_2$:} By a similar argument, it holds that 
    \begin{align*}
        \| T_2 \|_F &=   \frac{1}{L}\bigg\| 
    \bm{U}_{\perp} \bm{U}_{\perp}\t \bm{\hat U}_t\bm{\hat U}_t\t \bigg(  \sum_l \bm{N}^{(l)}  \bm{UU}\t \bm{N}^{(l)} \bm{\hat U}_t   - \frac{\sigma^2}{2}  \big(  \bm{UU}\t + r \bm{I}_r \big) \bigg)  \bm{\hat U}_t \bigg\|_F \\
     &\leq \| \bm{U}_{\perp}\t \bm{\hat U}_t \|_F  \| \bm{UU}\t  \|_F \frac{C \sigma^2 nr }{\sqrt{L}}\\
     &\leq \| \bm{U}_{\perp}\t \bm{\hat U}_t \|_F C \sigma^2 \frac{n r^{3/2}}{\sqrt{L}} .
    \end{align*}
    \item \textbf{The term $T_3$:} Following a similar argument, 
    \begin{align*}
        \| T_3 \|_F &\leq \| \uhat_t \uhat_t\t - \U\U\t \|_F \frac{C \sigma^2 n r}{\sqrt{L}} \\
        &\leq \| \pert \uhat \|_F \frac{C \sigma^2 n r}{\sqrt{L}}.
    \end{align*}
    \item \textbf{The term $T_4$.} We have that
    \begin{align*}
        \| T_4 \|_F &\leq \frac{1}{L} \| \pert \uhat_t \|_F \bigg\| \pert \big( \sum_l \nl \U\U\t \nl - \frac{\sigma^2 r}{2} \big) \per \bigg\| \\
        &\leq \|\pert \uhat_t\|_F \| \U\U\t \|_F \frac{C \sigma^2 n r}{\sqrt{L}} \\
                 &\leq \frac{C \sigma^2 n r^{3/2}}{\sqrt{L}}  \|\pert\uhat_t \|_F.
    \end{align*}
    \item \textbf{The term $T_5$.}  We now instead invoke \eqref{assertion3} to note that 
    \begin{align*}
        \| \frac{1}{L} \sum_l \bm{U}_{\perp} \bm{U}_{\perp}\t \bm{N}^{(l)} \bm{U} \bm{U}\t \bm{N}^{(l)} \bm{UU}\t \bm{\hat U}_t \|_F \leq C \frac{\sigma^2 r^{3/2} \max\{ \sqrt{nL}, n \}}{L}.
    \end{align*}
\end{itemize}
Combining these inequalities shows that 
\begin{align*}
    \| \per \pert {\sf Q}_t \|_F &\leq \| \pert \uhat_t \|_F^2 \frac{C \sigma^2 nr}{\sqrt{L}} + \| \pert \uhat_t\|_F \frac{C \sigma^2 n r^{3/2}}{\sqrt{L}} + \frac{C \sigma^2 r^{3/2} \max\{ \sqrt{nL},n\}}{L} \\
    &\leq \| \pert \uhat_t\|_F \frac{C \sigma^2 n r^{3/2}}{\sqrt{L}} + \frac{C \sigma^2 r^{3/2} \max\{ \sqrt{nL},n\}}{L},
\end{align*}
where the last inequality holds whenever $\|\pert\uhat_t \| \leq \frac{1}{2}$, which is assumed.

We now bound $\|{\sf Q}_t\|_F$.  We can rewrite ${\sf Q}_t$ via 
\begin{align*}
      {\sf Q}_t &:= \frac{1}{L} \bigg( \bm{I} - \uhat_t \uhat_t\t \bigg)  \sum_l \bm{N}^{(l)} \bm{\hat U}_t \bm{\hat U}_t\t \bm{N}^{(l)} \bm{\hat U}_t  \\
      &= \frac{1}{L} \bigg( \bm{I} - \uhat_t \uhat_t\t \bigg) \bigg(  \sum_l \bm{N}^{(l)} \bm{\hat U}_t \bm{\hat U}_t\t \bm{N}^{(l)} \bm{\hat U}_t  - \frac{\sigma^2}{2} \big( \uhat_t \uhat_t\t + r \bm{I}_r\big) \bigg) \uhat_t,
\end{align*}
since $(\bm{I} - \uhat_t\uhat_t\t ) \uhat_t = 0$.  Therefore, appealing again to \eqref{assertion2},
\begin{align*}
    \| {\sf Q}_t \|_F &\leq \frac{\sqrt{r}}{L} \bigg\| \bigg( \bm{I} - \uhat_t \uhat_t\t \bigg) \bigg(  \sum_l \bm{N}^{(l)} \bm{\hat U}_t \bm{\hat U}_t\t \bm{N}^{(l)} \bm{\hat U}_t  - \frac{\sigma^2}{2} \big( \uhat_t \uhat_t\t + r \bm{I}_r\big) \bigg) \uhat_t \bigg\| \\
    &\leq \frac{\sqrt{r}}{L} \bigg\|   \sum_l \bm{N}^{(l)} \bm{\hat U}_t \bm{\hat U}_t\t \bm{N}^{(l)} \bm{\hat U}_t  - \frac{\sigma^2}{2} \big( \uhat_t \uhat_t\t + r \bm{I}_r\big) \bigg\| \\
    &\leq \frac{C \sigma^2 n r^{3/2}}{\sqrt{L}}  \|\uhat_t \uhat_t\t \|_F  \\
    &= \frac{C \sigma^2 n r^2}{\sqrt{L}}
\end{align*}
as required.
\end{proof}

\subsection{Proof of \cref{lem:contractionlemma}} \label{sec:contractionproof}
\begin{proof}[Proof of \cref{lem:contractionlemma}]
The upper bound is immediate.  For the lower bound, we first note that
\begin{align*}
    \lambda_r (\uhat_t\t \U \U\t\uhat_t ) = \lambda_r( \cos\Theta(\uhat_t,\U))^2 \geq 1 - \| \sin\Theta(\uhat_t,\U) \|_F^2 \geq \frac{3}{4}
\end{align*}
since $ \| \sin\Theta(\uhat_t,\U) \|_F^2  \leq \frac{1}{4}$ by assumption.  Consequently, with respect to the positive semidefinite ordering, it holds that 
\begin{align*}
     \uhat\t \bm{S}^{(l)} \uhat_t \uhat_t\t \bm{S}^{(l)} \uhat_t   
     &=  \uhat_t\t \U \rl \U\t\uhat_t\uhat_t\t \U \rl \U\t\uhat_t
     \succeq \frac{3}{4}  \uhat_t\t\U (\rl)^2 \U\t \uhat_t.
\end{align*}
Thus,
\begin{align*}
    \sum_l   \uhat_t\t \bm{S}^{(l)} \uhat_t \uhat_t\t \bm{S}^{(l)} \uhat_t  \succeq \frac{3}{4}  \uhat_t\t\U  \bigg( \sum_l (\rl)^2 \bigg) \U\t \uhat_t.
\end{align*}
Therefore, it holds that
\begin{align*}
    \lambda_r\bigg( \sum_l \uhat_t\t \sl \uhat_t\uhat_t\t \sl \uhat_t \bigg) &\geq \frac{3}{4} \lambda_r\bigg[ \uhat_t\t \U \bigg( \sum_l (\rl)^2 \bigg) \U\t \uhat_t \bigg] \\
    &\geq \frac{3}{4} \lambda_r( \uhat_t\t \U \U\t \uhat_t ) \lambda_r \bigg( \sum_l (\rl)^2 \bigg) \\
    &\geq \frac{9}{16} L \lambda^2 \\
    &\geq \frac{L \lambda^2}{4}
\end{align*}
which completes the proof.
\end{proof}

\section{Proofs for Asymptotic Normality and Inference} \label{sec:limitheorem}
In this section we prove our main asymptotic results; namely \cref{thm:normality,thm:minimaxoptimalconfidenceinterval}, both of which build on \cref{thm:mainthm}. In order to prove these results we require the following theorem, which isolates the leading-order term in the asymptotic expansion.
\begin{theorem} \label{thm:asymptotics}
Suppose that the conditions of \cref{thm:mainthm} and \cref{thm:normality} hold.  Then we have that     \begin{align*}
        \| \pert \uhat \|_F^2 &= \| \sum_l \pert \nl \U \rl \R \|_F^2 + {\sf Res},
    \end{align*}
    where ${\sf Res}$ satisfies the bound
\begin{align*}
    {\sf Res} \lesssim \frac{\sigma^3 \kappa^4 r^{5/2} n^{3/2}}{\lambda^3 L^{3/2}} + \frac{\sigma^4 \kappa^2 n^2 r^{7/2}}{\lambda^4 L^{3/2}}
\end{align*}
with probability at least $1- \exp( - cn) - \exp( - c \sqrt{L})$.
 \end{theorem}
\begin{proof}
    See \cref{sec:asymptoticsproof}.
\end{proof}
 
\subsection{Proof of  \cref{thm:asymptotics}} \label{sec:asymptoticsproof}

\begin{proof}[Proof of \cref{thm:asymptotics}]
Under the conditions of the theorem, with probability at least $1 - \exp( - c n)$ we have that
\begin{align}
    \| \sin\Theta(\bm{\hat U}_{t} ,\bm{U}) \|_F \lesssim \frac{\sigma \kappa \sqrt{nr}}{\lambda \sqrt{L}}.   \label{yadayada}
\end{align}
We use this bound without additional remarks in the subsequent analysis.  Throughout the proof we will suppress the dependence of $\uhat$ on the time index $t$, and assume only that $t$ is taken such that \eqref{yadayada} holds. We also note that it is immediate that 
\begin{align*}
    \| \R \| \leq \frac{1}{L\lambda^2},
\end{align*}
a fact we will use repeatedly.  
In the subsequent analysis we will define the orthogonal matrix $$\mathcal{O} = \argmin_{\mathcal{OO}\t= \bm{I}_r} \| \uhat \mathcal{O} - \bm{U} \|_F;$$ the Frobenius-optimal orthogonal matrix aligning $\uhat$ and $\U$.  
Finally, by \cref{yadayada} it holds that
\begin{align*}
    \| \uhat \mathcal{O} - \U \|_F \lesssim \frac{\sigma \kappa \sqrt{nr}}{\lambda \sqrt{L}}; \qquad \| \uhat \uhat\t - \U \U\t \|_F \lesssim \frac{\sigma \kappa \sqrt{nr}}{\lambda \sqrt{L}},
\end{align*}
which follows from Lemma 1 of \citet{cai_rate-optimal_2018}.  
\begin{itemize}
\item 
\textbf{Step 1: Initial Expansion}.   First, it holds that
\begin{align*}
    \sum_l \bm{S}^{(l)} \bm{U} \bm{R}^{(l)}\R = \bm{U}
\end{align*}
and also that  $\bm{\hat R}^{(l)} = \uhat\t \bm{A}\l \uhat$.  Thus, using the fact that $\al = \sl + \nl$, it holds that 
\begin{align*}
    \bm{\hat U} \mathcal{O} &= \bm{U} + \sum_{l} \nl \U \rl \R  \\
    &\quad + \sum_l \bm{N}^{(l)} \bm{U} \bm{R}^{(l)}( \mathcal{O}\t \rhatinv \mathcal{O} - \R) + \sum_l \bm{S}^{(l)} \bm{U} \bm{R}^{(l)}( \mathcal{O}\t \rhatinv \mathcal{O} - \R) \\
    &\quad + \sum_l \bm{N}^{(l)} \big( \bm{\hat U} \bm{\hat R}^{(l)} \mathcal{O} - \bm{U} \bm{R}^{(l)} \big) \mathcal{O}\t \rhatinv \mathcal{O}   + \sum_{l} \bm{S}^{(l)} \big( \bm{\hat U} \bm{\hat R}^{(l)} \mathcal{O} - \bm{U R}^{(l)} \big) \mathcal{O}\t \rhatinv \mathcal{O} \\
    &\quad + 
    \bigg(\uhat - \sum_{l} \bm{A}^{(l)} \uhat \rhatl \rhatinv  \bigg) \mathcal{O}. \numberthis \label{tobound}
\end{align*}
We therefore expand $\pert \uhat \mathcal{O}$ via 
\begin{align*}
    \bm{U}_{\perp}\t \bm{\hat U} \mathcal{O} &=
\sum_l \pert \nl \U \rl \R + \sum_l \pert \nl \big( \uhat \uhat\t \bm{A}\l \uhat \mathcal{O} - \U\rl \big) \mathcal{O}\t \rhatinv \mathcal{O} \\ 
&\quad + \sum_{l} \pert \nl \U \rl \big( \mathcal{O}\t \rhatinv \mathcal{O} - \R \big) + \pert \bigg( \uhat - \sum_{l}\bm{A}\l \uhat \rhatl \rhatinv \bigg) \mathcal{O}
\\
 &= \sum_l \bm{U}_{\perp}\t \bm{N}^{(l)} \bm{U} \bm{R}^{(l)} \R + \sum_l \pert \nl \U\U\t \nl \U  \R \\
 &\quad + \sum_l \pert \nl \U\U\t \nl \U  \big( \O\t \rhatinv \O - \R \big) \\
 &\quad + \sum_l \pert\nl (\uhat \uhat\t - \U\U\t) \bm{S}^{(l)} \uhat  \O \O\t \rhatinv \O \\
 &\quad +  \sum_l \pert\nl  \bm{S}^{(l)} (\uhat  \O - \U) \O\t \rhatinv \O \\
    &\quad + \sum_l \bigg( \pert \nl \uhat \uhat\t \bm{N}^{(l)} \uhat  \O \O\t \rhatinv \O - \pert \nl \U\U\t \nl \U \O\t \rhatinv \O\bigg) \\
    &\quad + \sum_{l} \pert \nl \U \rl \big( \mathcal{O}\t \rhatinv \mathcal{O} - \R \big) + \pert \bigg( \uhat - \sum_{l}\bm{A}\l \uhat \rhatl \rhatinv \bigg) \mathcal{O}.
\numberthis\label{callousdaoboys}
\end{align*}
Define
\begin{align*}
   \bm{G}^{(1)} &:= \sum_l \pert \nl \U \rl \R; \\
   \bm{G}^{(2)} &:= \sum_l \pert \nl \U\U\t \nl \U\U\t \R; \\
   T_1 &:= \sum_l \pert\nl (\uhat \uhat\t - \U\U\t) \bm{S}^{(l)} \uhat  \O \mathcal{O}\t \rhatinv \mathcal{O} +  \sum_l \pert\nl  \bm{S}^{(l)} (\uhat  \O - \U) \mathcal{O}\t \rhatinv \mathcal{O}; \\
    T_2 &:= \sum_{l} \pert \nl \U\U\t \nl \U \bigg(\mathcal{O}\t \rhatinv \mathcal{O} - \R\bigg) + \sum_{l} \pert \nl \U \rl \bigg( \mathcal{O}\t \rhatinv \mathcal{O} - \R \bigg); \\
    T_3 &:= \sum_l \pert\nl \uhat \uhat\t \bm{N}^{(l)} \uhat  \O  \mathcal{O}\t \rhatinv \mathcal{O} - \sum_l \pert \nl \U\U\t \nl \U\U\t \mathcal{O}\t \rhatinv \mathcal{O}; \\
    T_4 &:= \pert \bigg( \uhat - \sum_{l} \bm{A}\l \uhat \rhatl \rhatinv \bigg) \mathcal{O}.
\end{align*}
We have thus shown that
\begin{align*}
    \pert\uhat \mathcal{O} &= \bm{G}^{(1)} + \bm{G}^{(2)} + T_1 + T_2 + T_3 + T_4. \numberthis \label{initialexpansion}
\end{align*}
In the next steps we will bound the residual quantities.  
\item 
\textbf{Step 2: Bounding $T_1$ through $T_3$.}
First, in order to bound $T_1$ through $T_3$, we first require the following bound controlling the approximation of $\rhatinv$ to $\R$.
\begin{lemma} \label{lem:rlsquared}
Suppose the conditions of \cref{thm:asymptotics} hold. Let $\mathcal{O}= \mathcal{O}_{\bm{U},\bm{\hat U}}$.  Then   with probability at least $1 - \exp( - c n)$, \begin{align*}
      \bigg\| \sum_l \bigg[ \big( \bm{\hat R}^{(l)} \big)^2 - \O\t \big( \rl \big)^2\O \bigg] \bigg\|_F &\lesssim L \lambda_{\max}^2 \err + \err \sigma^2 n r^{3/2} \sqrt{L} + \sigma^2 L r^{3/2} \\
    &\quad + \sigma^2 n r^{3/2} + \sigma^2 r^{3/2} \sqrt{nL} + \sigma r \sqrt{n} \lambda_{\max} \err + \sigma \sqrt{nLr}\lambda_{\max};  \numberthis \label{102825} \\
\bigg\| \rhatinv - \mathcal{O}\t \R \mathcal{O} \bigg\|_F &\lesssim  \frac{1}{\lambda^2 L} \frac{ \sigma \kappa^3 \sqrt{nr}}{\lambda \sqrt{L}} \numberthis \label{rhatinvbd}
    \end{align*}
\end{lemma}
\begin{proof}
    See \cref{sec:rlsquaredproof}.
\end{proof}
In particular, when $\lambda/\sigma \gg n r^2 \kappa^4/\sqrt{L}$, on the event above, it holds that
\begin{align*}
    \| \rhatinv \| \lesssim \frac{1}{L \lambda^2},\numberthis \label{rhatinvbd2}
\end{align*}
a fact that we will use repeatedly.

First, we note that on $\mathcal{E}_{{\sf good}}$, by \eqref{assertion1}, 
\begin{align*}
 \bigg\|    \sum_l \bm{U}_{\perp}\t \bm{N}^{(l)} \big( \bm{\hat U} \bm{\hat U}\t - \bm{UU}\t \big) \bm{S}^{(l)} \bm{\hat U} \mathcal{O} \O\t \rhatinv \O  \bigg\|_F  &\leq \big\|     \sum_l \bm{U}_{\perp}\t \bm{N}^{(l)} \big( \bm{\hat U} \bm{\hat U}\t - \bm{UU}\t \big) \bm{S}^{(l)} \big\| \big\|\rhatinv \big\|_F \\
&\lesssim \frac{\sigma^2 \kappa^2 r^{3/2} n}{L \lambda^2}. \numberthis \label{bigboy1}
\end{align*}
Similarly, on the event $\mathcal{E}_{{\sf good}}$, by \cref{assertion4} it holds that 
  \begin{align*}
  \| \sum_l \bm{U}_{\perp}\t \bm{N}^{(l)} \bm{S}^{(l)} \big( \bm{\hat U} \mathcal{O} - \bm{U}\big)\O\t \rhatinv \O \|_F 
      &\lesssim \sigma \sqrt{nL} \lambda_{\max} \frac{\sigma \sqrt{nr}\kappa}{\lambda \sqrt{L}} \frac{1}{\lambda^2 L} \\
      &\asymp \frac{\sigma^2 n \sqrt{r} \kappa^2}{\lambda^2 L} \numberthis \label{bigboy2}
  \end{align*}
Thus, we have shown that 
\begin{align*}
    \|T_1\|_F  \lesssim \frac{\sigma^2 \kappa^2 r^{3/2} n}{L \lambda^2}. \numberthis \label{t1t2}
\end{align*}
Next, we bound the term $T_2$.  First, we have that by \eqref{assertion3} and \cref{lem:rlsquared},
\begin{align*}
    \| \sum_{l} \pert \nl &\U\U\t \nl \U \big( \O\t \rhatinv \O - \R \big) \|_F  \\ &\lesssim \| \sum_{l} \pert \nl \U \U\t \nl \U \| \| \O\t \rhatinv \O - \R \|_F \\
    &\lesssim \sigma^2 r n \frac{1}{\lambda^2 L} \frac{\sigma \kappa^3 \sqrt{nr}}{\lambda \sqrt{L}} \asymp \frac{\sigma^3 \kappa^3 r^{3/2} n^{3/2}}{L^{3/2}}.
\end{align*}
Similarly, using \eqref{assertion4} instead,
\begin{align*}
    \bigg\| \sum_{l} \pert \nl \U \rl \big( \O\t \rhatinv \O - \R \big) \bigg\| \lesssim   \frac{\sigma \sqrt{nL} \lambda_{\max}}{\lambda^2 L} \frac{\sigma \kappa^3 \sqrt{nr}}{\lambda \sqrt{L}} \asymp \frac{\sigma^2 \kappa^4 n \sqrt{r}}{\lambda^2 L}.
\end{align*}
Combining these bounds, with probability at least $1 - \exp( - cn)$, 
\begin{align*}
    \| T_2 \|_F \lesssim \frac{\sigma^2 \kappa^4 n \sqrt{r}}{\lambda^2 L}. \numberthis \label{t2}
\end{align*}
Analyzing $T_3$ is slightly more complicated and requires additional expansion, so this analysis is deferred to the following lemma. 
\begin{lemma} \label{lem:complicatedresidual}
Under the conditions of \cref{thm:asymptotics},
on the event $\mathcal{E}_{{\sf good}}$ it holds that 
    \begin{align*}
      \| T_3\|_F  &\lesssim   \frac{\sigma^3 n^{3/2} r^2 \kappa}{\lambda^3 L}. 
  \end{align*}
\end{lemma}
\begin{proof}
    See \cref{sec:complicatedresidualproof}.
\end{proof}
\noindent
Thus, we have shown that with high probability, 
\begin{align*}
    \| T_1 + T_2 + T_3 \|_F \lesssim \frac{\sigma^3 n^{3/2} r^2 \kappa}{\lambda^3 L} + \frac{\sigma^2 \kappa^4 n r^{3/2}}{\lambda^2 L}.
\end{align*}
\item \textbf{Step 3: Bounding $T_4$.}
We now turn to the quantity $T_4$.  This quantity is defined via
\begin{align*}
    T_4 &:= \pert \bigg( \uhat - \sum_{l} \al \uhat \rhatl \rhatinv \bigg) \mathcal{O} \\
    &= \pert \bigg( \uhat \bm{\mathcal{\hat R\hat R}}\t - \sum_{l} \al \uhat \rhatl \bigg) \rhatinv \mathcal{O} \\
    &= - \pert \big( \bm{I} - \uhat \uhat\t \big) \sum_{l} \al \uhat \rhatl \rhatinv \mathcal{O} \\
    &= - L \pert \nabla_{{\sf Riemannian}}h(\bm{\hat U}) \rhatinv \mathcal{O}. \numberthis \label{t4def}
\end{align*}
Here $\nabla_{{\sf Riemannian}}$ is understood as the Riemannian gradient of $h$ with respect to the Grassmann manifold.  We will show that on a good event, this gradient is small deterministically.  As a first step we demonstrate a form of geodesic strong convexity and smoothness of $h$.
\begin{lemma} \label{lem:geodesicconvexity}
Under the conditions of \cref{thm:asymptotics}, there exists an event $\mathcal{E}_{\nabla^2 h}$ satisfying $\mathbb{P}(\mathcal{E}_{\nabla^2 h}) \geq 1 - \exp( - c nr)$ such that 
\begin{align*}
3 \lambda_{\max}^2 \| \X \|_F^2 \geq 
   \nabla^2_{{\sf Riemannian}} h(\bm{W})[\X,\X] \geq \frac{\lambda^2}{4} \| \X \|_F^2
\end{align*}
uniformly over all $\W \in \mathbb{R}^{n\times r}$ satisfying $\W\t\W = \bm{I}_r$, $\| \sin\Theta(\W,\U)\|_F \leq \frac{C_0 \sigma \kappa \sqrt{nr}}{\lambda \sqrt{L}}$, and $\X \in \mathbb{R}^{n\times r} $ satisfying $\X\t \W = 0$. Here $\nabla^2_{{\sf Riemannian}}$ is the Riemannian Hessian of $h$ defined on the Grassmann Manifold $\bm{W}\t \bm{W} = \bm{I}_r$.
\end{lemma}
\begin{proof}
    See \cref{sec:geodesicconvexity}. 
\end{proof}
By taking $C_0 > 100C_1$ in \cref{thm:mainthm}, for all iterations $t$ larger than $C \log\big( \frac{\lambda/\sigma \sqrt{L}}{C_1 \kappa \sqrt{nr}}\big)$ it holds that every iterate lies strictly in the interior of the set $\{ \bm{W}: \|\sin\Theta(\W,\U)\|_F \leq \frac{C_0 \sigma \kappa \sqrt{nr}}{\lambda \sqrt{L}}\}$. In addition, \cref{alg} can be written as 
\begin{align*}
    \bm{\hat U}_{t+1} = {\sf SVD}_r\bigg( \bm{\hat U}_t - \eta \nabla_{{\sf Riemannian}} h( \bm{\hat U}_t) \bigg),
\end{align*}
so that ${\sf SVD}_r$ is a retraction of the iterate onto the Grassmann manifold.  Thus, \cref{lem:geodesicconvexity} implies that $h$ is smooth along the retraction curve \begin{align*}
    s \mapsto h\big( {\sf SVD}_r\big( \bm{\hat U}_t - s \nabla_{{\sf Riemannian}} h(\bm{\hat U}_t) \big), 
\end{align*}
and thus
\begin{align*}
    h(\bm{\hat U}_{t+1} ) &\leq h(\bm{\hat U}_t) + \langle \nabla_{{\sf Riemannian}} h(\bm{\hat U}_t), - \eta \nabla_{{\sf Riemannian}} h(\bm{\hat U}_t) \rangle + \frac{3 \lambda_{\max}^2}{2} \eta^2 \| \nabla_{{\sf Riemannian}}h(\bm{\hat U}_t) \|_F^2 \\
    &\leq h(\bm{\hat U}_t) - \eta \bigg(1 - \frac{3 \lambda_{\max}^2 \eta}{2} \bigg) \| \nabla_{{\sf Riemannian}} h(\bm{\hat U}_t) \|_F^2. \numberthis \label{contraction}
\end{align*}
The coefficient $\eta \big(1 - \frac{3 \lambda_{\max}^2 \eta}{2} \big)$ is strictly positive for $\eta$ chosen as in \cref{thm:mainthm}.  Letting $t_0$ be the smallest value $t$ such that $\|\sin\Theta(\bm{\hat U}_{t_0},\U)\|_F \leq \frac{2C_1 \sigma \kappa \sqrt{nr}}{\lambda \sqrt{L}}$ and applying the above result up to time $T$ yields
\begin{align*}
  \eta \bigg(1 - \frac{3 \lambda_{\max}^2 \eta}{2} \bigg)  \sum_{t=t_0}^{T-1}\| \nabla_{{\sf Riemannian}} h(\bm{\hat U}_t) \|_F^2 \leq h(\bm{\hat U}_{t_0}) - h(\bm{\hat U}_T) \leq h(\bm{\hat U}_{t_0}) < \infty.
\end{align*}
Letting $T \to \infty$ shows that $\| \nabla_{{\sf Riemannian}} h(\bm{\hat U}_t) \| \to 0$.  Therefore, there is at least one convergent subsequence by compactness of the Grassmann manifold.  Let this limit point be denoted as $\bm{\hat U}_{\infty}$.  Since the Riemannian gradient is continuous, we must have that $\nabla_{{\sf Riemannian}} h(\bm{\hat U}_{\infty}) = 0$.  Thus, every limit point of the iterates must be stationary.  

Next, note that $\frac{C_0 \sigma \kappa \sqrt{nr}}{\lambda \sqrt{L}} \ll 1$ under the conditions of \cref{thm:asymptotics}.  Thus, there is a geodesically convex ball $\mathcal{B}(\U)$ satisfying
\begin{align*}
    \big\{ \bm{W}: \|\sin\Theta(\W,\U) \|_F \leq \frac{2C_1  \sigma \kappa\sqrt{nr}}{\lambda \sqrt{L}} \big\} \subset \mathcal{B}(\U) \subset  \big\{ \bm{W}: \|\sin\Theta(\W,\U) \|_F \leq \frac{ C_0  \sigma \kappa\sqrt{nr}}{\lambda \sqrt{L}} \big\}.
\end{align*}
If necessary we may increase the value of $C_0$.  Then every geodesic lies within $\mathcal{B}(\U)$.  Hence, when the event in \cref{lem:geodesicconvexity} holds, the function $h$ is geodesically strongly convex and smooth within $\mathcal{B}(\U)$.  Thus, by Corollary 11.22 of \citet{boumal_introduction_2023}, every stationary point is a global minimizer within this set.

Now, suppose that $\bm{\hat U}_{\infty}$ is a local minimizer of $h$ within the ball of radius $\frac{2 C_1 \sigma \kappa \sqrt{nr}}{\lambda \sqrt{L}}$.  By Lemma 11.28 of \citet{boumal_introduction_2023}, it holds that
\begin{align*}
    h(\bm{\hat U}_t) - h(\bm{\hat U}_{\infty}) 
    &\leq \frac{2}{\lambda^2} \| \nabla_{{\sf Riemannian}} h(\bm{\hat U}_t )\|_F^2.
\end{align*}
Thus, plugging this into \eqref{contraction}, we have
\begin{align*}
    h(\bm{\hat U}_{t+1}) - h(\bm{\hat U}_{\infty}) &\leq h(\bm{\hat U}_t) - h(\bm{\hat U}_{\infty}) - \frac{\lambda^2 \eta}{2} \bigg( 1 - \frac{3 \lambda_{\max}^2 \eta}{2} \bigg) \big[h(\bm{\hat U}_t) - h(\bm{\hat U}_{\infty}) \big] \\
    &= \bigg( 1 - \frac{\lambda^2 \eta}{2} + \frac{3 \lambda_{\max}^2 \lambda^2 \eta^2}{4} \bigg) \big[h(\bm{\hat U}_t) - h(\bm{\hat U}_{\infty}) \big] \\
    &\leq \bigg( 1 - \frac{\lambda^2 \eta}{4}\bigg) \big[h(\bm{\hat U}_t) - h(\bm{\hat U}_{\infty}) \big],
\end{align*}
 where we recall that $\eta \leq c_{\eta}/\lambda_{\max}^2$, and provided that $c_{\eta}$ is sufficiently small.  
Next,  \cref{lem:geodesicconvexity} and Corollary 10.47 of \citet{boumal_introduction_2023} implies that $\nabla_{{\sf Riemannian}} h( \cdot)$ is $3 \lambda_{\max}^2$-Lipschitz, and hence
\begin{align*}
    \| \nabla_{{\sf Riemannian}}h(\bm{\hat U}_t) \|_F^2 &= \| \nabla_{{\sf Riemannian}} h(\bm{\hat U}_t) - \nabla_{{\sf Riemannian}} h(\bm{\hat U}_{\infty}) \|_F^2 \\
    &\leq 9 \lambda_{\max}^4 {\sf dist}(\bm{\hat U}_t, \bm{\hat U}_{\infty})^2. \\
    &\leq  \frac{18\lambda_{\max}^4  }{\lambda^2} (h (\bm{\hat U}_t) - h(\bm{\hat U}_{\infty})) \\
    &\leq 18\lambda_{\max}^2 \kappa^2 \bigg( 1 - \frac{\lambda^2 \eta}{4}\bigg)^{t- t_0} \frac{ C\kappa \sigma \sqrt{nr}}{\lambda \sqrt{L}},
\end{align*}
where the penultimate inequality is due to restricted strong convexity in \cref{lem:geodesicconvexity} (see the argument leading to Equation 11.19 in \citet{boumal_introduction_2023}).  
Thus, plugging in this bound to the expression \eqref{t4def}, we obtain
\begin{align*}
    \| T_4\|_F &\leq 18 L \lambda_{\max}^2 \kappa^2 \bigg( 1 - \frac{\lambda^2 \eta}{8}\bigg)^{t- t_0} \frac{ C\kappa \sigma \sqrt{nr}}{\lambda \sqrt{L}} \| \rhatinv \| \\
    &\lesssim \kappa^4 \bigg( 1 - \frac{\lambda^2 \eta}{4}\bigg)^{t- t_0}.
\end{align*}
Thus, when $t - t_0 \geq \log\bigg( \frac{C \sigma^2 \kappa^2 n r}{\lambda^2 L} \bigg( 1 - \frac{\lambda^2 \eta}{4}\bigg)\inv \bigg)$ (which is guaranteed by the conditions of \cref{thm:normality}) we have that
\begin{align*}
    \| T_4 \|_F\lesssim \frac{\sigma^2 \kappa^2 n r}{\lambda^2 L}. \numberthis\label{t4}
\end{align*}
\item 
\textbf{Step 4: Expanding the square and controlling small-order terms.}
From \eqref{initialexpansion},
\begin{align*}
       \| \pert \uhat \|_F^2 &= \| \bm{G}^{(1)} \|_F^2 + \| \bm{G}^{(2)} \|_F^2  + 2 \langle \bm{G}^{(1)}, \bm{G}^{(2)} \rangle + 2 \sum_{i=1}^{4} \langle \bm{G}^{(1)}, T_i \rangle + 2 \sum_{i=1}^{4} \langle \bm{G}^{(2)}, T_i \rangle  + \sum_{i=1}^{4} \| T_i \|_F^2.
   \end{align*}
We will bound all but the first term. 
\begin{itemize}
 \item \textbf{Bounding $\| \bm{G}^{(2)} \|_F^2$.}
 \cref{lem:goodevent} and \eqref{assertion3} show that  on $\mathcal{E}_{{\sf good}}$, 
\begin{align*}
    \| \bm{G}^{(2)} \|_F &\leq \sqrt{r} \| \R \| \| \sum_l \pert \nl \U \U\t \nl \U \| \\
    &\lesssim \frac{\sqrt{r}}{\lambda^2 L} \sigma^2 r \max\{ \sqrt{nL},n \} \\
    &\lesssim \frac{\sigma^4 r^3 n^2}{\lambda^4 L^2}. \numberthis \label{gtwobound}
\end{align*}
where we have used the fact that $L \lesssim n$.
 \item \textbf{Bounding $\langle \bm{G}^{(1)}, \bm{G}^{(2)} \rangle$}. Analyzing this quantity is somewhat involved and requires decoupling arguments, so its analysis is contained in the following lemma.
\begin{lemma} \label{lem:thirdorderguy}
Under the conditions of \cref{thm:asymptotics}, for all $u \leq L^{1/4}$, it holds that
\begin{align*}
\bigg|    \langle \bm{G}\one, \bm{G}\two \rangle \bigg| 
&\lesssim  u  \frac{\kappa \sigma^3 \sqrt{n} r^{5/2}}{\lambda^3 L} \max\{ 1, \sqrt{\frac{n}{L}} \} 
\end{align*}
with probability at least $1 - r^2 \exp( - c u^2) - r \exp( - c n)$. 
   \end{lemma}
   \begin{proof}
       See \cref{sec:thirdorderguyproof}.
   \end{proof}
Thus, by taking $u \asymp L^{1/4}$, as long as $\log(r) \leq c \sqrt{L}$, we have, with probability at least $1 - \exp(- c \sqrt{L})$, that 
\begin{align*}
    \bigg| \langle \bm{G}^{(1)}, \bm{G}^{(2)} \rangle \bigg| \lesssim L^{1/4} \frac{\kappa \sigma^3 n r^{5/2}}{\lambda^3 L^{3/2}} \lesssim \frac{\sigma^3 \kappa n^{3/2} r^{5/2}}{\lambda^3 L^{3/2}}. \numberthis \label{g1g2}
\end{align*}
    \item \textbf{Bounding $\sum_{i} \langle \bm{G}^{(1)},T_i\rangle$}.  For terms $T_1, T_2$, and $T_4$, we apply Cauchy-Schwarz.  By \eqref{t1t2},\eqref{t2}, and \eqref{t4} we have that
    \begin{align*}
        | \langle \bm{G}^{(1)}, T_i \rangle | \lesssim \frac{\sigma \sqrt{nr}}{\lambda \sqrt{L}} \frac{\sigma^2 \kappa^4 n r^{3/2}}{\lambda^2 L } \asymp \frac{\sigma^3 \kappa^4 n^{3/2} r^2}{\lambda^3 L^{3/2}}.
    \end{align*}
    For the final term $\langle \bm{G}^{(1)}, T_3 \rangle$, we also require a more involved argument, which is given in the following result.
\begin{lemma} \label{lem:complicatedresidual2}
Under the conditions of \cref{thm:asymptotics}, with probability at least $1 - \exp( - c n)$ it holds that
    \begin{align*}
        \bigg| \langle \bm{G}^{(1)} , T_3 \rangle \bigg| \lesssim   \frac{\sigma^4 \kappa^2 n^2 r^{7/2}}{\lambda^4 L^{3/2}}
    \end{align*}
\end{lemma}
\begin{proof}
    See \cref{sec:complicatedresidual2proof}.
\end{proof}
Thus,  we have that with probability at least $1 - \exp(-cn),$
\begin{align*}
    \bigg| \langle \bm{G}^{(1)} , T_1 + T_2 + T_3 + T_4 \rangle \bigg| \lesssim  \frac{\sigma^3 n^{3/2} \kappa^4 r^2}{\lambda^3 L^{3/2}} + \frac{\sigma^4 \kappa^2 n^2 r^{7/2}}{\lambda^4 L^{3/2}}. \numberthis \label{g1resbound}
\end{align*}
    \item \textbf{Bounding $\sum_{i} \langle \bm{G}^{(2)},T_i\rangle$}.
 By the argument leading to \eqref{gtwobound} we have
   \begin{align*}
       \| \bm{G}^{(2)} \|_F \lesssim \frac{\sigma^2 r^{3/2} \sqrt{n}}{\lambda^2 \sqrt{L}} \max\{ 1, \sqrt{\frac{n}{L}}\} \lesssim \frac{\sigma^2 r^{3/2} n}{\lambda^2 L}.
   \end{align*}
   Therefore, by \eqref{t1t2}, \eqref{t2}, \cref{lem:complicatedresidual}, and \eqref{t4},
   \begin{align*}
       \big|\sum_{i} \langle \bm{G}^{(2)},T_i\rangle\big| &\lesssim \frac{\sigma^2 r^{3/2}n}{\lambda^2 L} \bigg( \frac{\sigma^2 \kappa^2 r^{3/2} n}{L \lambda^2} +\frac{\sigma^2 \kappa^4 n \sqrt{r}}{\lambda^2 L} +  \frac{\sigma^3 n^{3/2} r^2 \kappa}{\lambda^3L}  +  \frac{\sigma^2 \kappa^2 nr}{\lambda^2 L} \bigg) \\
       &\asymp \frac{\sigma^4 r^3 n^2 \kappa^2}{\lambda^4 L^2} + \frac{\sigma^4 \kappa^4 r^2 n^2}{\lambda^4 L} + \frac{\sigma^5 r^{7/2} \kappa n^{5/2}}{\lambda^5 L^2} + \frac{\sigma^4 \kappa^2 r^{5/2} n^2}{\lambda^4 L^2} \\
       &\lesssim \frac{\sigma^4 r^3 \kappa^4 n^2}{\lambda^4 L^2} + \frac{\sigma^5 r^{7/2} \kappa n^{5/2}}{\lambda^5 L^2}. \numberthis \label{G2Tibound}
   \end{align*}
    \item \textbf{Bounding $\sum_{i=1}^{4} \| T_i \|_F^2$}.  By \eqref{t1t2}, \eqref{t2}, \cref{lem:complicatedresidual}, and \eqref{t4}, we have that 
    \begin{align*}
        \sum_{i=1}^{4} \|T_i\|_F^2 &\lesssim \bigg(\frac{\sigma^2 \kappa^2 r^{3/2} n}{L \lambda^2} \bigg)^2 + \bigg(\frac{\sigma^2 \kappa^4 n \sqrt{r}}{\lambda^2 L} \bigg)^2 + \bigg( \frac{\sigma^3 n^{3/2} r^2 \kappa}{\lambda^3L} \bigg)^2 + \bigg( \frac{\sigma^2 \kappa^2 nr}{\lambda^2 L} \bigg)^2 \\
        &\lesssim  \frac{\sigma^4 \kappa^8 r^3 n^2}{\lambda^4 L^2}  + \frac{\sigma^6 n^3 r^4 \kappa^2}{\lambda^6 L^2}.  \numberthis \label{tisquared}
    \end{align*}
\end{itemize}

Thus, combining \eqref{gtwobound}, \eqref{g1g2},  \eqref{g1resbound}, \eqref{G2Tibound}, \eqref{tisquared}, we have that
\begin{align*}
    \| \pert\uhat \|_F^2 &= \| \bm{G}^{(1)} \|_F^2
    + {\sf Res},
   \end{align*}
   where
   \begin{align*}
       {\sf Res} &\lesssim \frac{\sigma^4 r^3 n^2}{\lambda^4 L^2} + \frac{\sigma^3 \kappa n^{3/2} r^{5/2}}{\lambda^3 L^{3/2}} + \frac{\sigma^3 n^{3/2}\kappa^4 r^2}{\lambda^3 L^{3/2}} + \frac{\sigma^4 \kappa^2 n^2 r^{7/2}}{\lambda^4 L^{3/2}} \\
       &\quad + \frac{\sigma^4 r^3 \kappa^4 n^2}{\lambda^4 L^2} + \frac{\sigma^5 r^{7/2} \kappa n^{5/2}}{\lambda^5 L^2} + \frac{\sigma^4 \kappa^8 r^3 n^2}{\lambda^4 L^2} + \frac{\sigma^6 n^3 r^4 \kappa^2}{\lambda^6 L^2}.
\end{align*}
Recall that in \cref{thm:asymptotics} we assume that 
\begin{align*}
    \lambda/\sigma \gg \frac{n r^{2} \kappa^4}{\sqrt{L}}.
\end{align*}
Thus, term-by-term comparison together with the assumption $L \lesssim n$ shows that 
\begin{align*}
    {\sf Res} \lesssim \frac{\sigma^3 \kappa^4 r^{5/2} n^{3/2}}{\lambda^3 L^{3/2}} + \frac{\sigma^4 \kappa^2 n^2 r^{7/2}}{\lambda^4 L^{3/2}}.
\end{align*}
\end{itemize}
This completes the proof. 
   \end{proof}
 
\subsubsection{Proof of \cref{lem:rlsquared}} \label{sec:rlsquaredproof}

\begin{proof}[Proof of \cref{lem:rlsquared}]
    We have that
    \begin{align*}
          \bigg\| \sum_{l} \bigg[ (\bm{\hat R}^{(l)})^2 - \mathcal{O}\t (\bm{R}^{(l)})^2 \mathcal{O} \bigg] \bigg\|_F &= \bigg\| \sum_{l} \bigg[ \bm{\hat U}\t ( \sl + \nl ) \uhat \uhat\t  ( \sl + \nl ) \uhat  - \mathcal{O}\t \U\t (\sl)^2  \U  \mathcal{O} \bigg]\bigg\|_F \\
          &\leq  \bigg\| \sum_{l} \bigg[ \bm{\hat U}\t \sl\uhat \uhat\t  \sl  \uhat  - \mathcal{O}\t \U\t (\sl)^2  \U  \mathcal{O} \bigg]\bigg\|_F \\
          &\quad + \bigg\| \sum_l \uhat\t \nl \uhat \uhat\t \nl \uhat \bigg\|_F +  2\bigg\| \sum_l \uhat\t \nl \uhat \uhat\t \sl \uhat \bigg\|_F \\
            &\leq 
            \bigg\| \sum_{l} \bigg[ \O \U\t  \sl \U \U\t   \sl  \big( \U \O - \uhat \big)  \bigg]\bigg\|_F \\
            &\quad +  \bigg\| \sum_{l} \bigg[ \O \U\t  \sl \big( \U \U\t - \uhat \uhat\t \big)    \sl  \uhat \bigg]\bigg\|_F  \\
            &\quad + \bigg\| \sum_{l} \bigg[ ( \U \O - \uhat )\t  \sl\uhat \uhat\t  \sl  \uhat   \bigg]\bigg\|_F \\
          &\quad + \bigg\| \sum_l \uhat\t \nl \uhat \uhat\t \nl \uhat \bigg\|_F +  2\bigg\| \sum_l \uhat\t \nl \uhat \uhat\t \sl \uhat \bigg\|_F.
    \end{align*}
    It is straightforward to show that since $\|\pert \uhat \|_F \lesssim \err$ with probability at least $1 - \exp(  - c n)$ when $t$ is chosen as in the lemma then it holds that 
    \begin{align*}
         \bigg\| \sum_{l} \bigg[ \O \U\t  \sl \U \U\t   \sl  \big( \U \O - \uhat \big)  \bigg]\bigg\|_F &\lesssim  L \lambda_{\max}^2 \err \\
          \bigg\| \sum_{l} \bigg[ \O \U\t  \sl \big( \U \U\t - \uhat \uhat\t \big)    \sl  \uhat \bigg]\bigg\|_F &\lesssim  L \lambda_{\max}^2 \err \\
          \bigg\| \sum_{l} \bigg[ ( \U \O - \uhat )\t  \sl\uhat \uhat\t  \sl  \uhat   \bigg]\bigg\|_F  &\lesssim L \lambda_{\max}^2 \err.
    \end{align*}
    Therefore, we focus on the remaining two terms containing $\nl$.  
    \begin{itemize}
        \item \textbf{The Term $\| \sum_l \uhat\t \nl \uhat \uhat\t \nl \uhat \|_F$}.  We note that
        \begin{align*}
            \| \sum_l &\uhat\t \nl \uhat \uhat\t \nl \uhat \|_F \\
            &=   \| \sum_l \uhat\t \bigg( \per \pert + \U \U\t \bigg)  \nl \uhat \uhat\t \nl \bigg( \per \pert + \U \U\t \bigg) \uhat \|_F \\
            &\leq  \| \sum_l \uhat\t  \per \pert   \nl  \uhat \uhat\t  \nl \per \pert \uhat \|_F \\ 
            &\quad + 2\| \sum_l \uhat\t  \per \pert   \nl\bigg( \uhat \uhat\t - \U\U\t \bigg) \nl\U \U\t  \uhat \|_F \\
            &\quad +2  \| \sum_l \uhat\t  \per \pert   \nl \U\U\t  \nl\U \U\t  \uhat \|_F \\
            &\quad +  \| \sum_l \uhat\t  \U \U\t   \nl  \bigg( \uhat \uhat\t  - \U\U\t \bigg) \nl \U \U\t  \uhat \|_F \\ 
            &\quad +  \| \sum_l \uhat\t  \U \U\t   \nl  \U\U\t \nl \U \U\t  \uhat \|_F \\
            &=: \sum_{i=1}^{5} M_i.
        \end{align*}
        We analyze each in turn. 
        \begin{itemize}
            \item \textbf{The Term $M_1$}. By \eqref{assertion2} we have that on the event $\mathcal{E}_{{\sf good}}$,
            \begin{align*}
                T_1 &\leq \| \uhat\t \per \|_F^2 \bigg\| \sum_l \pert \nl  \uhat \uhat\t  \nl \per \bigg\| \\
                &\leq \| \uhat\t \per \|_F^2 \bigg\|\pert \sum_l \bigg[ \nl  \uhat \uhat\t  \nl - \frac{\sigma^2  }{2} \bigg( \uhat \uhat\t + r\bm{I} \bigg)  \bigg]  \per \bigg\| \\ 
                &\quad + \| \uhat\t \per \|_F^4 \frac{\sigma^2 L r}{2} \\
                &\lesssim \bigg( \err \bigg)^2 \sigma^2 nr \sqrt{L} \| \uhat \uhat\t \|_F + \bigg( \err \bigg)^2 L \sigma^2  \bigg\| \pert \uhat \uhat\t\per + r \bm{I}_{n-r} \bigg\| \\
                &\lesssim \bigg( \err \bigg)^2 \sigma^2 nr^{3/2} \sqrt{L} + \bigg( \err \bigg)^2 \sigma^2 L r.
            \end{align*}
            \item \textbf{The Term $M_2$.} We bound similarly via
            \begin{align*}
                M_2 &\lesssim \| \uhat\t\per\|_F \bigg\| \sum_l \pert \nl \bigg(\uhat \uhat\t - \U \U\t \bigg) \nl \U \bigg\| \\
                &\lesssim \| \uhat \t \per \|_F \bigg\| \pert \bigg( \sum_l \nl \bigg[ \uhat\uhat\t - \U\U\t \bigg] \nl - \frac{\sigma^2}{2} \bigg[  \uhat\uhat\t - \U\U\t \bigg] \U \bigg\| + \| \uhat\t \per \|_F \frac{\sigma^2 L}{2} \bigg\| \pert \uhat \uhat\t \U \bigg\| \\
                &\lesssim \err \sigma^2 nr \sqrt{L} \| \uhat \uhat\t - \U\U\t \|_F + \bigg( \err \bigg)^2 \sigma^2 L \\
                &\lesssim \bigg( \err \bigg)^2 \sigma^2 n r \sqrt{L} + \bigg( \err \bigg)^2 \sigma^2 L.
            \end{align*}
            \item \textbf{The Term $M_3$.} We again proceed via a similar argument:
            \begin{align*}
                M_3 &\lesssim \| \uhat\t \per \|_F \bigg\|  \pert \sum_l \bigg( \nl \U\U\t \nl - \frac{\sigma^2}{2} \bigg[ \U\U\t +r \bm{I} \bigg] \bigg) \per \bigg\| + \| \uhat\t \per \|_F \frac{\sigma^2 L r}{2}  \\
                &\lesssim \err \sigma^2 nr \sqrt{L} \| \U\U\t \|_F + \err \sigma^2 L r \\
                &\lesssim \err \sigma^2 n r^{3/2} \sqrt{L} + \err \sigma^2 Lr.
            \end{align*}
            \item \textbf{The Term $M_4$.} We simply have that
            \begin{align*}
                M_4 &\lesssim  \bigg\| \U\t \sum_l \bigg( \nl \bigg( \uhat \uhat\t -\U\U\t \bigg) \nl - \frac{\sigma^2}{2} \bigg[ \uhat\uhat\t - \U\U\t \bigg] \bigg] \bigg) \U \bigg\|_F +  \frac{\sigma^2 L}{2} \| \uhat\uhat\t - \U\U\t \|_F \\
                &\leq \sqrt{r} \bigg\| \U\t \sum_l \bigg( \nl \bigg( \uhat \uhat\t -\U\U\t \bigg) \nl - \frac{\sigma^2}{2} \bigg[ \uhat\uhat\t - \U\U\t \bigg] \bigg] \bigg) \U \bigg\| + \frac{\sigma^2 L}{2} \| \uhat\uhat\t - \U\U\t \|_F
                \\
                &\lesssim \sigma^2 n r^{3/2} \sqrt{L} \| \uhat \uhat\t - \U\U\t \|_F + \sigma^2 L     \| \uhat \uhat\t - \U\U\t \|_F        \\
                &\lesssim \err \sigma^2 n r^{3/2} \sqrt{L} + \err \sigma^2 L                \end{align*}
                \item \textbf{The Term $M_5$}.  First, we note that
                \begin{align*}
                    M_5 &\leq  \sqrt{r} \bigg\| \sum_l \U\t \nl \U \U\t \nl \U \bigg\|.
                \end{align*}
                Define the matrix
                \begin{align*}
                    \bm{Z} := \big[ \U\t \bm{N}^{(1)} \U, \cdots , \U\t \bm{N}^{(L)} \U \big] \in \mathbb{R}^{r \times rL}.
                \end{align*}
                Then it holds that
                \begin{align*}
                    \bigg\| \sum_l \U\t \nl \U \U\t \nl \U \bigg\| &= \big\| \bm{ZZ}\t \big\| = \| \bm{Z} \|^2.
                \end{align*}
                A straightforward $\eps$-net argument shows that $\|\bm{Z}\| \lesssim \sigma \sqrt{Lr} + \sigma \sqrt{nr}$ with probability at least $1 - \exp( - nr)$.  Consequently, $\|\bm{ZZ}\t \|^2 \lesssim \sigma^2 Lr + \sigma^2 n r + \sigma^2 r \sqrt{nL}$. 
        \end{itemize}
    Combining these bounds shows that with probability at least $1 - \exp( - c n)$,
    \begin{align*}
        \| \sum_l \uhat\t \nl \uhat \uhat\t \nl \uhat \|_F &\lesssim 
        \err \sigma^2 n r^{3/2} \sqrt{L} + \err \sigma^2 Lr + \sigma^2 L r^{3/2}+ \sigma^2 n r^{3/2} + \sigma^2 r^{3/2} \sqrt{nL} \\
        &\lesssim \err \sigma^2 n r^{3/2} \sqrt{L} + \sigma^2 L r^{3/2} + \sigma^2 n r^{3/2} + \sigma^2 r^{3/2} \sqrt{nL}.  
    \end{align*}
    \item \textbf{The Term $\| \sum_l \uhat\t \nl \uhat \uhat\t \sl \uhat \|_F$}.  We bound using \eqref{assertion1} by
    \begin{align*}
        \| \sum_l \uhat\t \nl \uhat \uhat\t \sl \uhat \|_F &\leq \big\| \sum_l \uhat\t \nl \big( \uhat\uhat\t - \U\U\t \big) \sl \uhat \big\|_F \\
        &\quad + \big\| \sum_l \uhat\t \nl \sl \uhat \big\|_F \\
        &\lesssim C \sigma r \sqrt{n}\lambda_{\max} \err + C \sigma \sqrt{nLr} \lambda_{\max}
    \end{align*}.
    \end{itemize}
Combining all these inequalities, we end up with
\begin{align*}
    \bigg\| \sum_l \bigg[ \big( \bm{\hat R}^{(l)} \big)^2 - \O\t \big( \rl \big)^2\O \bigg] \bigg\|_F &\lesssim L \lambda_{\max}^2 \err + \err \sigma^2 n r^{3/2} \sqrt{L} + \sigma^2 L r^{3/2} \\
    &\quad + \sigma^2 n r^{3/2} + \sigma^2 r^{3/2} \sqrt{nL} + \sigma r \sqrt{n} \lambda_{\max} \err + \sigma \sqrt{nLr}\lambda_{\max} 
\end{align*}
with probability at least $1 - \exp( - c n)$.  

To prove the second bound, we note that
 $\bm{\mathcal{RR}}\t$ has smallest eigenvalue at least $L \lambda^2$. Thus, under the assumption that $\lambda/\sigma \gg \frac{\kappa^3 r^2 n}{\sqrt{L}}$, it holds that
\begin{align*}
   \bigg\| \sum_l \bigg[ \big( \bm{\hat R}^{(l)} \big)^2 - \O\t \big( \rl \big)^2\O \bigg] \bigg\|_F &\ll L \lambda^2,
\end{align*}
and hence Weyl's inequality implies that \begin{align*}\lambda_{\min} \bigg( \sum_l \big( \bm{\hat R}^{(l)} \big)^2\bigg) \gtrsim L \lambda^2.
\end{align*}
Thus, we have that 
\begin{align*}
    \bigg\| \bigg( \sum_l \big( \bm{\hat R}^{(l)} \big)^2 \bigg)\inv - \O\t \bigg( \sum_l \big( \rl \big)^2 \bigg)\inv \O \bigg\|_F &\lesssim \frac{1}{L^2 \lambda^4}   \bigg\| \sum_l \bigg[ \big( \bm{\hat R}^{(l)} \big)^2 - \O\t \big( \rl \big)^2\O \bigg] \bigg\|_F \\
    &\lesssim \frac{1}{\lambda^4 L^2} \bigg\{ L \lambda_{\max}^2 \err + \err \sigma^2 n r^{3/2} \sqrt{L} + \sigma^2 L r^{3/2} \\
    &\quad + \sigma^2 n r^{3/2} + \sigma^2 r^{3/2} \sqrt{nL} + \sigma r \sqrt{n} \lambda_{\max} \err + \sigma \sqrt{nLr}\lambda_{\max} \bigg\} \\
    &\asymp \err \frac{\kappa^2}{\lambda^2 L} + \err \frac{\sigma^2 n r^{3/2}}{\lambda^4 L^{3/2}} + \frac{\sigma^2 r^{3/2}}{\lambda^4 L^{3/2}}\\
    &\quad + \frac{\sigma^2 n r^{3/2}}{\lambda^4 L^2} + \frac{\sigma^2 r^{3/2} \sqrt{n}}{\lambda^4 L^{3/2}} + \err \frac{\sigma r \sqrt{n} \kappa}{\lambda^3 L^2} + \frac{\sigma \sqrt{nr} \kappa}{L^{3/2} \lambda^3}. 
\end{align*}
The final bound can be directly verified from the assumptions $L \lesssim n$ and $\lambda/\sigma \gg n r^2 \kappa^3/\sqrt{L}$.
\end{proof}

\subsubsection{Proof of \cref{lem:complicatedresidual}}

\label{sec:complicatedresidualproof}

\begin{proof}
We start with the full decomposition
   \begin{align*}
   T_3 &=  \sum_l \bm{U}_{\perp}\t \bm{N}^{(l)} \bm{\hat U} \bm{\hat U}\t \bm{N}^{(l)} \bm{\hat U} \mathcal{O}  \O\t \rhatinv\O -   \sum_l \bm{U}_{\perp}\t \bm{N}^{(l)} \U\U\t \bm{N}^{(l)} \U  \O\t \rhatinv\O\\
       &= \sum_l \bm{U}_{\perp}\t \bm{N}^{(l)} \bm{U}_{\perp} \bm{U}_{\perp}\t \bm{\hat U} \bm{\hat U}\t \per \bm{N}^{(l)} \bm{U}_{\perp} \bm{U}_{\perp}\t \bm{\hat U} \mathcal{O}  \O\t \rhatinv\O \\
       &\quad+ \sum_l \bm{U}_{\perp}\t \bm{N}^{(l)} \bm{U}_{\perp} \bm{U}_{\perp}\t \bm{\hat U} \bm{\hat U}\t \U\U\t \bm{N}^{(l)} \bm{U}_{\perp} \bm{U}_{\perp}\t \bm{\hat U} \mathcal{O}  \O\t \rhatinv\O \\
      &\quad + \sum_l \bm{U}_{\perp}\t \bm{N}^{(l)} \U \U\t \bm{\hat U} \bm{\hat U}\t \pert\nl \bm{U}_{\perp} \bm{U}_{\perp}\t \bm{\hat U} \mathcal{O}  \O\t \rhatinv\O \\
      &\quad + \sum_l \bm{U}_{\perp}\t \bm{N}^{(l)} \U \U\t \bm{\hat U} \bm{\hat U}\t \U\U\t \bm{N}^{(l)} \bm{U}_{\perp} \bm{U}_{\perp}\t \bm{\hat U} \mathcal{O}  \O\t \rhatinv\O \\
       &\quad+ \sum_l \bm{U}_{\perp}\t \bm{N}^{(l)} \per \bm{\hat U} \bm{\hat U}\t \per\per \bm{N}^{(l)} \U\U\t \bm{\hat U} \mathcal{O}  \O\t \rhatinv\O \\
           &\quad+  \sum_l \bm{U}_{\perp}\t \bm{N}^{(l)} \per \bm{\hat U} \bm{\hat U}\t \U\U\t \bm{N}^{(l)} \U\U\t \bm{\hat U} \mathcal{O}  \O\t \rhatinv\O \\
       &\quad+  \sum_l \bm{U}_{\perp}\t \bm{N}^{(l)} \U\U\t \bm{\hat U} \bm{\hat U}\t \per\per \bm{N}^{(l)} \U\U\t \bm{\hat U} \mathcal{O}  \O\t \rhatinv\O \\
         &\quad+  \sum_l \bm{U}_{\perp}\t \bm{N}^{(l)} \U\U\t\bigg(  \bm{\hat U} \bm{\hat U}\t - \U\U\t \bigg) \U\U\t \bm{N}^{(l)} \U\U\t \bm{\hat U} \mathcal{O}  \O\t \rhatinv\O \\
          &\quad+  \sum_l \bm{U}_{\perp}\t \bm{N}^{(l)} \U\U\t \bm{N}^{(l)} \U\U\t \big( \bm{\hat U} \mathcal{O} - \U)  \O\t \rhatinv\O \\
           &\quad+  \sum_l \bm{U}_{\perp}\t \bm{N}^{(l)} \U\U\t \bm{N}^{(l)} \U  \O\t \rhatinv\O \\
         &=: \sum_{i=1}^{9} J_i. \numberthis \label{decomp}
   \end{align*}
   We will bound each of these terms separately and collect the bounds at the end of the proof.  However, first recall that on the event $\mathcal{E}_{{\sf good}}$  by \eqref{assertion2} it holds that 
\begin{align*}
    \bigg\| \sum_{l} \bm{N}^{(l)} \bm{Q} \bm{N}^{(l)} - \sigma^2  L \big( \bm{Q}\t + {\sf Tr}(\bm{Q}) \bm{I} \big) \bigg\| \lesssim \sigma^2 nr \sqrt{L} \| \bm{Q} \|_F
\end{align*}
for all matrices $\bm{Q}$ of rank at most $2r$.  We will apply this result repeatedly without reference. For various choices of the matrix $\bm{Q}$ above, we refer to $\sigma^2 L \big( \bm{Q}\t + {\sf Tr}(\bm{Q}) \bm{I} \big)$ as the \emph{centering term}.  In the subsequent analysis, either the centering term will vanish, or we will include it directly.  
\begin{itemize}
    \item \textbf{The Term $J_1$}: By adding and subtracting the centering term, we have that 
    \begin{align*}
        \| J_1 \|_F
        &= \bigg\| \sum_l \bm{U}_{\perp}\t \nl \per \pert  \uhat\uhat\t \pert\nl \per\pert \uhat \O \O\t \rhatinv\O \bigg\|_F \\
        &\leq \bigg\| \bm{U}_{\perp}\t\bigg(  \sum_l \nl \per\pert  \uhat\uhat\t \per \pert \nl  - \sigma^2 L \big[\per\pert \uhat\uhat\t \per\pert  \\
        &\qquad \qquad \qquad \qquad \qquad + {\sf Tr} \big[ \per \pert \uhat \uhat\t \per \pert \big] \bm{I} \big]
        \bigg) \per \pert  \uhat \O \O\t \rhatinv\O
        \bigg\|_F \\
        &\quad + \sigma^2 L \bigg\|  \bm{U}_{\perp}\t\bigg( \per\pert \uhat\uhat\t \per \pert + {\sf Tr} \big[ \per \pert \uhat \uhat\t \per\pert  \big] \bm{I} \bigg) \per \pert \uhat \O \O\t \rhatinv\O \bigg\|_F\\
        &\leq  \bigg\| \bm{U}_{\perp}\t\bigg(  \sum_l \nl \per\pert  \uhat\uhat\t \per \pert \nl  - \sigma^2 L \big(\per\pert \uhat\uhat\t \per\pert  \\
        &\qquad \qquad \qquad \qquad \qquad + {\sf Tr} \big[ \per \pert \uhat \uhat\t \per \pert \big]  \big)
        \bigg) \per \bigg\| \big\| \pert  \uhat \big\|_F \| \O\t \rhatinv\O \| \\
          &\quad + \sigma^2 L \bigg\|  \bm{U}_{\perp}\t\bigg( \per\pert \uhat\uhat\t \per \pert + {\sf Tr} \big[ \per \pert \uhat \uhat\t \per\pert  \big] \bm{I} \bigg) \per \bigg\| \big\| \pert \uhat \big\|_F \big\| \O\t \rhatinv\O \big\|\\
        &\lesssim  \assertiontwo \| \per \pert \uhat \uhat\t \per \pert \|_F  \big\| \pert  \uhat \big\|_F \| \O\t \rhatinv\O \|
        \\
        &\quad + \sigma^2 L \bigg( \| \pert\uhat\uhat\t \per \|_F + {\sf Tr} \big[ \pert \uhat \uhat\t \per \big] \bigg) \| \pert\uhat \|_F \| \O\t \rhatinv\O \| \\
        &\lesssim \assertiontwo \| \pert\uhat\|_F^3 \| \O\t \rhatinv\O \|  + \sigma^2 L \| \pert\uhat\|_F^3 \| \O\t \rhatinv\O \| \\
        &\lesssim \frac{\assertiontwo }{L\lambda^2} \bigg( \err \bigg)^3 + \frac{\sigma^2}{\lambda^2} \bigg( \err \bigg)^3
        \numberthis \label{J1}
    \end{align*}
    \item \textbf{The Term $J_2$}: Since $\pert \U = 0$, it holds that ${\sf Tr}( \per \pert \uhat \uhat\t \U \U\t ) = {\sf Tr}( \U\t \per \pert \uhat \uhat\t \U) = 0$.  Then the centering term vanishes and hence
    \begin{align*}
    \| J_2 \| &= \bigg\| \sum_{l} \pert \nl \per\pert \uhat\uhat\t \U\U\t \nl \per\pert \uhat \O \O\t \rhatinv\O \bigg\|_F \\
        &= \bigg\| \pert \bigg( \sum_{l}  \nl \per\pert \uhat\uhat\t \U\U\t \nl \\
        &\qquad - \sigma^2 L \big(\U\U\t  \uhat \uhat\t \per \pert + {\sf Tr} \big[ \per \per \uhat \uhat\t \U\U\t \big] \bm{I} \big) \bigg) \per\pert \uhat \O \O\t \rhatinv\O \bigg\|_F \\
        &\leq \bigg\| \pert \bigg( \sum_{l}  \nl \per\pert \uhat\uhat\t \U\U\t \nl \\
        &\qquad - \sigma^2 L \big( \U\U\t  \uhat \uhat\t \per \pert + {\sf Tr} \big[ \per \per \uhat \uhat\t \U\U\t \big] \bm{I} \big) \bigg) \per \bigg\| \| \pert \uhat\|_F \| \O\t \rhatinv\O \| \\
        &\lesssim \assertiontwo \| \per \pert \uhat \uhat\t \U \U\t \|_F \| \pert\uhat\|_F \| \O\t \rhatinv\O \| \\
&\lesssim \frac{\assertiontwo}{L \lambda^2} \bigg( \err \bigg)^2. \numberthis \label{J2}
    \end{align*}
    \item \textbf{The Term $J_3$}. Similar to the prior argument, the centering term vanishes, and hence
    \begin{align*}
         \| J_3\|_F &= \bigg\| \sum_l \bm{U}_{\perp}\t \bm{N}^{(l)} \U \U\t \bm{\hat U} \bm{\hat U}\t \per\pert \bm{N}^{(l)} \bm{U}_{\perp} \bm{U}_{\perp}\t \bm{\hat U} \mathcal{O}  \O\t \rhatinv\O \bigg\|_F \\
         &= \bigg\| \bm{U}_{\perp}\t\bigg( \sum_l  \bm{N}^{(l)} \U \U\t \bm{\hat U} \bm{\hat U}\t \per \pert \bm{N}^{(l)} \\
         &\quad - \sigma^2 L \big( \per\pert \bm{\hat U} \bm{\hat U}\t\U\U\t + {\sf Tr} \big[ \U \U\t \bm{\hat U} \bm{\hat U}\t \per\pert\big] \bm{I}\big) \bigg)
         \bm{U}_{\perp} \bm{U}_{\perp}\t \bm{\hat U} \mathcal{O}  \O\t \rhatinv\O \bigg\|_F\\
      &\lesssim \frac{\assertiontwo}{\lambda^2 L} \| \U \U\t \uhat \uhat\t \per \pert \|_F \| \pert \uhat \|_F \\
      &\lesssim  \frac{ \assertiontwo}{\lambda^2 L} \bigg( \err \bigg)^2. \numberthis\label{j3}
    \end{align*}
    \item \textbf{The Term $J_4$}.  We note that only the trace part of the centering term does not vanish, which gives 
    \begin{align*}
         \|  J_4 \|_F &= \bigg\| \sum_l \bm{U}_{\perp}\t \bm{N}^{(l)} \U \U\t \bm{\hat U} \bm{\hat U}\t \U\U\t \bm{N}^{(l)} \bm{U}_{\perp} \bm{U}_{\perp}\t \bm{\hat U} \mathcal{O}  \O\t \rhatinv\O \bigg\|_F \\ 
         &\leq \bigg\| \bm{U}_{\perp}\t\sum_l  \bigg( \bm{N}^{(l)} \U \U\t \bm{\hat U} \bm{\hat U}\t \U\U\t \bm{N}^{(l)}  \\
         &\qquad - \sigma^2 L \big( \U \U\t \bm{\hat U} \bm{\hat U}\t \U\U\t + {\sf Tr}\big[\bm{\U \U\t \uhat \uhat\t \U \U\t }\big] \bm{I} \big) \bigg) \bm{U}_{\perp} \bm{U}_{\perp}\t \bm{\hat U} \mathcal{O}  \O\t \rhatinv\O \bigg\|_F \\
         &\quad + \sigma^2 L {\sf Tr}\big[\bm{\U \U\t \uhat \uhat\t \U \U\t }\big]   \bigg\| \bm{U}_{\perp}\t \bm{\hat U} \mathcal{O}  \O\t \rhatinv\O \bigg\|_F \\
         &\lesssim \frac{\assertiontwo}{\lambda^2 L} \| \U\U\t \uhat \uhat\t \U \U\t \|_F \| \pert \uhat \|_F + \frac{\sigma^2 L r}{\lambda^2 L} \| \pert \uhat \|_F \\
         &\lesssim \frac{\assertiontwo \sqrt{r} }{\lambda^2 L} \err + \frac{\sigma^2 r}{\lambda^2} \err, \numberthis\label{j4}
       \end{align*}
       where we have observed that $\|\U\U\t \uhat\uhat\t \U\U\t \|_F \leq \sqrt{r}$ and ${\sf Tr}( \U\t \uhat \uhat\t \U ) = \| \uhat\t \U \|_F^2 \leq r$.  
    \item \textbf{The Term $J_5$}.  As in previous cases, the centering term vanishes (due to the pre and post multiplication by $\pert$ and $\U$ respectively).  Consequently, 
    \begin{align*}
      \|  J_5\|_F &= \bigg\|  \sum_l \bm{U}_{\perp}\t \bm{N}^{(l)} \per \pert \bm{\hat U} \bm{\hat U}\t \per\per \bm{N}^{(l)} \U\U\t \bm{\hat U} \mathcal{O}  \O\t \rhatinv\O  \bigg\|_F \\
   &= \bigg\| \bm{U}_{\perp}\t\bigg(  \sum_l  \bm{N}^{(l)} \per \pert  \bm{\hat U} \bm{\hat U}\t \per\pert \bm{N}^{(l)} \\
   &\qquad \qquad \qquad - \sigma^2 L\big( \per \pert  \bm{\hat U} \bm{\hat U}\t \per\pert + {\sf Tr}\big[ \per \pert  \bm{\hat U} \bm{\hat U}\t \per\pert \big] \bm{I} \big) \bigg) \U\U\t \bm{\hat U} \mathcal{O}  \O\t \rhatinv\O  \bigg\|_F \\
   &\lesssim \frac{\assertiontwo}{\lambda^2 L}  \| \pert\uhat \uhat\t \pert \|_F \| \U\t \uhat\|_F \\
   &\lesssim  \frac{\assertiontwo \sqrt{r} }{\lambda^2 L} \bigg( \err \bigg)^2. \numberthis\label{j5}
         \end{align*}
    \item \textbf{The Term $J_6$}.  The centering term vanishes, yielding
    \begin{align*}
         \| J_6  \|_F &= \big\|   \sum_l \bm{U}_{\perp}\t \bm{N}^{(l)} \per\pert  \bm{\hat U} \bm{\hat U}\t \U\U\t \bm{N}^{(l)} \U\U\t \bm{\hat U} \mathcal{O}  \O\t \rhatinv\O \big\|_F \\
         &\leq \big\|   \bm{U}_{\perp}\t \bigg( \sum_l  \bm{N}^{(l)} \per\pert \bm{\hat U} \bm{\hat U}\t \U\U\t \bm{N}^{(l)} \\
         &\qquad \qquad \qquad - \sigma^2 L \big( \U\U\t \bm{\hat U} \bm{\hat U}\t \per\pert + {\sf Tr}\big[\per \pert \bm{\hat U} \bm{\hat U}\t \U\U\t\big] \bm{I} \big) \bigg) \U\U\t \bm{\hat U} \mathcal{O}  \O\t \rhatinv\O \big\|_F \\
         &\lesssim \frac{\assertiontwo}{\lambda^2 L} \|\per \pert \uhat \uhat \t \U \U\t \| \| \U\t \uhat \|_F \\ 
         &\lesssim \frac{\assertiontwo}{\lambda^2 L} \bigg( \err \bigg)^2. 
         \numberthis\label{j6}
        \end{align*}
    \item \textbf{The Term $J_7$}. Only the trace term vanishes, yielding
    \begin{align*}
    \| J_7 \|_F &= \bigg\|   \sum_l \bm{U}_{\perp}\t \bm{N}^{(l)} \U\U\t \bm{\hat U} \bm{\hat U}\t \per\per \bm{N}^{(l)} \U\U\t \bm{\hat U} \mathcal{O}  \O\t \rhatinv\O \bigg\|_F  \\
    &\leq \bigg\|   \bm{U}_{\perp}\t \bigg( \sum_l \bm{N}^{(l)} \U\U\t \bm{\hat U} \bm{\hat U}\t \per\pert \bm{N}^{(l)} \\
    &\qquad \qquad \qquad - \sigma^2 L \big( \per\pert \bm{\hat U} \bm{\hat U}\t \U\U\t + {\sf Tr}\big[\U\U\t \bm{\hat U} \bm{\hat U}\t \per\pert\big] \bm{I} \big) \bigg)  \U\U\t \bm{\hat U} \mathcal{O}  \O\t \rhatinv\O \bigg\|_F \\
    &\quad + \sigma^2 L \| \pert \uhat \uhat\t \U \|_F \| \O\t \rhatinv\O \| \\
    &\lesssim \frac{\assertiontwo}{\lambda^2 L} \| \U\U\t \uhat \uhat\t \per \pert \|_F \| \U\t \uhat \|_F + \frac{\sigma^2 }{\lambda^2} \err \\
    &\lesssim  \frac{\assertiontwo \sqrt{r} }{\lambda^2 L} \err + \frac{\sigma^2 }{\lambda^2} \err. \numberthis \label{j7}
        \end{align*}
    \item \textbf{The Term $J_8$}.  Here the centering term vanishes giving
    \begin{align*}
        \|  J_8 \|_F &= \bigg\|   \sum_l \bm{U}_{\perp}\t \bm{N}^{(l)} \U\U\t\bigg(  \bm{\hat U} \bm{\hat U}\t - \U\U\t \bigg) \U\U\t \bm{N}^{(l)} \U\U\t \bm{\hat U} \mathcal{O}  \O\t \rhatinv\O  \bigg\| \\
        &= \bigg\|   \bm{U}_{\perp}\t\bigg(  \sum_l \bm{N}^{(l)} \U\U\t\bigg(  \bm{\hat U} \bm{\hat U}\t - \U\U\t \bigg) \U\U\t \bm{N}^{(l)} \\
        &\qquad \qquad 
        -\sigma^2L \bigg[  \U\U\t\big(  \bm{\hat U} \bm{\hat U}\t - \U\U\t \big) \U\U\t + {\sf Tr}\big[ \U\U\t\bigg(  \bm{\hat U} \bm{\hat U}\t - \U\U\t \bigg) \U\U\t \big] \bm{I}\bigg]
        \bigg) \U\U\t \bm{\hat U} \mathcal{O}  \O\t \rhatinv\O  \bigg\| \\
        &\lesssim  \frac{\sigma^2 \maxterm}{\lambda^2 L} \| \U \U\t \big( \uhat\uhat\t - \U \U\t \big) \U \U\t \|_F \| \U\t \uhat \|_F \\
        &\lesssim  \frac{\assertiontwo  \sqrt{r} }{\lambda^2 L} \err. \numberthis \label{j8}
    \end{align*}
    \item \textbf{The term $J_9$}.  For $J_9$ we instead invoke \eqref{assertion3} to yield
    \begin{align*}
      \|   J_9 \|_F &=\bigg\|  \sum_l \bm{U}_{\perp}\t \bm{N}^{(l)} \U\U\t \bm{N}^{(l)} \U\U\t \big( \bm{\hat U} \mathcal{O} - \U)  \O\t \rhatinv\O \bigg\|_F \\
&\lesssim \bigg\|  \sum_l \bm{U}_{\perp}\t \bm{N}^{(l)} \U\U\t \bm{N}^{(l)} \U\U\t \bigg\| \|  \bm{\hat U} \mathcal{O} - \U \|_F \| \O\t \rhatinv\O \|   \\
      &\lesssim  \frac{\sigma^2 r \max\{ \sqrt{nL},n\}}{L \lambda^2} \err. \numberthis \label{j9}
    \end{align*}
\end{itemize}
Define
\begin{align*}
    \delta := \err.
\end{align*}
Combining \eqref{J1},\eqref{J2},\eqref{j3},\eqref{j4},\eqref{j5},\eqref{j6},\eqref{j7},\eqref{j8}, and \eqref{j9}, we have that 
\begin{align*}
\sum_i \| J_i \|_F &\lesssim 
    \delta^3 \bigg( \frac{\assertiontwo}{L\lambda^2} + \frac{\sigma^2}{\lambda^2} \bigg) + \delta^2 \bigg( \frac{\sigma^2 nr^{3/2} \sqrt{L}}{\lambda^2 L} + \bigg) + \delta \bigg( \frac{\sigma^2 n r^{3/2} \sqrt{L}}{\lambda^2 L} + \frac{\sigma^2 r}{\lambda^2} + \frac{\sigma^2 r \max\{ \sqrt{nL},n\}}{L\lambda^2} \bigg) \\
    &\lesssim  \delta \bigg( \frac{\sigma^2 n r^{3/2} }{\lambda^2 \sqrt{L}} + \frac{\sigma^2 r}{\lambda^2} + \frac{\sigma^2 r \max\{ \sqrt{nL},n\}}{L\lambda^2} \bigg) \\
    &\lesssim \err \frac{\sigma^2 n r^{3/2}}{\lambda^2 \sqrt{L}},
\end{align*}
where the final bound is due to the assumption that $L \lesssim n$. 
This completes the proof. 
   \end{proof}

  \subsubsection{Proof of \cref{lem:geodesicconvexity}} \label{sec:geodesicconvexity}
  \begin{proof}
Define
\begin{align*}
    h^{{\sf clean}}(\bm{W}) &= \frac{1}{4L} \sum_{l} \| \sl - \bm{WW}\t \sl \bm{WW}\t \|_F^2,
\end{align*}
which is the noiseless loss function.  Then for any $\X$ satisfying $\bm{W}\t \X = 0$, 
\begin{align*}
    \nabla^2_{{\sf Riemannian}}h   ^{{\sf clean}}( \bm{W})[ \X, \X] &:=  \frac{1}{L} \sum_{l} {\sf Tr} \bigg[ \X\t \X (\bm{W}\t \sl \bm{W})^2 \bigg] - \frac{1}{L} \sum_{l} {\sf Tr} \bigg[ \X\t \sl \X \bm{W}\t \sl \bm{W} \bigg] \\
    &\quad - \frac{1}{L} \sum_{l} {\sf Tr} \bigg[ \big( \X\t \sl \bm{W} \big)^2 \bigg] - \frac{1}{L} \sum_{l} \| \X\t \sl \bm{W} \|_F^2. \numberthis \label{nablahclean}
\end{align*}
Thus, we have the decomposition
\begin{align*}
    \nabla^2_{{\sf Riemannian}}& h(\bm{W})[\X,\X]\\
    &=  \nabla^2_{{\sf Riemannian}} h^{{\sf clean}}(\bm{W})[\X,\X] \\
    &+ \frac{1}{L} \sum_{l} \tr \bigg[ \X\t \X \W\t \sl \W \W\t \nl \W \bigg] + \frac{1}{L} \sum_{l} \tr\bigg[ \X\t \X \W\t \nl \W \W\t \sl \W \bigg] \\
    &\quad + \frac{1}{L} \sum_{l} \tr \bigg[ \X\t\X ( \W\t \nl \W )^2 \bigg] - \frac{1}{L} \sum_{l} \tr \bigg[ \X\t \sl \X \W\t \nl \W \bigg] - \frac{1}{L} \sum_{L} \tr\bigg[ \X\t \nl \X \W\t \sl \W \bigg] \\
    &\quad - \frac{1}{L} \sum_{l} \tr\bigg[ \X\t\nl \X \W\t\nl \W \bigg] - \frac{2}{L} \sum_{l} \tr\bigg[ \X\t\sl \W \X\t\nl \W \bigg] \\
    &\quad - \frac{1}{L} \sum_{l} \tr \bigg[ ( \X\t \nl \W )^2 \bigg] - \frac{2}{L} \sum_{l} \tr\bigg[ \W\t\sl \X \X\t \nl \W \bigg] - \frac{1}{L} \sum_{l} \| \X\t \nl \W \|_F^2.
\end{align*}
Define
\begin{align*}
{\sf L}_{\nabla^2 h}(\W,\X) &:= \frac{1}{L} \sum_{l} \tr \bigg[ \X\t \X \W\t \sl \W \W\t \nl \W \bigg]+ \frac{1}{L} \sum_{l} \tr\bigg[ \X\t \X \W\t \nl \W \W\t \sl \W \bigg] \\
&\quad - \frac{1}{L} \sum_{l} \tr \bigg[ \X\t \sl \X \W\t \nl \W \bigg] - \frac{1}{L} \sum_{L} \tr\bigg[ \X\t \nl \X \W\t \sl \W \bigg] \\
&\quad - \frac{2}{L} \sum_{l} \tr\bigg[ \X\t\sl \W \X\t\nl \W \bigg]- \frac{2}{L} \sum_{l} \tr\bigg[ \W\t\sl \X \X\t \nl \W \bigg]; \\
{\sf Q}_{\nabla^2 h}(\W,\X) &:= \frac{1}{L} \sum_{l} \tr \bigg[ \X\t\X ( \W\t \nl \W )^2 \bigg] - \frac{1}{L} \sum_{l} \tr\bigg[ \X\t\nl \X \W\t\nl \W \bigg] \\
&\quad - \frac{1}{L} \sum_{l} \tr \bigg[ ( \X\t \nl \W )^2 \bigg] - \frac{1}{L} \sum_{l} \| \X\t \nl \W \|_F^2, \numberthis \label{Qnabladef}
\end{align*}
the first being linear in $\nl$ and the second being quadratic in $\nl$.  Then we have that
\begin{align*}
   \nabla^2_{{\sf Riemannian}} h(\W) [ \X,\X] = \nabla^2_{{\sf Riemannian}} h^{{\sf clean}}(\W)[\X,\X] + {\sf L}_{{\nabla^2 h}}(\W,\X)  + {\sf Q}_{\nabla^2 h}(\W,\X) .  
\end{align*}
First we bound $\nabla^2_{{\sf Riemannian}} h^{{\sf clean}}(\W)[\X,\X]$ and then we upper bound the two noise terms.
\\ \ \\
\noindent
\textbf{Bounding $\nabla^2_{{\sf Riemannian}} h^{{\sf clean}}$:} 
Observe that
\begin{align*}
\lambda_{\max}^2 \|\X\|_F^2 \geq \frac{1}{L} \sum_{l} \tr\bigg[ \X\t\X ( \rl)^2 \bigg] \geq \| \X \|_F^2 \lambda^2.
\end{align*}
Without loss of generality we may assume that $\W = \W \mathcal{O}_{\W,\U}$ since $h(\W) = h(\W \mathcal{O})$ for any orthogonal $\mathcal{O}$.  Then we have that
\begin{align*}
    \| \W - \U \|_F \leq  R := C \err,
\end{align*}
where $C$ is some sufficiently large fixed constant.  Now suppose that $\X$ satisfies $\X\t \W = 0$.  Then the first term on the right hand side of \eqref{nablahclean} satisfies
\begin{align*}
    \frac{1}{L} \sum_{l} \tr \bigg[ \X\t \X (\W\t \sl \W)^2 \bigg] &\geq \| \X \|_F^2 \lambda_{\min} \bigg( \frac{1}{L} ( \W\t \sl \W)^2 \bigg).
\end{align*}
We have that
\begin{align*}
    \| \W\t \U - \bm{I}_r \|_F \leq \| \W - \U \|_F \leq R.
\end{align*}
Thus, for any $l$, since $\|\W\t\U\|\leq 1$, 
\begin{align*}
    \| (\W\t \U \rl \U\t \W)^2- (\rl)^2 \|_F &\leq \| (\W\t \U \rl \U\t \W - \rl ) \W\t \U \rl \U\t \W  \|_F + \| \rl ( \W\t \U \rl \U\t \W  - \rl)\|_F \\
    &\leq 2 \lambda_{\max} \| \W\t \U \rl \U\t \W - \rl \|_F \\
    &\leq 2 \lambda_{\max} \bigg( 2\| \W\t \U - \bm{I}  \|_F \| \rl \| + \| \W\t \U - \bm{I} \|_F^2 \| \rl\| \bigg) \\
    &\leq 6 R \lambda_{\max}^2.
\end{align*}
Thus, it holds that
\begin{align*}
    \frac{1}{L} \sum_{l} \tr \bigg[ \X\t\X (\W\t\sl \W)^2 \bigg] &\geq \| \X \|_F^2 \bigg( \lambda^2 - 6 R \lambda_{\max}^2 \bigg).
\end{align*}
Similarly
\begin{align*}
      \frac{1}{L} \sum_{l} \tr \bigg[ \X\t\X (\W\t\sl \W)^2 \bigg] &\leq  \| \X \|_F^2 \lambda_{\max}^2.
\end{align*}
We now upper bound the remaining three terms on the right hand side of \eqref{nablahclean}.  First, we note that since $ \W\t \X = 0$, then $ \U\t \X  = ( \U - \W )\t \X$ and hence $\| \U\t \X \|_F \leq R \|\X \|_F$.  Then it holds that
\begin{align*}
\bigg|    \frac{1}{L} \sum_{l} \tr \bigg[ \X\t \sl \X \W \t\sl \W \bigg] \bigg| &\leq \frac{1}{L} \sum_{l} \| \X\t \U \|_F^2 \lambda_{\max}^2 \leq R^2 \| \X \|^2 \lambda_{\max}^2.
\end{align*}
Similarly, 
\begin{align*}
    \frac{1}{L} \sum_{l} \tr\bigg[ (\X\t \sl \W )^2 \bigg] &\leq  \frac{1}{L} \sum_{l} \| \X\t \sl \W \|_F^2 \leq \| \X \|^2 R^2 \lambda_{\max}^2; \\
    \frac{1}{L} \sum_{l} \| \X\t \sl \W \|_F^2 &\leq \| \X\|^2 R^2 \lambda_{\max}^2.
\end{align*}
Thus, we have that
\begin{align*}
     \nabla^2_{{\sf Riemannian}} h^{{\sf clean}}(\W)[\X,\X] \geq \| \X \|_F^2 \bigg(  \lambda^2 - 6 R \lambda_{\max}^2  - 3 R^2 \lambda_{\max}^2 \bigg) \geq \| \X \|_F^2 \frac{\lambda^2}{2}
\end{align*}
as long as $\frac{\kappa^3 \sqrt{nr}}{\lambda \sqrt{L}} \ll 1$.  In addition,
\begin{align*}
      \nabla^2_{{\sf Riemannian}} h^{{\sf clean}}(\W)[\X,\X] \leq  2 \lambda_{\max}^2 \| \X \|_F^2.
\end{align*}
\\ \ \\
\noindent \textbf{Upper bounding ${\sf L}_{{\nabla^2 h}}(\W,\X)$}. We bound only the first term; the other terms are similar. The first term of ${\sf L}_{{\nabla^2 h}}(\W,\X)$ is of the form
\begin{align*}
    \frac{1}{L} \sum_{l} \langle \nl, \W \X\t \X \W\t \sl \W \W\t \rangle.  
\end{align*}
Let $\bm{Q}_1$ and $\bm{Q}_2$ be any fixed rank at most $r$ matrices with Frobenius norm at most 1.  Then a standard Hoeffding inequality argument shows that 
\begin{align*}
    \bigg| \frac{1}{L} \sum_{l} \langle \nl, \bm{Q}_1 \sl \bm{Q}_2 \rangle \bigg|  \lesssim \frac{\lambda_{\max} \sigma}{\sqrt{L}} t
\end{align*}
with probability at least $1- \exp( - c t^2)$.  Taking a union bound and applying a simple net argument shows that 
\begin{align*}
    \bigg| \frac{1}{L} \sum_{l} \langle \nl, \bm{Q}_1 \sl \bm{Q}_2 \rangle \bigg| \lesssim \frac{\lambda_{\max} \sigma \sqrt{nr}}{\sqrt{L}}
\end{align*}
with probability at least $1 - \exp( - c nr)$.  Applying this argument to each linear quantity shows that
\begin{align*}
    | {\sf L}_{\nabla^2}(\W,\X) | \lesssim \| \X \|^2 \frac{\lambda_{\max} \sigma \sqrt{nr}}{\sqrt{L}}
\end{align*}
with probability at least $1 - \exp( - c nr)$, uniformly over all matrices $\X$ and $\bm{W}$.  
\\ \ \\ 
\noindent \textbf{Upper bounding ${\sf Q}_{{\nabla^2 h}}(\W,\X)$}.  We bound each term in turn, though the argument is similar.  First, for any fixed $\bm{Q}_1$ and $\bm{Q}_2$ of rank at most $r$, define the process
\begin{align*}
    \mathcal{N}( \bm{Q}_1, \bm{Q}_2) := \frac{1}{L} \sum_{l} {\sf Tr}\bigg[ \bm{Q}_1 \nl \bm{Q}_2 \nl \bigg] - \mathbb{E} {\sf Tr}\bigg[ \bm{Q}_1 \nl \bm{Q}_2 \nl \bigg].
\end{align*}
It is straightforward to check that $\mathbb{E} {\sf Tr}\big[ \bm{Q}_1 \nl \bm{Q}_2 \nl \big] = \frac{\sigma^2}{2}\big[ {\sf Tr}(\bm{Q}_1 \bm{Q}_2\t) + {\sf Tr}(\bm{Q}_1) \tr(\bm{Q}_2) \big].$  Note that ${\sf Tr}(\bm{Q}_1 \nl \bm{Q}_2 \nl ) = {\sf Vec}(\nl)\t (\bm{Q}_1\t \otimes \bm{Q}_2 ) {\sf Vec}(\nl)$ (possibly up to transposes depending on vectorization convention, but this does not change the argument).  Therefore, up to potential rescaling due to repetition, the quantity $\mathcal{N}(\bm{Q}_1, \bm{Q}_2)$ is a quadratic form for the random variable ${\sf Vec}(\nl)$.  Thus, the Hanson-Wright inequality (Theorem 6.2.1 of \citet{vershynin_high-dimensional_2018}) gives
\begin{align*}
    \mathbb{P} \bigg\{ | \mathcal{N}(\bm{Q}_1, \bm{Q}_2 ) | \geq t \bigg\} &\leq 2 \exp\bigg\{ - cL \min\bigg( \frac{t^2}{\sigma^4 \| \bm{Q}_1 \otimes \bm{Q}_2\|_F^2}, \frac{t}{\sigma^2 \| \bm{Q}_1 \otimes \bm{Q}_2 \|} \bigg) \bigg\} \leq 2 \exp\bigg\{ - cL \min\big( \frac{t^2}{\sigma^4} , \frac{t}{\sigma^2} \big) \bigg\}.
\end{align*}
Take $t \asymp \frac{\sigma^2 nr}{L}$.  Then with probability at least $1 - \exp ( - c n r)$, 
\begin{align*}
    | \mathcal{N}(\bm{Q}_1,\bm{Q}_2) | \lesssim \frac{\sigma^2 n r}{L},
\end{align*}
where we have implicitly used the fact that $L \lesssim n$.  Taking a union bound over all $\bm{Q}_1$ and $\bm{Q}_2$ in an $\eps$-net shows that this bound holds uniformly over all $\bm{Q}_1$ and $\bm{Q}_2$ in the net.  Define 
\begin{align*}
    M := \sup_{\|\bm{Q}_1\|_F=1, \|\bm{Q}_2\|=1} | \mathcal{N}(\bm{Q}_1,\bm{Q}_2) |,
\end{align*}
where the supremum is over rank at most $r$ matrices. Let $\bm{\tilde Q}_1, \bm{\tilde Q}_2$ be points in the nets such that they are $\eps$-close to the maximizers $\bm{\check Q}_1, \bm{\check Q}_2$.  Then we have that
\begin{align*}
    M = | \mathcal{N}(\bm{\check Q}_1,\bm{ \check Q}_2) | &\leq | \mathcal{N}( \bm{\tilde Q}_1, \bm{\tilde Q}_2) | + |\mathcal{N}(\bm{\check Q}_1 - \bm{\tilde Q}_1, \bm{\check Q}_2 - \bm{\tilde Q}_2) | + |\mathcal{N}(\bm{\check Q}_1 - \bm{\tilde Q}_1, \bm{\tilde Q}_2) | + |\mathcal{N}( \bm{\tilde Q}_1, \bm{\check Q}_2 - \bm{\tilde Q}_2)| \\
    &\leq \frac{C \sigma^2 nr}{L} + 8\eps M,
\end{align*}
where we have implicitly used the fact that a matrix of rank at most $2r$ with Frobenius norm at most $\eps$ can be written as a sum of two rank at most $r$ matrices with Frobenius norm at most $\eps$.  Consequently, we have that $M \lesssim \frac{\sigma^2 n r}{L}$ by taking $\eps$ appropriately.  

We now apply this result to each term in the definition of ${\sf Q}_{{\nabla^2 h}}(\W,\X)$.  We have that
\begin{align*}
\bigg|    \frac{1}{L} \sum_{l}  \tr\bigg[ \W \X\t\X \W\t \nl \W \W\t \nl \bigg] - \frac{\sigma^2}{2}\bigg[ \tr\big[\W \X\t\X \W\t \W \W\t] +& \tr\big[ \W \X\t\X \W\t ] \tr\big[ \W \W\t \big] \bigg]\bigg| \\
&\lesssim \| \X \|_F^2 \frac{\sigma^2 nr}{L}.
\end{align*}
We also have that
\begin{align*}
    \frac{\sigma^2}{2}\big[ \tr\big[\W \X\t\X \W\t \W \W\t] + \tr\big[ \W \X\t\X \W\t ] \tr\big[ \W \W\t \big] = \frac{\sigma^2}{2} \big( \| \X\|_F^2 + \| \X \|_F^2 r \big) \lesssim \sigma^2 r  \| \X \|_F^2 \lesssim \frac{\sigma^2 n r}{L} \|\X \|_F^2.
\end{align*}
Similarly, the second term in \eqref{Qnabladef} satisfies
\begin{align*}
\bigg|    \frac{1}{L} \sum_{l} \tr\bigg[ \W \X\t \nl \X \W\t \nl \bigg] - \frac{\sigma^2}{2} \bigg[ \tr\big[ \W\X\t \X \W\t \big] + \tr\big[ \W \X\t \big] \tr\big[ \X \W\t\big] \bigg] \bigg| \lesssim \frac{\sigma^2 nr}{L} \|\X\|_F^2.
\end{align*}
The expectation satisfies $\tr\big[ \W \X \X\t \W \big] = \| \X \|_F^2,$ and $\tr[ \W \X\t ] = 0$ since $\X\t \W = 0$.  The remaining two terms are similar, with centering term equal to zero and $r \| \X \|_F^2$ respectively.  Thus, with probability at least $1- \exp( - c nr)$, we have that
\begin{align*}
    | {\sf Q}_{\nabla^2 h}(\W,\X) | \lesssim \frac{\sigma^2 nr}{L}\|\X\|_F^2.
\end{align*}
\noindent\textbf{Putting it all together.} Thus, we have shown that
\begin{align*}
    \nabla^2_{{\sf Riemannian}} h(\W)[\X,\X] \geq \| \X \|_F^2 \bigg( \frac{\lambda^2}{2} - C\frac{\lambda_{\max} \sigma \sqrt{nr}}{\sqrt{L}} - C\frac{\sigma^2 nr}{L} \bigg) \geq \|\X\|_F^2 \frac{\lambda^2}{4}
\end{align*}
which holds under the assumption that $\lambda/\sigma \geq C \frac{\kappa \sqrt{nr}}{\sqrt{L}}$. Similarly,
\begin{align*}
    \nabla^2_{{\sf Riemannian}} h(\W)[\X,\X] \leq  \| \X \|_F^2 \bigg( 2 \lambda_{\max}^2  + C\frac{\lambda_{\max} \sigma \sqrt{nr}}{\sqrt{L}} + C\frac{\sigma^2 nr}{L} \bigg) \leq  3 \lambda_{\max}^2 \|\X\|_F^2.
\end{align*}
This completes the proof.
\end{proof}

\subsubsection{Proof of \cref{lem:thirdorderguy}} \label{sec:thirdorderguyproof}

\begin{proof}
    First, we have that
    \begin{align*}
        \langle \bm{G}\one, \bm{G}\two \rangle &= \sum_{l} \sum_{l'} {\sf Tr}\bigg[ \bm{R}^{(l')} \U\t \bm{N}^{(l')} \per\pert \nl \U\U\t \nl \U (\bm{\mathcal{R} \mathcal{R}}\t)^{-2} \bigg] \\&= \sum_{j=1}^{r} \sum_{l} \sum_{l'} \bigg[\bm{R}^{(l')} \U\t \bm{N}^{(l')} \per\pert \nl \U\U\t \nl \U (\bm{\mathcal{R}\mathcal{R}}\t)^{-2} \bigg]_{jj}.
    \end{align*}
    We can view the summation above as a vector-valued (asymmetric) $U$-statistic computed from the matrix observations $\{\bm{N}^{(l)}\}_{l=1}^{L}$. We will bound it for fixed $j \in [r]$ and take a union bound.

    Define the vector
    \begin{align*}
        \bm{A}_{j}^{(l)} := \big(\per\pert \nl \U\U\t \nl \U(\bm{\mathcal{R} \mathcal{R}}\t)^{-2}  \big)_{\cdot j}.
    \end{align*}
    Then for fixed $j \in [r]$ we can write the inner product as
    \begin{align*}
        \sum_{l'} \bigg\langle \big(\bm{R}^{(l')} \U\t \bm{N}^{(l')} \big)_{j\cdot} , \sum_{l \neq l'} \bm{A}_{j}^{(l)} \bigg\rangle + \sum_{l} \big\langle \big( \bm{R}^{(l)} \U\t \nl \big)_{j\cdot}, \bm{A}_j^{(l)} \big\rangle =: K_1 + K_2.  
    \end{align*}
    We analyze each term separately.
    \begin{itemize}
        \item \textbf{The term $K_1$}:   
        By Theorem 3.4.1 of \citet{de_la_pena_decoupling_1999} (using an order-two decoupling inequality with an appropriate kernel $h$), 
        it suffices to consider the case where $\bm{N}^{(l')}$  is independent from  $\nl$.  Conditioning on $\sum_{l} \bm{A}_{j}^{(l)}$, $K_1$ is a sum of independent random variables.  Its $\psi_2$ norm satisfies
    \begin{align*}
        \bigg\|  \sum_{l'} \bigg\langle \big(\bm{R}^{(l')} \U\t \bm{N}^{(l')} \big)_{j\cdot} , \sum_{l} \bm{A}_{j}^{(l)} \bigg\rangle \bigg\|_{\psi_2}^2 \leq L \sigma^2 \max_{l'} \| \bm{R}^{(l')} \|^2 \bigg\| \sum_{l} \bm{A}_{j}^{(l)} \bigg\|^2.
    \end{align*}
    Therefore, by Hoeffding's inequality  it holds that
    \begin{align*}
        \bigg| \sum_{l'} \bigg\langle \big(\bm{R}^{(l')} \U\t \bm{N}^{(l')} \big)_{j\cdot} , \sum_{l} \bm{A}_{j}^{(l)} \bigg\rangle \bigg| \lesssim t \sigma \sqrt{L} \lambda_{\max}\bigg\| \sum_{l} \bm{A}_{j}^{(l)} \bigg\| 
    \end{align*}
    with probability at least $1 - \exp(- ct^2)$. 
Define the event  
       \begin{align*}
      \mathcal{E} := \bigcap_{j=1}^{r} \bigg\{  \bigg\| \sum_{l} \bm{A}^{(l)}_{j} \bigg\| \lesssim   \frac{\sigma^2 r \max \{\sqrt{nL},n\} }{\lambda^4 L^2} \bigg\}. \numberthis \label{event}
    \end{align*}
On the event \eqref{event} it holds that 
\begin{align*}
     \bigg| \sum_{l'} \bigg\langle \big(\bm{R}^{(l')} \U\t \bm{N}^{(l')} \big)_{j\cdot} , \sum_{l} \bm{A}_{j}^{(l)} \bigg\rangle \bigg| \lesssim t \sigma \sqrt{L} \lambda_{\max} \frac{\sigma^2 r \max\{ \sqrt{nL},n\}}{\lambda^4 L^2} &\asymp t \frac{\kappa \sigma^3 \sqrt{n} r}{\lambda^3 L} \max\{ 1, \sqrt{\frac{n}{L}}\}.
\end{align*}
Therefore, we have that 
\begin{align*}
    \p\bigg\{   \bigg| \sum_{l'} \bigg\langle \big(\bm{R}^{(l')} \U\t \bm{N}^{(l')} \big)_{j\cdot} &, \sum_{l} \bm{A}_{j}^{(l)} \bigg\rangle \bigg| > C t \frac{\kappa \sigma^3 \sqrt{n} r}{\lambda^3 L} \max\{ 1, \sqrt{\frac{n}{L}}\} \bigg\} \\
    &\leq \p\bigg\{  \bigg| \sum_{l'} \bigg\langle \big(\bm{R}^{(l')} \U\t \bm{N}^{(l')} \big)_{j\cdot} , \sum_{l} \bm{A}_{j}^{(l)} \bigg\rangle \bigg| > C t \frac{\kappa \sigma^3 \sqrt{n} r}{\lambda^3 L} \max\{ 1, \sqrt{\frac{n}{L}}\} \cap \mathcal{E} \bigg\} + \p ( \mathcal{E}^c ) \\
    &\leq 2 \exp( - c t^2) +  \p ( \mathcal{E}^c ).
\end{align*}
By \cref{lem:goodevent}, the event $\mathcal{E}$ holds with probability at least $1 - \exp( - c n)$ by \eqref{assertion3}.     
\item \textbf{The term $K_2$.} The term $K_2$ is directly a sum of independent random variables.  Moreover, we have that
\begin{align*}
    K_2 &= \sum_{k=1}^{r} \sum_{l=1}^{L} \bigg[ \rl \U\t \nl \per \pert \nl \U \bigg]_{jk} \bigg[ \U\t \nl \U \big( \R \big)^{-2} \bigg]_{kj}.
\end{align*}
For fixed $j$ and $k$ this is a sum of $L$ independent random variables with bounded  $\psi_{2/3}$ norm.
By Lemma 7 of \citet{hao_sparse_2020}, with probability at least $1 - \delta$ it holds that
\begin{align*}
    K_2 &\lesssim r \sqrt{L} \max_{j,k} \bigg\| \bigg[ \rl \U\t \nl \per \pert \nl \U \bigg]_{jk} \bigg[ \U\t \nl \U \big( \R \big)^{-2} \bigg]_{kj} \bigg\|_{\psi_{2/3}}  \sqrt{\log(1/\delta)} \\
    &\quad + r \max_{j,k} \bigg\| \bigg[ \rl \U\t \nl \per \pert \nl \U \bigg]_{jk} \bigg[ \U\t \nl \U \big( \R \big)^{-2} \bigg]_{kj} \bigg\|_{\psi_{2/3}} \big( \log(1/\delta) \big)^{3/2}
\end{align*}
Therefore, it suffices to bound the Orlicz $\psi_{2/3}$ norm of the random variable above.  We have
\begin{align*}
    \bigg\| &\bigg[ \rl \U\t \nl \per \pert \nl \U \bigg]_{jk} \bigg[ \U\t \nl \U \big( \R \big)^{-2} \bigg]_{kj} \bigg\|_{\psi_{2/3}}\\ &\lesssim \bigg\| \bigg[ \rl \U\t \nl \per \pert \nl \U \bigg]_{jk} \bigg\|_{\psi_1} \bigg\| \bigg[ \U\t \nl \U  \big( \R \big)^{-2} \bigg]_{kj} \bigg\|_{\psi_2} \\
    &\lesssim \sigma^3 n \lambda_{\max} \sqrt{r} \frac{1}{L^2 \lambda^4} \asymp \frac{\sigma^3 n \kappa \sqrt{r}}{L^2 \lambda^3}.
\end{align*}
Therefore, letting $t = \sqrt{\log(1/\delta)}$, by taking a union bound for each $k$, we have that with probability at least $1 - r\exp( - c t^2)$ it holds that
\begin{align*}
    K_2 &\lesssim t \frac{\kappa \sigma^3 n r^{3/2}}{\lambda^3 L^{3/2}} + t^3 \frac{\sigma^3 n \kappa r^{3/2}}{L^2 \lambda^3}.
\end{align*}
In particular, as long as $t\leq L^{1/4}$ it holds that
\begin{align*}
    K_2 \lesssim t \frac{\kappa \sigma^3 n r^{3/2}}{\lambda^3 L^{3/2}}.
\end{align*}
    \end{itemize}
Combining these two inequalities and bounding the summation over $j$ by $r$ times the maximum entry and taking a union bound shows that for all $t \leq L^{1/4}$, 
\begin{align*}
\bigg|    \langle \bm{G}\one, \bm{G}\two \rangle \bigg| &\lesssim  t  \frac{\kappa \sigma^3 \sqrt{n} r^2}{\lambda^3 L} \max\{ 1, \sqrt{\frac{n}{L}} \} + t \frac{\kappa \sigma^3 n \kappa r^{5/2}}{L^{3/2} \lambda^3} \\ 
&\lesssim  t  \frac{\kappa \sigma^3 \sqrt{n} r^{5/2}}{\lambda^3 L} \max\{ 1, \sqrt{\frac{n}{L}} \} 
\end{align*}
with probability at least $1 - r^2 \exp( - c t^2) - r \exp( - c n)$. 
\end{proof}

\subsubsection{Proof of \cref{lem:complicatedresidual2}}
\label{sec:complicatedresidual2proof}
\begin{proof}
    First, according to the same decomposition in \eqref{decomp}, we have that
    \begin{align*}
        \langle \bm{G}^{(1)}, T_3 \rangle &= \sum_{i=1}^9 \langle \bm{G}^{(1)}, J_i \rangle,
    \end{align*}
    where we  have that
   \begin{align*}
       J_1 &=      \sum_l \bm{U}_{\perp}\t \bm{N}^{(l)} \bm{U}_{\perp} \bm{U}_{\perp}\t \bm{\hat U} \bm{\hat U}\t \per \bm{N}^{(l)} \bm{U}_{\perp} \bm{U}_{\perp}\t \bm{\hat U} \mathcal{O}   \O\t \rhatinv\O \\
      J_2 &= \sum_l \bm{U}_{\perp}\t \bm{N}^{(l)} \bm{U}_{\perp} \bm{U}_{\perp}\t \bm{\hat U} \bm{\hat U}\t \U\U\t \bm{N}^{(l)} \bm{U}_{\perp} \bm{U}_{\perp}\t \bm{\hat U} \mathcal{O}   \O\t \rhatinv\O \\
     J_3 &= \sum_l \bm{U}_{\perp}\t \bm{N}^{(l)} \U \U\t \bm{\hat U} \bm{\hat U}\t \pert\nl \bm{U}_{\perp} \bm{U}_{\perp}\t \bm{\hat U} \mathcal{O}  \O\t \rhatinv\O\\
     J_4 &= \sum_l \bm{U}_{\perp}\t \bm{N}^{(l)} \U \U\t \bm{\hat U} \bm{\hat U}\t \U\U\t \bm{N}^{(l)} \bm{U}_{\perp} \bm{U}_{\perp}\t \bm{\hat U} \mathcal{O}  \O\t \rhatinv\O  \\
    J_5   &= \sum_l \bm{U}_{\perp}\t \bm{N}^{(l)} \per \bm{\hat U} \bm{\hat U}\t \per\per \bm{N}^{(l)} \U\U\t \bm{\hat U} \mathcal{O}  \O\t \rhatinv\O \\
      J_6     &=  \sum_l \bm{U}_{\perp}\t \bm{N}^{(l)} \per \bm{\hat U} \bm{\hat U}\t \U\U\t \bm{N}^{(l)} \U\U\t \bm{\hat U} \mathcal{O}   \O\t \rhatinv\O  \\
     J_7  &=  \sum_l \bm{U}_{\perp}\t \bm{N}^{(l)} \U\U\t \bm{\hat U} \bm{\hat U}\t \per\per \bm{N}^{(l)} \U\U\t \bm{\hat U} \mathcal{O}   \O\t \rhatinv\O  \\
      J_8   &=  \sum_l \bm{U}_{\perp}\t \bm{N}^{(l)} \U\U\t\bigg(  \bm{\hat U} \bm{\hat U}\t - \U\U\t \bigg) \U\U\t \bm{N}^{(l)} \U\U\t \bm{\hat U} \mathcal{O}  \O\t \rhatinv\O \\
        J_9  &=  \sum_l \bm{U}_{\perp}\t \bm{N}^{(l)} \U\U\t \bm{N}^{(l)} \U\U\t \big( \bm{\hat U} \mathcal{O} - \U)   \O\t \rhatinv\O.
   \end{align*}
We will bound each of these terms in turn.  Let $\mathcal{E}_{\bm{G}}$ denote the event that $\|\bm{G}^{(1)}\|_F \lesssim \frac{\kappa \sigma \sqrt{nr}}{\lambda \sqrt{L}}$, which can be shown to happen with probability at least $1 - \exp( - c n)$ through a similar argument as in the proof of \cref{lem:initialconcentration} in \cref{sec:initiallemproof}. For some of the subsequent terms we require additional technical arguments, and for some terms we simply apply Cauchy-Schwarz.  
\begin{itemize}
    \item \textbf{The term $\langle \bm{G}^{(1)},J_1\rangle$.} We have by Cauchy-Schwarz that
    \begin{align*}
        \bigg| \langle \bm{G}^{(1)}, J_1 \rangle \bigg| \leq \| \bm{G}^{(1)} \|_F \| J_1 \|_F.
    \end{align*}
   On the event $\mathcal{E}_{{\bm{G}}}$ it holds that $\| \bm{G}^{(1)} \|_F \lesssim \err$.  In addition, by the bound \eqref{J1}, we have that on the event $\mathcal{E}_{{\sf good}}$
    \begin{align*}
        \| J_1 \|_F \lesssim \frac{\sigma^2 nr \sqrt{L}}{L\lambda^2} \bigg( \err \bigg)^3 + \frac{\sigma^2}{\lambda^2} \bigg( \err \bigg)^3,
    \end{align*}
    which gives that
    \begin{align*}
          \bigg| \langle \bm{G}^{(1)}, J_1 \rangle \bigg|  &\lesssim \frac{\sigma^2 nr \sqrt{L}}{L\lambda^2} \bigg( \err \bigg)^4 + \frac{\sigma^2}{\lambda^2} \bigg( \err \bigg)^4. \numberthis\label{g1j1}
    \end{align*}
    \item \textbf{The term $\langle \bm{G}^{(1)},J_2 \rangle$.} Again we use Cauchy Schwarz and the bound \eqref{J2}, to obtain 
    \begin{align*}
        \bigg| \langle \bm{G}^{(1)}, J_2 \rangle \bigg| &\lesssim \| \bm{G}^{(1)} \|_F \| J_2 \|_F \\
        &\lesssim \err  \frac{\sigma^2 nr}{\sqrt{L}\lambda^2} \bigg( \err \bigg)^2 \asymp \bigg( \err \bigg)^3 \frac{\sigma^2 nr }{\sqrt{L}\lambda^2}, \numberthis \label{g1j2}
    \end{align*}
    which again holds on the event $\mathcal{E}_{\bm{G}} \cap \mathcal{E}_{{\sf good}}$.  
    \item \textbf{The term $\langle \bm{G}^{(1)},J_3\rangle $.}  Again we use Cauchy-Schwarz and the bound \cref{j3} to yield
    \begin{align*}
         \bigg| \langle \bm{G}^{(1)}, J_3 \rangle \bigg| &\lesssim \| \bm{G}^{(1)} \|_F \| J_3 \|_F \\
        &\lesssim \bigg( \err \bigg)^3 \frac{\sigma^2 nr }{\sqrt{L}\lambda^2} \numberthis \label{g1j3}
    \end{align*}
    which matches the bound \eqref{g1j2}.  
    \item \textbf{The term $\langle \bm{G}^{(1)},J_4\rangle $.} We require a more involved argument.  
   We first decompose 
   \begin{align*}
       \langle \bm{G}^{(1)}, J_4 \rangle &= \sum_{l} \sum_{l'} {\sf Tr} \bigg[ ( \bm{\mathcal{RR}}\t )^{-1} \bm{R}^{(l')} \U\t \bm{N}^{(l')} \per  \pert \nl \U \U\t\uhat\uhat\t \U\U\t \nl \per \pert\uhat \mathcal{O} \O\t \rhatinv \O  \bigg] \\
       &= \sum_{l \neq l'}{\sf Tr} \bigg[ ( \bm{\mathcal{RR}}\t )^{-1} \bm{R}^{(l')} \U\t \bm{N}^{(l')} \per  \pert \nl \U \U\t\uhat\uhat\t \U\U\t \nl \per \pert\uhat  \rhatinv \O   \bigg] \\
       &\quad + \sum_{l}{\sf Tr} \bigg[ ( \bm{\mathcal{RR}}\t )^{-1} \bm{R}^{(l)} \U\t \bm{N}^{(l)} \per  \pert \nl \U \U\t\uhat\uhat\t \U\U\t \nl \per \pert\uhat  \rhatinv \O   \bigg].
   \end{align*}
   The first term above represents the off-diagonal contribution, and the second term represents the diagonal contribution.  We bound both terms in turn.
\begin{itemize}
    \item \textbf{The off-diagonal.}   
   To bound this term we will use an $\eps$-net argument.    First, let $\bm{Q}_1 \in \mathbb{R}^{r\times r}$ be any symmetric positive semidefinite  matrix of Frobenius norm at most $1$, and let $\bm{Q}_2 \in \mathbb{R}^{n-r \times r}$ be another matrix of Frobenius norm at most one. 
We have that
\begin{align*}
    \sum_{l \neq l'}{\sf Tr} &\bigg[ ( \bm{\mathcal{RR}}\t )^{-1} \bm{R}^{(l')} \U\t \bm{N}^{(l')} \per  \pert \nl \U \bm{Q}_1 \U\t \nl \per \bm{Q}_2  \bigg] \\
    &= \sum_{j=1}^r \sum_{l \neq l'} \bigg\langle  \bigg( ( \bm{\mathcal{RR}}\t )^{-1} \bm{R}^{(l')} \U\t \bm{N}^{(l')}\bigg)_{j\cdot} , \bigg(\per \pert \nl \U \bm{Q}_1 \U\t \nl \per \bm{Q}_2\bigg)_{j\cdot} \bigg \rangle.
\end{align*}
Fix an index $j$.  
Define
    \begin{align*}
        \bm{A}_j^{(l)} := \bigg(\per \pert \nl \U \bm{Q}_1 \U\t \nl \per \bm{Q}_2\bigg)_{j\cdot}.
    \end{align*}
    By repeating the argument in the proof of \cref{lem:thirdorderguy} in \cref{sec:thirdorderguyproof}, through a decoupling argument we can show that with probability at least $1 - \exp ( - c t^2)$ it holds that
    \begin{align*}
        \bigg| \sum_{l'} \bigg\langle  \bigg( ( \bm{\mathcal{RR}}\t )^{-1} \bm{R}^{(l')} \U\t \bm{N}^{(l')}\bigg)_{j\cdot} , \sum_{l}\bigg(\per \pert \nl \U \bm{Q}_1 \U\t \nl \per \bm{Q}_2\bigg)_{j\cdot} \bigg\rangle  \bigg| \lesssim t \sigma \sqrt{L} \frac{\lambda_{\max}}{L \lambda^2} \bigg\| \sum_{l} \bm{A}_j^{(l)} \bigg\|.
    \end{align*}
    We note that by \cref{assertion2}, it holds that with probability at least $1 - \exp( - c nr)$, uniformly over $\bm{Q}_1$,
\begin{align*}
    \bigg\| \sum_{l} \pert \nl \U \bm{Q}_1 \U\t \nl \per \bigg\| &\leq    \bigg\| \pert  \bigg( \sum_{l} \nl \U \bm{Q}_1 \U\t \nl - \frac{\sigma^2 L}{2} {\sf Tr} \big( \U \bm{Q}_1 \bm{U}\t \big) \bigg) \per \bigg\|_F + \frac{\sigma^2 L}{2} {\sf Tr} \big( \U \bm{Q}_1 \bm{U}\t \big) \\
    &\lesssim \sigma^2 nr \sqrt{L} \| \U \bm{Q}_1 \bm{U}\t \|_F + \sigma^2 L \sqrt{r} \\
    &\lesssim \sigma^2 nr \sqrt{L} + \sigma^2 L \sqrt{r}.
\end{align*}
Thus, with probability at least $1 - \exp( - c t^2) - \exp( - c nr)$, uniformly over $\bm{Q}_1$ it holds that
\begin{align*}
         \bigg| \sum_{l'} \bigg\langle  \bigg( ( \bm{\mathcal{RR}}\t )^{-1} \bm{R}^{(l')} \U\t \bm{N}^{(l')}\bigg)_{j\cdot} , \sum_{l}\bigg(\per \pert \nl \U \bm{Q}_1 \U\t \nl \per \bm{Q}_2\bigg)_{j\cdot} \bigg\rangle  \bigg| &\lesssim t \sigma \sqrt{L} \frac{\lambda_{\max}}{L \lambda^2} \bigg( \sigma^2 nr \sqrt{L} + \sigma^2 L \sqrt{r} \bigg) \\
         &\asymp t \frac{\sigma^3 \kappa n r}{\lambda} + t \frac{\sigma^3 \kappa \sqrt{rL} }{\lambda  }.
\end{align*}
Thus, by taking a union bound over $\bm{Q}_2$ (and possibly increasing implicit constants in the proof of \cref{assertion2}), with probability at least $1 - \exp( - c nr)$ it holds that
\begin{align*}
    \bigg|  \sum_{l \neq l'} \bigg\langle  \bigg( ( \bm{\mathcal{RR}}\t )^{-2} \bm{R}^{(l')} \U\t \bm{N}^{(l')}\bigg)_{j\cdot} , \bigg(\per \pert \nl \U \bm{Q}_1 \U\t \nl \per \bm{Q}_2\bigg)_{j\cdot} \bigg \rangle\bigg| \lesssim  \frac{\sigma^3 \kappa n^{3/2} r^{3/2}}{\lambda} +  \frac{\sigma^3 \kappa \sqrt{nL} r}{\lambda}.
\end{align*}
In particular, since this term is linear in $\bm{Q}_i$, by \eqref{rhatinvbd2} it holds that with this same probability
\begin{align*}
    \sum_{j=1}^{r} \bigg| &\sum_{l\neq l} \bigg\langle  \bigg( ( \bm{\mathcal{RR}}\t )^{-1} \bm{R}^{(l')} \U\t \bm{N}^{(l')}\bigg)_{j\cdot} , \bigg(\per \pert \nl \U \U\t \uhat \uhat\t \U \U\t \nl \per \pert \uhat \rhatinv \mathcal{O}\bigg)_{j\cdot} \bigg \rangle\bigg| \\
    &\lesssim r \| \U\t \uhat \uhat\t \U \|_F \| \pert \uhat \|_F \bigg(  \frac{\sigma^3 \kappa n^{3/2} r^{3/2}}{\lambda^3 L} +  \frac{\sigma^3 \kappa \sqrt{n} r}{\lambda^3 \sqrt{L}} \bigg) \\
    &\lesssim r^{3/2} \err  \bigg( \frac{\sigma^3 \kappa n^{3/2} r^{3/2}}{\lambda^3 L} +  \frac{\sigma^3 \kappa \sqrt{n} r}{\lambda^3 \sqrt{L}}  \bigg) \\
    &\asymp \frac{\sigma^4 \kappa^2 n^2 r^{7/2}}{\lambda^4 L^{3/2}} + \frac{\sigma^4 \kappa^2 n r^{3}}{\lambda^4 L}.
\end{align*}
\item \textbf{The diagonal.} We now consider the diagonal term.  Instead of a net argument, we will consider the particular choices $\bm{Q}_1 = \U\t \uhat \uhat\t \U $ and $\bm{Q}_2 = \pert \uhat \rhatinv \mathcal{O}$.  Observe that
\begin{align*}
   \bigg| \sum_{l} {\sf Tr} &\bigg[ ( \bm{\mathcal{RR}}\t )^{-1} \bm{R}^{(l)} \U\t \bm{N}^{(l)}\per \pert \nl \U \bm{Q}_1 \U\t \nl \per \bm{Q}_2 \bigg] \bigg| \\
   &\leq L \max_{l} \| \nl \|^3 \| (\bm{\mathcal{RR}}\t )^{-1} \bm{R}^{(l)} \U\t \| \| \per \pert \| \| \U \bm{Q}_1 \U\t \|   \| \per \bm{Q}_2 \|_F \\
 &\lesssim L \sigma^3 n^{3/2} \frac{\kappa}{L \lambda} \err \frac{1}{L\lambda^2}\\
   &\asymp \frac{\sigma^4 n^2 \kappa^2 \sqrt{r}}{ \lambda^4 L^{3/2}}.
\end{align*}
where we used the fact that $\|\bm{Q}_2 \|_F \lesssim \err \frac{1}{L\lambda^2}$ when $\bm{Q}_2 = \pert \uhat \rhatinv \mathcal{O}$ when \eqref{rhatinvbd2} holds, together with the high probability bound $\|\nl\|\lesssim \sigma \sqrt{n}$.  These bounds hold cumulatively with probability at least $1 - \exp( - c n)$.
\end{itemize}
   Thus, we have shown that
   \begin{align*}
       \bigg| \langle \bm{G}^{(1)}, J_4 \rangle \bigg| &\lesssim \frac{\sigma^4 \kappa^2 n^2 r^{7/2}}{\lambda^4 L^{3/2}} + \frac{\sigma^4 \kappa^2 n r^{3}}{\lambda^4 L} +  \frac{\sigma^4 n^2 \kappa^2 \sqrt{r}}{ \lambda^4 L^{3/2}} \\
       &\lesssim \frac{\sigma^4 \kappa^2 n^2 r^{7/2}}{\lambda^4 L^{3/2}} + \frac{\sigma^4 \kappa^2 n r^{3}}{\lambda^4 L} \\
       &\asymp \frac{\sigma^4 \kappa^2 n^2 r^{7/2}}{\lambda^4 L^{3/2}},
       \numberthis \label{g1j4}
   \end{align*}
   since $L \lesssim n$. 
\item \textbf{The term $\langle \bm{G}^{(1)},J_5\rangle$}. Here again we use Cauchy-Schwarz and the bound \eqref{j5} to yield
\begin{align*}
\bigg| \langle \bm{G}^{(1)} , J_5 \rangle \bigg| \lesssim \err \frac{\sigma^2 nr^{3/2}}{\lambda^2 \sqrt{L}} \bigg( \err \bigg)^2 \asymp \frac{\sigma^5 \kappa^3 n^{5/2} r^3}{\lambda^5 L^2}. \numberthis \label{g1j5}
\end{align*}
\item \textbf{The term $\langle \bm{G}^{(1)},J_6\rangle.$} Again we apply Cauchy-Schwarz with the bound \eqref{j6} to yield
\begin{align*}
    \bigg| \langle \bm{G}^{(1)}, J_6 \rangle \bigg| &\lesssim \err \bigg[ \frac{\sigma^2 n r}{\lambda^2 \sqrt{L}} \bigg( \err \bigg)^2 + \frac{\sigma^2}{\lambda^2} \err \bigg] \\
    &\asymp \frac{\sigma^5 \kappa^3 n^{5/2} r^{5/2}}{\lambda^5 L^4} + \frac{\sigma^4 \kappa^2 nr}{\lambda^4 L}. \numberthis \label{g1j6}
\end{align*}
\item \textbf{The term $\langle \bm{G}^{(1)},J_7\rangle$}. By a completely identical argument to the term $\langle \bm{G}^{(1)}, J_4\rangle$, we can derive the exact same bound, yielding
\begin{align*}
    \bigg| \langle \bm{G}^{(1)}, J_7 \rangle \bigg| \lesssim   \frac{\sigma^4 \kappa^2 n^2 r^{7/2}}{\lambda^4 L^{3/2}}   
    \numberthis \label{g1j7}
\end{align*}
which holds with the same probability as in \eqref{g1j4}.
\item \textbf{The term $\langle \bm{G}^{(1)}, J_8\rangle$}.  The same argument as for the term $\langle \bm{G}^{(1)}, J_4\rangle$ applies verbatim with different choices of $\bm{Q}_1$ and $\bm{Q}_2$, yielding the bound
\begin{align*}
    \bigg| \langle \bm{G}^{(1)}, J_8 \rangle \bigg| \lesssim  \frac{\sigma^4 \kappa^2 n^2 r^{7/2}}{\lambda^4 L^{3/2}}. 
    \numberthis\label{g1j8}
\end{align*}
\item \textbf{The term $\langle \bm{G}^{(1)}, J_9\rangle$.} By Cauchy-Schwarz and \eqref{j9},
\begin{align*}
    \bigg| \langle \bm{G}^{(1)}, J_9 \rangle \bigg| &\lesssim \err \frac{\sigma^2 r n}{L\lambda^2} \err \\
    &\asymp  \frac{\sigma^4 \kappa^2 r^2 n^2}{\lambda^4 L^2}. \numberthis \label{g1j9}
\end{align*}
\end{itemize}
Combining \cref{g1j1,g1j2,g1j3,g1j4,g1j5,g1j6,g1j7,g1j8,g1j9}, we have that with probability at least $1 - \exp( - c n)$,
\begin{align*}
    \bigg| \sum_{i=1}^{9} \langle \bm{G}^{(1)}, J_i \rangle \bigg| &\lesssim \frac{\sigma^2 n r}{\sqrt{L}\lambda^2} \bigg( \err \bigg)^4 + \frac{\sigma^2}{\lambda^2} \bigg( \err \bigg)^4 + \bigg( \err \bigg)^3 \frac{\sigma^2 n r}{\sqrt{L}\lambda^2} \\
    &\quad + \frac{\sigma^4 \kappa^2 n^2 r^{7/2}}{\lambda^4 L^{3/2}} +  \frac{\sigma^5 \kappa^3 n^{5/2} r^3}{\lambda^5 L^2} + \frac{\sigma^5 \kappa^3 n^{5/2} r^{5/2}}{\lambda^5 L^4} + \frac{\sigma^4 \kappa^2 nr}{\lambda^4 L}+ \frac{\sigma^4 \kappa^2 r^2 n^2}{\lambda^4 L^2} \\
    &\lesssim  
    \frac{\sigma^4 \kappa^2 n^2 r^{7/2}}{\lambda^4 L^{3/2}}, 
\end{align*}
where we have used the assumption that $L \lesssim n$ and $\err \lesssim 1$. This completes the proof.  
\end{proof}

\subsection{Proof of \cref{thm:normality}}

\begin{proof}[Proof of \cref{thm:normality}]
First invoke \cref{thm:asymptotics} to have that
 \begin{align*}
     \| \pert \uhat \|_F^2 &=     {\sf Tr}\bigg( \bm{U}_{\perp} \bm{U}_{\perp}\t \bigg[ \sum_l \bm{N}^{(l)} \bm{U} \bm{R}^{(l)} \R \bigg] \bigg[ \sum_l \R  \bm{R}^{(l)}\bm{U} \bm{N}^{(l)} \bigg] \bigg) + {\sf Res},
 \end{align*}
where, with probability at least $1 - \exp( - c n) - \exp( - c \sqrt{L})$, 
\begin{align*}
    {\sf Res} \lesssim \frac{\sigma^3 \kappa^4 r^{5/2} n^{3/2}}{\lambda^3 L^{3/2}} + \frac{\sigma^4 \kappa^2 n^2 r^{7/2}}{\lambda^4 L^{3/2}}.
\end{align*}
We note that
\begin{align*}
    \sigma^2 \sqrt{n/2} \| \R \|_F \gtrsim \frac{\sigma^2 \sqrt{nr}}{\lambda^2 L}.
\end{align*}
Thus, with probability at least $1 - \exp( - c n) - \exp( - c \sqrt{L})$, it holds that
\begin{align*}
  \frac{ {\sf Res} }{\sigma^2 \sqrt{n/2} \| \R \|_F } \lesssim  
  \frac{\sigma\kappa^4 r^{2} n}{\lambda \sqrt{L}} + \frac{\sigma^2 \kappa^2 n^{3/2} r^3}{\lambda^2 \sqrt{L}}.
\end{align*}
We claim that this quantity is $o(1)$.  Indeed, the first quantity is $o(1)$ by assumption.  The second term satisfies
\begin{align*}
    \frac{\sigma^2 \kappa^2 n^{3/2} r^3}{\lambda^2 \sqrt{L}} \times \sqrt{\frac{n}{L}} \sqrt{\frac{L}{n}} \asymp \frac{\sigma^2 \kappa^2 n^2 r^3}{\lambda^2 L} \sqrt{\frac{L}{n}},
\end{align*}
which is $o(1)$ again with the assumption that $L \lesssim n$.

 It therefore suffices to focus on the leading-order term. We further decompose the leading-order term via
\begin{align*}
    {\sf Tr}\bigg( &\bm{U}_{\perp} \bm{U}_{\perp}\t \bigg[ \sum_l \bm{N}^{(l)} \bm{U} \bm{R}^{(l)} \R \bigg] \bigg[ \sum_l \R  \bm{R}^{(l)}\bm{U} \bm{N}^{(l)} \bigg] \bigg) \\
    &=  {\sf Tr}\bigg(  \bigg[ \sum_l \bm{N}^{(l)} \bm{U} \bm{R}^{(l)} \R \bigg] \bigg[ \sum_l \R  \bm{R}^{(l)}\bm{U} \bm{N}^{(l)} \bigg] \bigg) \\
    &\quad - {\sf Tr} \bigg( \bm{U} \bm{U}\t \bigg[ \sum_l \bm{N}^{(l)} \bm{U} \bm{R}^{(l)} \R \bigg] \bigg[ \sum_l \R  \bm{R}^{(l)}\bm{U} \bm{N}^{(l)} \bigg] \bigg).
\end{align*}
The following lemma bounds the second term.
\begin{lemma} \label{lem:extraterm}
Under the conditions of \cref{thm:normality}, with probability at least $1 - \exp( - c t^2)$ it holds that 
    \begin{align*}
        \| \U\t \sum_l \nl \sl \U \R \|_F \lesssim (\sqrt{r} + t) \frac{\sigma \kappa \sqrt{r}}{\sqrt{L} \lambda}
    \end{align*}
\end{lemma}
\begin{proof}
    See \cref{sec:extratermproof}.
\end{proof}
As a result, we have that with probability at least $1 - O(n^{-10})$, 
\begin{align*}
    \bigg| {\sf Tr} \bigg( \bm{U} \bm{U}\t \bigg[ \sum_l \bm{N}^{(l)} \bm{U} \bm{R}^{(l)} \R  \bigg] \bigg[ \sum_l \R  \bm{R}^{(l)}\bm{U} \bm{N}^{(l)} \bigg] \bigg) \bigg| &= \| \U\t \sum_l \nl \sl \U \R \|_F^2 \\
    &\lesssim \big( r + \log(n)  + \sqrt{r\log(n)}) \frac{\sigma^2 \kappa^2 r}{L \lambda^2} \\
    &\lesssim \frac{\sigma^2 \kappa^2 r^2\log(n)}{L \lambda^2}. \numberthis \label{extratermbound}
\end{align*}
We note that $\sigma^2 \sqrt{n/2}\|\R \|_F \gtrsim \sigma^2 \frac{\sqrt{nr}}{\lambda^2 L}$.  Therefore, on the event above it holds that 
\begin{align*}
   \frac{1}{\sigma^2 \sqrt{n/2} \|\R\|_F } \| \U\t \sum_l \nl \sl \U \R \|_F^2  &\lesssim \frac{\lambda^2 L}{\sigma^2 \sqrt{nr}} \frac{\sigma^2 \kappa^2 r^2\log(n)}{L \lambda^2} \\
   &\lesssim \frac{\kappa^2 r^{3/2} \log(n)}{\sqrt{n}} \ll 1,
\end{align*}
where the final inequality follows from the assumption
\eqref{variancelowerbound}.  Thus, it suffices to consider the asymptotics of the remaining term. 

The remaining term 
can be written via
\begin{align*}
{\sf Tr}\bigg(  \bigg[ \sum_l \bm{N}^{(l)} &\bm{U} \bm{R}^{(l)} \R \bigg] \bigg[ \sum_l \R  \bm{R}^{(l)}\bm{U} \bm{N}^{(l)} \bigg] \bigg)  \\
&= \sum_{i=1}^{n} \sum_{k_1,k_2=1}^{n} \sum_{l_1,l_2=1}^{L} \bm{N}_{ik_1}^{(l_1)} \bm{N}_{ik_2}^{(l_2)} \bigg[  \big( \bm{S}^{(l_1)} \bm{U} \R  \R \U\t \bm{S}^{(l_2)}  \big)_{k_1k_2} \bigg].
\end{align*}
We will apply the martingale central limit theorem to this quantity.  
Define the coefficient 
\begin{align*}
    c_{k_1k_2}^{l_1l_2} :=  \big( \bm{S}^{(l_1)} \bm{U} \R  \R \U\t \bm{S}^{(l_2)}  \big)_{k_1k_2} = e_{k_1}\t \U \bm{R}^{(l_1)} \R \R \bm{R}^{(l_2)} \U\t e_{k_2}.
\end{align*}
Define the random variable
\begin{align*}
    \bm{Z}_{ikl} := \begin{cases} \bm{N}^{(l)}_{kk} & i = k \\
    \sqrt{2} \bm{N}^{(l)}_{ik} & i \neq k,
    \end{cases}
\end{align*}
so that each $\bm{Z}_{ikl}$ has common variance $\sigma^2$ and fourth moment $\sigma_4^4$. 
Then the sum can be written as \renewcommand{\Z}{\bm{Z}}
\begin{align*}
    \sum_{i=1}^{n} &\sum_{k_1,k_2=1}^{n} \sum_{l_1,l_2} \bm{N}_{ik_1}^{(l_1)} \bm{N}_{ik_2}^{(l_2)} c_{k_1k_2}^{l_1l_2}   \\
   &= \sum_{i} \sum_{l} c_{ii}^{ll} \bm{Z}_{iil}^2 + \sum_{i} \sum_{l_1} \sum_{l_2 \neq l_1} c_{ii}^{l_1l_2} \bm{Z}_{iil_1} \bm{Z}_{iil_2}  + \frac{1}{\sqrt{2}} \sum_{i} \sum_{l} \sum_{k\neq i} c_{ik}^{ll} \Z_{iil} \Z_{ikl} + \frac{1}{\sqrt{2}} \sum_{i} \sum_{l_1} \sum_{l_2 \neq l_1} \sum_{k \neq i} c_{ik}^{l_1l_2} \Z_{iil_1} \Z_{ikl_2} \\
   &\quad + \frac{1}{\sqrt{2}} \sum_{i} \sum_{l_1} \sum_{l_2 \neq l_1} \sum_{k\neq i} c_{ki}^{l_1l_2} \Z_{ikl_1} \Z_{iil_2}  + \frac{1}{2} \sum_{i} \sum_{l} \sum_{k_1\neq i} \sum_{k_2 \neq i, k_1} c_{k_1k_2}^{ll} \Z_{ik_1l} \Z_{ik_2l} + \frac{1}{2} \sum_{i} \sum_{l} \sum_{k_1 \neq i} c_{k_1k_1}^{ll} \Z_{ik_1l}^2 \\
   &\quad + \frac{1}{2} \sum_{i} \sum_{l_1} \sum_{l_2 \neq l_1} \sum_{k_1 \neq i} \sum_{k_2 \neq i,k_1} c_{k_1k_2}^{l_1l_2} \Z_{ik_1l_1} \Z_{ik_2l_2}  + \frac{1}{2} \sum_{i} \sum_{l_1} \sum_{l_2 \neq l_1} \sum_{k_1 \neq i} c_{k_1k_1}^{l_1l_2} \Z_{ik_1l_1} \Z_{ik_1l_2} \\
 &= \sum_{i} \sum_{l} c_{ii}^{ll} \bm{Z}_{iil}^2 + 2 \sum_{i} \sum_{l_1} \sum_{l_2< l_1} c_{ii}^{l_1l_2} \bm{Z}_{iil_1} \bm{Z}_{iil_2} \\
    &\quad + \frac{1}{\sqrt{2}} \sum_{i} \sum_{l} \sum_{k< i} c_{ik}^{ll} \big( \Z_{iil} \Z_{ikl} + \Z_{kkl} \Z_{ikl} \big) \\
    &\quad + \frac{1}{\sqrt{2}} \sum_{i} \sum_{l_1} \sum_{l_2 < l_1} \sum_{k < i} c_{ik}^{l_1l_2} \big( \Z_{iil_1} \Z_{ikl_2} + \Z_{kkl_1} \Z_{ikl_2} \big) \\
   &\quad + \sum_{i} \sum_{l} \sum_{k_1 < i} \sum_{k_2 < k_1} c_{k_1k_2}^{ll} \Z_{ik_1l} \Z_{ik_2l} + 2 \sum_{i} \sum_{l_1} \sum_{l_2 < l_1} \sum_{k_1 < i} \sum_{k_2 < k_1} c_{k_1k_2}^{l_1l_2} \Z_{ik_1l_1} \Z_{ik_2l_2} \\
   &\quad + \frac{1}{2} \sum_{i} \sum_{k_1 < i} \sum_{l} ( c_{k_1k_1}^{ll}  + c_{ii}^{ll}) \Z_{ik_1l}^2 \\
   &\quad + \sum_{i} \sum_{k_1 < i} \sum_{l_1} \sum_{l_2 < l_1} (c_{k_1k_1}^{l_1l_2}  + c_{ii}^{l_1l_2}) \Z_{ik_1l_1} \Z_{ik_1l_2}. 
\end{align*} 
Let $M$ denote the random variable above.  We  now calculate the expectation of $M$.  Only the terms with $\bm{Z}_{i k_1 l_1}^2$ contribute, yielding
\begin{align*}
    \mathbb{E}( M) &= \sigma^2 \sum_{i} \sum_{l} c_{ii}^{ll} + \frac{\sigma^2}{2}\sum_{i} \sum_{k_1 < i} \sum_{l} ( c_{k_1k_1}^{ll} + c_{ii}^{ll}) \\
    &=  \frac{\sigma^2 (n-1)}{2} \sum_{i} \sum_{l} e_i\t \U \rl \R \R \rl  \U\t e_i \\
    &= \frac{\sigma^2 (n-1)}{2} {\sf Tr} \bigg( \sum_{l} (\rl)^2 \R \R \bigg) \\
    &=  \frac{\sigma^2 (n-1)}{2} {\sf Tr}\big( \R \big),
\end{align*}
which is nearly the centering term.  Note that
\begin{align*}
    \frac{\frac{\sigma^2}{2} {\sf Tr}\big( \R \big) }{\sigma^2\sqrt{n/2} \| \R \|_F} \lesssim \frac{\lambda^2 L}{\sqrt{nr}} \frac{\sigma^2 r}{L \lambda^2} &\asymp \sqrt{\frac{r}{n}} \ll 1.
\end{align*}
Therefore, if we can show that $\frac{M - \mathbb{E}(M)}{\sigma^2 \sqrt{n/2} \| \R \|_F} \to \mathcal{N}(0,1)$, then it also holds that $\frac{M - \frac{\sigma^2 n}{2} {\sf Tr}\big( \R \big)}{\sigma^2 \sqrt{n/2} \| \R \|_F} \to \mathcal{N}(0,1)$ since $r/n \to 0$ by \eqref{variancelowerbound}.  We therefore focus on the asymptotic normality of $M - \mathbb{E} M$.

We now calculate the variance.  
In order to do so, we first define a $\sigma$-algebra $\mathcal{F}_{t} := \sigma\{ \bm{z}_{1}, .. \bm{z}_t \}$, where $\bm{z}_{t} = \bm{Z}_{ik l}$ with $t(i,k,l)$ as follows.  First, define the set
\begin{align*}
    \mathcal{I} := \{(i,k,l): 1 \leq i \leq n, 1 \leq l \leq L, 1 \leq k \leq i\}.
\end{align*}
Then enumerate $\mathcal{I}$ via 
\begin{align*}
    (1,1,1), (1,2,1), ..., (1,i,1), (2,1,1), \dots, (n,i,1), (1,1,2), \dots,
\end{align*}
so that first $i$ is fixed, then $l$ is fixed, and then $k$ is fixed (so $k$ increases fastest and $i$ increases slowest).  
Let $M_T = M - \frac{\sigma^2 (n-1)}{2} {\sf Tr} (\R )$, which is mean-zero.  
From this definition it is evident that $M_T$ is a sum of martingale differences $Y_t := M_t - M_{t-1}$ with increments given by
\begin{align*}
    Y_t &=  A_t \big( \bm{Z}_t^2 - \sigma^2 \big) + B_t \bm{Z}_t,
\end{align*}
where for $t = t(i,k,l)$ we have
\begin{align*}
    A_{t} &= \begin{cases} c_{kk}^{ll} & i = k \\
    c_{kk}^{ll} + c_{ii}^{ll} & k < i; \end{cases} \\
    B_t &= \begin{cases}
        2 \sum_{l_2 < l} c_{ii}^{ll_2} \Z_{iil_2} + \frac{1}{\sqrt{2}} \sum_{k_2 <i} \bigg( c_{ik_2}^{ll_2} \Z_{ik_2l} + \sum_{l_2 < l} c_{ik_2}^{ll_2} \Z_{ik_2l_2} \bigg), & k = i; \\
     \sum_{l_2 < l} (c_{kk}^{ll_2} + c_{ii}^{ll_2}) \Z_{ikl_2} + \frac{1}{\sqrt{2}} \sum_{l_3 \geq l} c_{ik}^{l_3l} \Z_{kkl_3} + \sum_{k_2 < k}\bigg( c_{kk_2}^{ll} \Z_{ik_2l} + 2 \sum_{l_2 < l} c_{kk_2}^{ll_2} \Z_{ik_2l_2} \bigg) & k < i.
    \end{cases}
\end{align*}
Note that these coefficients are $\mathcal{F}_{t-1}$ measurable random variables since any appearance of $\bm{Z}_{i'k'l'}$ in the definition of $B_t$ has at least one index $t$ less than $t(i,k,l)$ with respect to this enumeration of $\mathcal{I}$.  
We compute the conditional variance:
\begin{align*}
    \mathbb{E} Y_t^2 | \mathcal{F}_{t-1} &= A_t^2 \big( \sigma_4^4 - \sigma^4) + B_t^2 \sigma^2 + 2 A_t B_t \sigma_3^3,
\end{align*}
where $\sigma_4^4$ and $\sigma_3^3$ are the noncentral moments of each $\bm{Z}_t$.  Furthermore, observe that $B_t$ is linear in previous $\bm{Z}_s$ (for $s < t$) and $A_t$ is deterministic.  Therefore, since cross-terms cancel,
\begin{align*}
    \sum_{t=1}^{T} \mathbb{E} Y_t^2 &= \sum_{t=1}^{T} A_t^2 \big( \sigma_4^4 - \sigma^4) +  \sigma^2 \mathbb{E} B_t^2.
\end{align*}
Moreover, we have that
\begin{align*}
    \mathbb{E} B_t^2 &= \begin{cases} 4 \sigma^2 \sum_{l_2 < l} (c_{ii}^{ll_2})^2 + \frac{\sigma^2 }{2} \sum_{k_2 < i} (c_{ik_2}^{ll_2})^2  + \frac{\sigma^2}{2} \sum_{k_2 < i} \sum_{l_2 < l} (c_{ik_2}^{ll_2})^2 & k = i; \\
\sigma^2 \sum_{l_2 < l} ( c_{kk}^{ll_2} + c_{ii}^{ll_2})^2 + \frac{\sigma^2}{2} \sum_{l_3 \geq l} (c_{ik}^{l_3l})^2 + \sigma^2 \sum_{k_2 < k} (c_{kk_2}^{ll})^2 + 2\sigma^2 \sum_{k_2 < k} \sum_{l_2 < l} (c_{kk_2}^{ll_2})^2 & k <i.
    \end{cases}
\end{align*}
Therefore,
\begin{align*}
    \sum_{t=1}^{T} \mathbb{E} Y_t^2 &= \sum_{i=1}^{n} \sum_{l=1}^{L} (\sigma_4^4 - \sigma^4) (c_{ii}^{ll})^2 + \sum_{i=1}^{n} \sum_{k < i} \sum_{l=1}^{L} (\sigma_4^4 - \sigma^4)( c_{kk}^{ll} + c_{ii}^{ll})^2 \\
&\quad + 4 \sigma^4 \sum_{i=1}^{n} \sum_{l=1}^{L} \sum_{l_2 < l} (c_{ii}^{ll_2})^2 + \frac{\sigma^4 }{2}\sum_{i=1}^{n} \sum_{l=1}^{L} \sum_{k_2 < i} (c_{ik_2}^{ll_2})^2  + \frac{\sigma^4}{2} \sum_{i=1}^{n}  \sum_{l=1}^{L}\sum_{k_2 < i} \sum_{l_2 < l} (c_{ik_2}^{ll_2})^2\\
&\quad + \sigma^4 \sum_{i=1}^{n}  \sum_{l=1}^{L} \sum_{k < i} \sum_{l_2 < l} ( c_{kk}^{ll_2} + c_{ii}^{ll_2})^2 + \frac{\sigma^4}{2} \sum_{i=1}^{n}  \sum_{l=1}^{L} \sum_{k < i}  \sum_{l_3 \geq l} (c_{ik}^{l_3l})^2 + \sigma^4 \sum_{i=1}^{n}  \sum_{l=1}^{L} \sum_{k < i}  \sum_{k_2 < k} (c_{kk_2}^{ll})^2 \\
&\quad + 2\sigma^4 \sum_{i=1}^{n}  \sum_{l=1}^{L} \sum_{k < i}  \sum_{k_2 < k} \sum_{l_2 < l} (c_{kk_2}^{ll_2})^2\\
&= (n - 1)\sum_{i,l} (\sigma^4_4 - \sigma^4) (c_{ii}^{ll})^2 +    \sum_{i=1}^{n} \sum_{k} \sum_{l=1}^{L} (\sigma_4^4-\sigma^4) c_{ii}^{ll} c_{kk}^{ll}  \\
&\quad + 2\sigma^4 \sum_{i} \sum_{l_1} \sum_{l_2} (c_{ii}^{l_1l_2})^2 - 2\sigma^4 \sum_{i} \sum_{l} (c_{ii}^{ll})^2 \\
&\quad + \frac{\sigma^4}{2} \sum_{i} \sum_{l} \sum_{k} (c_{ik}^{ll})^2 - \frac{\sigma^4}{2} \sum_{i=1}^{n} \sum_{l} (c_{ii}^{ll})^2 \\
&\quad + \frac{\sigma^4}{8} \sum_{i} \sum_{k} \sum_{l_1} \sum_{l_2} (c_{ik}^{l_1l_2})^2 - \frac{\sigma^4}{8} \sum_{i} \sum_{k} \sum_{l} (c_{ik}^{ll})^2 \\
&\quad + \frac{\sigma^4}{4} \sum_{i} \sum_{l} \sum_{k_1} \sum_{k_2} (c_{k_1k_2}^{ll})^2 - \frac{\sigma^4}{2} \sum_{i} \sum_{l} \sum_{k_1} (c_{k_1i}^{ll})^2 +\frac{\sigma^4}{4} \sum_{i} \sum_{l} (c_{ii}^{ll})^2 \\
&\quad + \frac{\sigma^4}{2} \sum_{i,k_1,k_2,l_1,l_2} (c_{k_1k_2}^{l_1l_2})^2 - \sigma^4 \sum_{i,k_1,l_1,l_2} (c_{k_1i}^{l_1l_2})^2 + \frac{\sigma^4}{2} \sum_{i,l_1,l_2} (c_{ii}^{l_1l_2})^2 \\
&\quad + \frac{\sigma^4}{4} \sum_{i,k_1,l_1,l_2}( c_{k_1k_1}^{l_1l_2})^2 +\frac{\sigma^4}{4} \sum_{i,k_1,l_1,l_2}  c_{ii}^{l_1l_2})^2  + \frac{\sigma^4}{2} \sum_{i,k_1,l_1,l_2} c_{k_1k_1}^{l_1l_2}c_{ii}^{l_1l_2} \\
&\quad - \frac{\sigma^4}{4} \sum_{i,k_1,l} (c_{k_1k_1}^{ll})^2 - \frac{\sigma^4}{4} \sum_{i,k_1,l} (c_{ii}^{ll})^2 - \frac{\sigma^4}{2}  \sum_{i,k_1,l} c_{k_1k_1}^{ll}c_{ii}^{ll}\numberthis \label{tosimplify}
\end{align*}
We now simplify \eqref{tosimplify}.  First, we note that we have the following identity: 
\begin{align*}
    \sum_{k_1} \sum_{k_2} ( c^{l_1l_2}_{k_1k_2} )^2 &= \sum_{k_1} \sum_{k_2} \Bigg[ e_{k_1}\t \U \bm{R}^{(l_1)} \R \R \bm{R}^{(l_2)} \U\t e_{k_2} \Bigg]^2 \\
    &= \sum_{j_1,j_2,j_3,j_4}\bigg(\bm{R}^{(l_1)} \R \R \bm{R}^{(l_2)}\bigg)_{j_1j_2}\bigg(\bm{R}^{(l_1)} \R \R \bm{R}^{(l_2)}\bigg)_{j_3j_4} \\
    &\quad \times \sum_{k_1} \sum_{k_2} \U_{k_1j_1} \U_{k_1j_3} \U_{k_2j_2} \U_{k_2j_4} \\
    &= \sum_{j_1,j_2,j_3,j_4} \bigg(\bm{R}^{(l_1)} \R\R  \bm{R}^{(l_2)}\bigg)_{j_1j_2}\bigg(\bm{R}^{(l_1)} \R \R \bm{R}^{(l_2)}\bigg)_{j_3j_4} \mathbb{I}\{j_1 = j_3\} \mathbb{I}\{ j_2 = j_4\} \\
    &= \sum_{j_1,j_2} \bigg( \bm{R}^{(l_1)} \R \R \bm{R}^{(l_2)}\bigg)_{j_1j_2}^2 \\
    &= \bigg\| \bm{R}^{(l_1)} \R\R  \bm{R}^{(l_2)} \bigg\|_F^2,
\end{align*}
where we have used the orthogonality of $\U$ above.  
Therefore, \eqref{tosimplify} can be written as
\begin{align*}
    \eqref{tosimplify} &= (n-1) (\sigma_4^2 - \sigma^4) \sum_{i,l} (c_{ii}^{ll})^2 +(\sigma_4^2 - \sigma^4) \sum_{i,k,l} c_{ii}^{ll} c_{kk}^{ll} + \sigma^4 \bigg(-2 - \frac{1}{2} + \frac{1}{4} - \frac{n}{4} - \frac{n}{4} \bigg)    \sum_{i,l} (c_{ii}^{ll})^2 \\
    &\quad + \sigma^4 \bigg( \frac{n}{2} + 2 + \frac{1}{2} \bigg) \sum_{i,l_1,l_2} (c_{ii}^{l_1l_2})^2 \\
    &\quad + \frac{\sigma^4}{2} \sum_{i,k_1,l_1,l_2} c_{k_1k_1}^{l_1l_2} c_{ii}^{l_1l_2} - \frac{\sigma^4}{2} \sum_{k,k_1,l} c_{k_1k_1}^{ll} c_{kk}^{ll} \\
    &\quad + \sigma^4 \bigg[ \frac{n}{2} - 1 + \frac{1}{8} \bigg] \| \R \|_F^2 + \sigma^4 \bigg[ \frac{n}{4} - \frac{1}{8} \bigg] \sum_l \| \rl \R \R \rl \|_F^2 \\
    &= \frac{\sigma^4 n}{2} \bigg[ 1 + \frac{2}{n} \bigg] \| \R \|_F^2 + \frac{\sigma^4 n}{2} \bigg[ \frac{1}{2} - \frac{1}{4n} \bigg]  \sum_l \| \rl \R \R \rl \|_F^2 \\
    &\quad +\bigg[ (n-1)(\sigma_4^4 - \sigma^4) - \sigma^4 \big( \frac{n}{2} - 2.25\big) \bigg] \sum_{i,l} (c_{ii}^{ll})^2 \\
    &\quad + n \sigma^4 \bigg( \frac{1}{2} + \frac{2.5}{n} \bigg) \sum_{i,l_1,l_2} (c_{ii}^{l_1l_2})^2 + \frac{\sigma^4}{2} \sum_{i,k_1,l_1,l_2} c_{k_1k_1}^{l_1l_2} c_{ii}^{l_1l_2} - \frac{\sigma^4}{2} \sum_{k,k_1,l} c_{k_1k_1}^{ll} c_{ii}^{ll}.\end{align*}
We will show that under the conditions of \cref{thm:normality} that this equals $\frac{\sigma^4 n}{2} \| \R \|_F^2 ( 1+ o(1))$.  Indeed, we have that
\begin{align*}
    \bigg| \eqref{tosimplify} - \frac{\sigma^4 n}{2} \| \R \|_F^2 \bigg| &\lesssim \sigma^4 \| \R \|_F^2 + C \sigma^4 n \sum_{l} \| \rl \R \R \rl\|_F^2 + C n \sigma^4 \sum_{i,l} (c_{ii}^{ll})^2  \\
    &\quad + C n \sigma^4 \sum_{i,l_1,l_2} (c_{ii}^{l_1l_2})^2 + C \sigma^4 \sum_{i,k_1,l_1,l_2} c_{k_1k_1}^{l_1l_2} c_{ii}^{l_1l_2} + C \sigma^4 \sum_{k,k_1,l} c_{k_1k_1}^{ll} c_{ii}^{k_1k_1}.
\end{align*}
Note that $\frac{\sigma^4 n}{2} \| \R \|_F^2 \gtrsim \frac{\sigma^4 n r }{L^2 \lambda^4}$.  We also have that
\begin{align*}
    \sigma^4 n\sum_{i,l_1,l_2} (c_{ii}^{l_1l_2})^2 &= \sum_{i,l_1,l_2} \bigg( e_i\t \U \bm{R}^{(l_1)} \R \R \bm{R}^{(l_2)} \U\t e_i \bigg)^2 \\
    &\leq\sigma^4 n\max_{i} \| e_i\t \U \|^2 \sum_{i,l_1,l_2} \| e_i\t \U \|^2 \| \bm{R}^{(l_1)} \R \R \bm{R}^{(l_2)} \|_F^2 \\
    &\leq \| \U \|_{2,\infty}^2 \frac{ \sigma^4 n \kappa^4 r }{L^2 \lambda^4} \numberthis \label{ciboundguy} \\
    &\ll \frac{\sigma^4 n r}{L^2 \lambda^4}
\end{align*}
as long as $\|\U\|_{2,\infty}^2 \ll \frac{1}{\kappa^4}$.   The same arguments hold for each of the other terms containing $c_{ii}^{l_1l_2}$ or $c_{ii}^{ll}$ as well.  In addition,
\begin{align*}
    \sigma^4 n \sum_{l} \| \rl \R \R \rl \|_F^2 &\lesssim  \frac{\sigma^4 n r \kappa^4}{L^3 \lambda^4 } \ll \frac{\sigma^4 nr}{L^2 \lambda^4}
\end{align*}
as long as $\frac{\kappa^4}{L} \ll 1$.  Thus,
\begin{align*}
     \sum_{t=1}^{T} \mathbb{E} Y_t^2 = \frac{\sigma^4n}{2} \| \R \|_F^2 \bigg( 1 + o(1) \bigg).
\end{align*}
 For simplicity, define $$s^2 = \sum_{t} \mathbb{E} Y_t^2 = \frac{\sigma^4 n}{2} \| \R \|_F^2\big(1 + o(1) \big).$$ 
It suffices to prove the asymptotic normality of $\frac{M_t - \mathbb{E} M_t}{s}$.  We apply the martingale central limit theorem (Lemma 9.12 of \citet{bai_spectral_2010}).  We need to show that \begin{align*}
    \frac{\sum_{t=1}^{T} \mathbb{E} Y_t^2 | \mathcal{F}_{t-1} }{s^2} \xrightarrow{p} 1 ; \numberthis \label{martingale1} \\
    \frac{\sum_{t=1}^{T} \mathbb{E} Y_t^2 \mathbb{I}\{Y_t^2 \geq \eps s^2\}}{s^2} \to 0 \numberthis \label{martingale2}
\end{align*}
 for all $\eps > 0$. To prove \eqref{martingale1}, we  have that
\begin{align*}
    \sum_{t=1}^{T} \mathbb{E} Y_t^2 | \mathcal{F}_{t-1} 
    &= \sum_{t} A_t^2( \sigma_4^4 - \sigma^4) + B_t^2 \sigma^2 + 2 A_t B_t \sigma_3^3 \\
    &= (\sigma_4^4-\sigma^4) \sum_{i,l} (c_{ii}^{ll})^2 + \sigma^2 \sum_{i,l} \Bigg[  2 \sum_{l_2 < l} c_{ii}^{ll_2} \Z_{iil_2} + \frac{1}{\sqrt{2}} \sum_{k_2 <i} \bigg( c_{ik_2}^{ll_2} \Z_{ik_2l} + \sum_{l_2 < l} c_{ik_2}^{ll_2} \Z_{ik_2l_2} \bigg) \Bigg]^2 \\
    &\quad + 2 \sigma_3^3 \sum_{i,l} c_{ii}^{ll} \Bigg[  2 \sum_{l_2 < l} c_{ii}^{ll_2} \Z_{iil_2} + \frac{1}{\sqrt{2}} \sum_{k_2 <i} \bigg( c_{ik_2}^{ll_2} \Z_{ik_2l} + \sum_{l_2 < l} c_{ik_2}^{ll_2} \Z_{ik_2l_2} \bigg)\Bigg] \\
    &\quad + \sum_{i,l} \sum_{k<i} \big[ c_{kk}^{ll} + c_{ii}^{ll} \big]^2 (\sigma_4^4-\sigma^4 ) \\
    &\quad + \sigma^2 \sum_{i,l} \sum_{k <i} \Bigg[  \sum_{l_2 < l} (c_{kk}^{ll_2} + c_{ii}^{ll_2}) \Z_{ikl_2} + \frac{1}{\sqrt{2}} \sum_{l_3 \geq l} c_{ik}^{l_3l} \Z_{kkl_3} + \sum_{k_2 < k}\bigg( c_{kk_2}^{ll} \Z_{ik_2l} + 2 \sum_{l_2 < l} c_{kk_2}^{ll_2} \Z_{ik_2l_2} \bigg)\Bigg]^2 \\
    &\quad + 2 \sigma_3^3 \sum_{i,l} \sum_{k < i} [ c_{kk}^{ll} + c_{ii}^{ll}] \Bigg[  \sum_{l_2 < l} (c_{kk}^{ll_2} + c_{ii}^{ll_2}) \Z_{ikl_2} + \frac{1}{\sqrt{2}} \sum_{l_3 \geq l} c_{ik}^{l_3l} \Z_{kkl_3} + \sum_{k_2 < k}\bigg( c_{kk_2}^{ll} \Z_{ik_2l} + 2 \sum_{l_2 < l} c_{kk_2}^{ll_2} \Z_{ik_2l_2} \bigg)\Bigg].
\end{align*}
Define
\begin{align*}
    T_1 &:= (\sigma_4^4 - \sigma^4) \sum_{i,l} (c_{ii}^{ll})^2  + \sum_{i,l} \sum_{k<i} \big[ c_{kk}^{ll} + c_{ii}^{ll} \big]^2 (\sigma_4^4 - \sigma^4 ) ; \\
    T_2 &:= 2 \sigma_3^3 \sum_{i,l} c_{ii}^{ll} \Bigg[  2 \sum_{l_2 < l} c_{ii}^{ll_2} \Z_{iil_2} + \frac{1}{\sqrt{2}} \sum_{k_2 <i} \bigg( c_{ik_2}^{ll_2} \Z_{ik_2l} + \sum_{l_2 < l} c_{ik_2}^{ll_2} \Z_{ik_2l_2} \bigg)\Bigg] \\
    &\quad + 2 \sigma_3^3 \sum_{i,l} \sum_{k < i} [ c_{kk}^{ll} + c_{ii}^{ll}] \Bigg[  \sum_{l_2 < l} (c_{kk}^{ll_2} + c_{ii}^{ll_2}) \Z_{ikl_2} + \frac{1}{\sqrt{2}} \sum_{l_3 \geq l} c_{ik}^{l_3l} \Z_{kkl_3} + \sum_{k_2 < k}\bigg( c_{kk_2}^{ll} \Z_{ik_2l} + 2 \sum_{l_2 < l} c_{kk_2}^{ll_2} \Z_{ik_2l_2} \bigg)\Bigg] \\
    T_3 &:=    \sigma^2 \sum_{i,l} \Bigg[  2 \sum_{l_2 < l} c_{ii}^{ll_2} \Z_{iil_2} + \frac{1}{\sqrt{2}} \sum_{k_2 <i} \bigg( c_{ik_2}^{ll_2} \Z_{ik_2l} + \sum_{l_2 < l} c_{ik_2}^{ll_2} \Z_{ik_2l_2} \bigg) \Bigg]^2   \\
    &\quad + \sigma^2\sum_{i,l} \sum_{k <i} \Bigg[  \sum_{l_2 < l} (c_{kk}^{ll_2} + c_{ii}^{ll_2}) \Z_{ikl_2} + \frac{1}{\sqrt{2}} \sum_{l_3 \geq l} c_{ik}^{l_3l} \Z_{kkl_3} + \sum_{k_2 < k}\bigg( c_{kk_2}^{ll} \Z_{ik_2l} + 2 \sum_{l_2 < l} c_{kk_2}^{ll_2} \Z_{ik_2l_2} \bigg)\Bigg]^2.
\end{align*}
 We analyze $T_1,T_2$, and $T_3$ separately. 
\begin{itemize}
    \item \textbf{The term $T_1$.} By a similar argument to \eqref{ciboundguy} it holds that $T_1 \ll s^2$. 
    \item \textbf{The term $T_2$.}  As for $T_2$, we note that it is mean-zero, and, furthermore,
\begin{align*}
    \mathbb{E} T_2^2 &\leq 8 \sigma_3^6 \mathbb{E} \Bigg[ \sum_{i,l} c_{ii}^{ll} \bigg[2 \sum_{l_2 < l} c_{ii}^{ll_2} \bm{Z}_{iil_2} + \frac{1}{\sqrt{2}} \sum_{k_2< i } c_{ik_2}^{ll_2} \bm{Z}_{ik_2l} +\frac{1}{\sqrt{2}} \sum_{k_2< i } \sum_{l_2<l} c_{ik_2}^{ll_2} \bm{Z}_{ik_2l_2} \bigg] \Bigg]^2 \\
    &\quad + 8 \sigma_3^6 \mathbb{E} \Bigg[ \sum_{i,l} \sum_{k < i} [ c_{kk}^{ll} + c_{ii}^{ll} ] \Bigg( \sum_{l_2 < l} (c_{kk}^{ll_2} + c_{ii}^{ll_2}) \bm{Z}_{ikl_2} + \frac{1}{\sqrt{2}} \sum_{l_3 \geq l} c_{ik}^{l_3 l} \bm{Z}_{kkl_3} + \sum_{k_2 < k} c_{kk_2}^{ll} \bm{Z}_{ik_2l} +2 \sum_{k_2 < k} \sum_{l_2 < l} c_{kk_2}^{ll_2} \bm{Z}_{ik_2l_2} \Bigg) \Bigg]^2 \\
    &\leq 8 \sigma_3^6 \sigma^2 \Bigg(  \sum_{i,l} \sum_{l_2 < l} 4 (c_{ii}^{ll})^2 (c_{ii}^{ll_2})^2 + \sum_{i,l} \sum_{k_2 < i} \frac{1}{2} (c_{ii}^{ll})^2 (c_{ik_2}^{ll_2})^2 + \sum_{i,l} \sum_{k_2 < i} \sum_{l_2 < l} \frac{1}{2} (c_{ii}^{ll})^2 (c_{ik_2}^{ll_2} )^2 \Bigg) \\
    &\quad + 8 \sigma_3^6 \sigma^2 \Bigg[ \sum_{i,l}\sum_{k < i} \sum_{l_2 < l} (c_{kk}^{ll} + c_{ii}^{ll})^2 (c_{kk}^{ll_2} + c_{ii}^{ll_2})^2 + \sum_{i,l} \sum_{k < i} \sum_{l_3 \geq l} \frac{1}{2} (c_{kk}^{ll} + c_{ii}^{ll})^2 (c_{ik}^{l_3l})^2 \Bigg] \\
    &\quad + 8 \sigma_3^6 \sigma^2 \Bigg[ \sum_{i,l} \sum_{k < i} \sum_{k_2 < k}  (c_{kk}^{ll} + c_{ii}^{ll})^2 (c_{kk_2}^{ll})^2  + 4 \sum_{i,l} \sum_{k < i}  \sum_{k_2 < k} \sum_{l_2 < l} (c_{kk}^{ll} + c_{ii}^{ll})^2 (c_{kk_2}^{ll_2})^2 \Bigg] \\
    &\lesssim  \sigma^8 \sum_{i,l,k,k_2,l_2} (c_{ii}^{ll})^2 (c_{kk_2}^{ll_2})^2 \\
    &\lesssim n\sigma^8 \max_{i,l} (c_{ii}^{ll})^2 \sum_{l,k,k_2l_2} (c_{kk_2}^{ll_2})^2 \\
    &\lesssim n \sigma^8 \max_{i,l} (c_{ii}^{ll})^2 \| \R \|_F^2.
\end{align*}
We note that Cauchy-Schwarz implies that
\begin{align*}
    | c_{kk_1}^{l_1l_2}| \leq \| \U \|_{2,\infty}^2 \frac{\kappa^2 r}{L^2 \lambda^2}.  
\end{align*}
Thus, 
\begin{align*}
    n \sigma^8 \max_{i,l} (c_{ii}^{ll})^2 \| \R \|_F^2 &\lesssim \frac{n \sigma^8 r }{L^2 \lambda^4 } \|  \U\|_{2,\infty}^4 \frac{\kappa^4 r^2}{L^4 \lambda^4} \asymp  \|\U\|_{2,\infty}^4 \frac{n \sigma^8 r^3 \kappa^4}{L^6 \lambda^8}.
\end{align*}
As a result, Chebyshev's inequality implies 
\begin{align*}
    \p\bigg\{ |T_2| \geq \eps s^2 \bigg\} &\lesssim \|\U\|_{2,\infty}^4 \frac{n \sigma^8 r^3 \kappa^4}{L^6 \lambda^8} \frac{1}{\eps^2 s^4} \\
    &\lesssim \|\U\|_{2,\infty}^4 \frac{n \sigma^8 r^3 \kappa^4}{L^6 \lambda^8} \frac{1}{\eps^2} \frac{L^4 \lambda^8}{r^2 n^2 \sigma^8} \\
    &\lesssim \| \U\|_{2,\infty}^4 \frac{r \kappa^4}{L^2 n \eps^2} \to 0
\end{align*}
by \cref{variancelowerbound}, 
which implies that $T_2 \to 0$ in probability.
\item \textbf{The term $T_3$.}  From the previous calculations, we note that $T_3$ has expected value $\frac{\sigma^4 n}{2} \| \R \|_F^2\big(1 + o(1) \big)$.  We therefore prove the convergence in probability to its expected value.  We decompose $T_3$ via:
\begin{align*}
T_3
&=
4\sigma^2
\sum_{i,l}\sum_{l_2<l}
\big(c_{ii}^{ll_2}\big)^2 \bm{Z}_{iil_2}^2
+
\frac{\sigma^2}{2}
\sum_{i,l}\sum_{k_2<i}
\big(c_{ik_2}^{ll}\big)^2 \bm{Z}_{ik_2l}^2
+
\frac{\sigma^2}{2}
\sum_{i,l}\sum_{k_2<i}\sum_{l_2<l}
\big(c_{ik_2}^{ll_2}\big)^2 \bm{Z}_{ik_2l_2}^2
\\
&\quad
+\sigma^2
\sum_{i,l}\sum_{k<i}\sum_{l_2<l}
\big(c_{kk}^{ll_2}+c_{ii}^{ll_2}\big)^2 \bm{Z}_{ikl_2}^2
+
 \sigma^2 \frac12
\sum_{i,l}\sum_{k<i}\sum_{l_3\ge l}
\big(c_{ik}^{l_3l}\big)^2 \bm{Z}_{kkl_3}^2
\\
&\quad
+
\sigma^2\sum_{i,l}\sum_{k<i}\sum_{k_2<k}
\big(c_{kk_2}^{ll}\big)^2 \bm{Z}_{ik_2l}^2
+
4\sigma^2\sum_{i,l}\sum_{k<i}\sum_{k_2<k}\sum_{l_2<l}
\big(c_{kk_2}^{ll_2}\big)^2 \bm{Z}_{ik_2l_2}^2
\\
&\quad
+
\frac{4\sigma^2}{\sqrt{2}}
\sum_{i,l}\sum_{l_2<l}\sum_{k_2<i}
c_{ii}^{ll_2}c_{ik_2}^{ll}
\bm{Z}_{iil_2}\bm{Z}_{ik_2l}
\\
&\quad
+
\frac{4\sigma^2}{\sqrt{2}}
\sum_{i,l}\sum_{l_2<l}\sum_{k_2<i}\sum_{l_2'<l}
c_{ii}^{ll_2}c_{ik_2}^{ll_2'}
\bm{Z}_{iil_2}\bm{Z}_{ik_2l_2'}
\\
&\quad
+
\sigma^2
\sum_{i,l}
\sum_{\substack{k_2<i\\ k_2'\neq k_2}}
c_{ik_2}^{ll}c_{ik_2'}^{ll_2}
\bm{Z}_{ik_2l}\bm{Z}_{ik_2'l_2}
+
\frac{\sigma^2}{2}
\sum_{i,l}
\sum_{\substack{k_2<i\\ k_2'\neq k_2}}
\sum_{l_2<l}\sum_{l_2'<l}
c_{ik_2}^{ll_2}c_{ik_2'}^{ll_2'}
\bm{Z}_{ik_2l_2}\bm{Z}_{ik_2'l_2'}
\\
&\quad
+
\sigma^2\frac{2}{\sqrt{2}}
\sum_{i,l}\sum_{k<i}\sum_{l_2<l}\sum_{l_3\ge l}
(c_{kk}^{ll_2}+c_{ii}^{ll_2})c_{ik}^{l_3l}
\bm{Z}_{ikl_2}\bm{Z}_{kkl_3}
\\
&\quad
+
2\sigma^2\sum_{i,l}\sum_{k<i}\sum_{l_2<l}\sum_{k_2<k}
(c_{kk}^{ll_2}+c_{ii}^{ll_2})c_{kk_2}^{ll}
\bm{Z}_{ikl_2}\bm{Z}_{ik_2l}
\\
&\quad
+
4\sigma^2\sum_{i,l}\sum_{k<i}\sum_{l_2<l}\sum_{k_2<k}\sum_{l_2'<l}
(c_{kk}^{ll_2}+c_{ii}^{ll_2})c_{kk_2}^{ll_2'}
\bm{Z}_{ikl_2}\bm{Z}_{ik_2l_2'}
\\
&\quad
+
\sigma^2\frac{2}{\sqrt{2}}
\sum_{i,l}\sum_{k<i}\sum_{l_3\ge l}\sum_{k_2<k}
c_{ik}^{l_3l}c_{kk_2}^{ll}
\bm{Z}_{kkl_3}\bm{Z}_{ik_2l}
\\
&\quad
+ \sigma^2\frac{4}{\sqrt{2}}
\sum_{i,l}\sum_{k<i}\sum_{l_3\ge l}\sum_{k_2<k}\sum_{l_2<l}
c_{ik}^{l_3l}c_{kk_2}^{ll_2}
\bm{Z}_{kkl_3}\bm{Z}_{ik_2l_2}
\end{align*}
We claim that we can write this as $\bm{Z}\t \bm{QQ}\t\bm{Z}$, where $\bm{Z}$ is the vector associated to all unique components of $\bm{Z}_{ikl}$ and $\bm{Q}$ is a matrix defined as follows. First, recall the set $\mathcal{I} := \{ (i,k,l), k\leq i, l \leq L\}$.  Define $\mathcal{J}_1 := \{ (i,l,l_2) : l_2 < l \}$ and $\mathcal{J}_2 := \{ (i,l,k,l_2) : k < i,\, l_2 < l \}$, with $\mathcal{J} = \mathcal{J}_1 \cup \mathcal{J}_2$. We write $\bm{Q} \in \mathbb{R}^{|\mathcal{I}| \times |\mathcal{J}|}$ with entries given by
\begin{align*}
\bm{Q}_{(i',k',l'),(i,l,l_2)} &=
\begin{cases}
2\sigma\,c_{ii}^{ll_2}, & (i',k',l')=(i,i,l_2),\\[1mm]
\dfrac{\sigma}{\sqrt{2}}\,c_{ik'}^{ll}, & i'=i,\, k'<i,\, l'=l,\\[1mm]
\dfrac{\sigma}{\sqrt{2}}\,c_{ik'}^{ll'}, & i'=i,\, k'<i,\, l'<l,\\[1mm]
0, & \text{otherwise},
\end{cases} \\[1mm]
\bm{Q}_{(i',k',l'),(i,l,k,l_2)} &=
\begin{cases}
 \sigma c_{kk}^{ll_2}+ \sigma c_{ii}^{ll_2}, & (i',k',l')=(i,k,l_2),\\[1mm]
\dfrac{\sigma}{\sqrt{2}}\,c_{ik}^{l'l}, & (i',k',l')=(k,k,l'),\, l'\ge l,\\[1mm]
 \sigma c_{kk_2}^{ll}, & i'=i,\, k'<k,\, l'=l,\\[1mm]
2\, \sigma c_{kk_2}^{ll'}, & i'=i,\, k'<k,\, l'<l,\\[1mm]
0, & \text{otherwise}.
\end{cases}
\end{align*}
With this definition, we note that $\mathbb{E} T_3 = \sigma^2 \| \bm{Q} \|_F^2 $, and, moreover, $T_3 = \bm{Z}\t \bm{QQ}\t \bm{Z}$.  Therefore, we may apply the Hanson-Wright inequality to $\bm{Z}\t \bm{QQ}\t \bm{Z}$.  
In particular, it holds that
\begin{align*}
    \mathbb{P} \bigg\{ \bigg| \bm{Z}\t \bm{QQ}\t \bm{Z} - \sigma^2 \| \bm{Q} \|_F^2 \bigg| >  \sigma^2 \|\bm{Q}\|_F^2 t \bigg\} \leq 2 \exp\bigg\{ -c \frac{\|\bm{Q}\|_F^2}{\|\bm{Q}\|^2} \min(t^2,t) \bigg\}.
\end{align*}
Thus, convergence in probability is achieved if $\|\bm{Q}\|_F^2/\|\bm{Q}\|^2 \to \infty$.  We now upper bound $\| \bm{Q}\|^2$.  First, we have that $\|\bm{Q}\|^2 \leq \|\bm{Q}\|_{1,1} \| \bm{Q}\|_{\infty,\infty}$ (these are operator norms with respect to these vector spaces).  Note that every entry of $\bm{Q}$ satisfies
\begin{align*}
    \max |\bm{Q}_{(i'k'l'),(i,l,l_2)}| \leq C \sigma \max_{ikll'} |c_{ik}^{ll'}|.
\end{align*}
We simply count the number of nonzeros per row and column of $\bm{Q}$.  Note that the only time that $\bm{Q}$ has nonzero entries within a row it must be that $i' = i$.  Thus, we may only vary $k'$ and $l'$, meaning there are at most $O(nL)$ nonzero entries.  Similarly for the columns, yielding
\begin{align*}
    \| \bm{Q}\|^2 &\leq C \sigma^2 n^2 L^2 \max_{ikll'}|c_{ik}^{ll'}|^2 \\
    &\lesssim C n^2 L^2 \| \U\|_{2,\infty}^4 \frac{\kappa^2 r^2}{L^4 \lambda^4} \asymp \| \U\|_{2,\infty}^4 \frac{\sigma^2 n^2}{L^2 \lambda^4}.
\end{align*}
Thus,
\begin{align*}
    \frac{\|\bm{Q}\|_F^2}{\|\bm{Q}\|^2} \gtrsim \frac{\sigma^2 n \| \R \|_F^2}{\| \U\|_{2,\infty}^4 \frac{\sigma^2 n^2}{L^2 \lambda^4}} \gtrsim \frac{\frac{\sigma^2 n}{\lambda^4 L^2}}{\| \U\|_{2,\infty}^4 \frac{\sigma^2 n^2}{L^2 \lambda^4}} \asymp \frac{1}{\|\bm{U}\|_{2,\infty}^4 n},
\end{align*}
which diverges by \eqref{variancelowerbound}.
\end{itemize}
Combining our results for $T_1$, $T_2$, and $T_3$, we have proven \eqref{martingale1}. 
We now prove \eqref{martingale2}.  We will first bound the quantity $\sum_{t} \mathbb{E} Y_t^4$.  We recall that
\begin{align*}
    Y_t = A_t (\bm{Z}_t^2 - \sigma^2) + B_t \bm{Z}_t,
\end{align*}
where $B_t$ can be written as $\sum_{s < t} a_{s} \bm{Z}_s$, where $a_{s}$ are deterministic coefficients on each $\bm{Z}_s$.  Thus, by subgaussianity, we have that
\begin{align*}
    \mathbb{E} Y_t^4 &= A_t^4 \bigg( \sigma_8^8 - 4 \sigma^2 \sigma_6^6 + 6 \sigma^4 \sigma_4^4 - 3 \sigma^8 \bigg) + 6 A_t^2 (\sigma_4^4 - \sigma^4) \sigma^2  \mathbb{E} B_t^2 + \sigma_4^4 \mathbb{E} B_t^4 \\
    &\lesssim \sigma^8 A_t^4 + \sigma^6 A_t^2 \mathbb{E} B_t^2 + \sigma^4 \mathbb{E} B_t^4. \numberthis \label{newyears}
    \end{align*}
We first compute $\sum_{t} \mathbb{E} B_t^4$. We have 
\begin{align*}
\mathbb{E}[B_t^4] = 
\begin{cases} 
\sigma_4^4 \Bigg[
\sum_{l_2<l} 16 (c_{ii}^{ll_2})^4 
+ \sum_{k_2<i} \frac{1}{4} (c_{ik_2}^{ll})^4 
+ \sum_{k_2<i} \sum_{l_2<l} \frac{1}{4} (c_{ik_2}^{ll_2})^4
\Bigg] \\
\qquad + 6 \sigma^4 \bigg[
\sum_{l_2<l} \sum_{l_2'<l,\, l_2'<l_2} 16 (c_{ii}^{ll_2})^2 (c_{ii}^{ll_2'})^2 \\
\qquad + \sum_{l_2<l} \sum_{k_2<i} 2 (c_{ii}^{ll_2})^2 (c_{ik_2}^{ll})^2 
+ \sum_{l_2<l} \sum_{k_2<i} \sum_{l_2'<l} 2 (c_{ii}^{ll_2})^2 (c_{ik_2}^{ll_2'})^2 \\
\qquad + \sum_{k_2<i} \sum_{k_2'<i,\, k_2'<k_2} \frac{1}{4} (c_{ik_2}^{ll})^2 (c_{ik_2'}^{ll})^2 
+ \sum_{k_2<i} \sum_{k_2'<i} \sum_{l_2'<l} \frac{1}{4} (c_{ik_2}^{ll})^2 (c_{ik_2'}^{ll_2'})^2 \\
\qquad + \sum_{k_2<i} \sum_{l_2<l} \sum_{k_2'<i} \sum_{l_2'<l,\, (k_2',l_2')<(k_2,l_2)} \frac{1}{4} (c_{ik_2}^{ll_2})^2 (c_{ik_2'}^{ll_2'})^2
\bigg], & k=i
\\
\sigma_4^4 \Bigg[
\sum_{l_2<l} (c_{kk}^{ll_2}+c_{ii}^{ll_2})^4 
+ \sum_{l_3 \ge l} \frac{1}{4} (c_{ik}^{l_3 l})^4 
+ \sum_{k_2<k} (c_{kk_2}^{ll})^4 
+ \sum_{k_2<k} \sum_{l_2<l} 16 (c_{kk_2}^{ll_2})^4
\Bigg] \\
\quad + 6 \sigma^4 \Bigg[
\sum_{l_2<l} \sum_{l_2'<l,\, l_2'<l_2} (c_{kk}^{ll_2}+c_{ii}^{ll_2})^2 (c_{kk}^{ll_2'}+c_{ii}^{ll_2'})^2
+ \sum_{l_2<l} \sum_{l_3 \ge l} \frac{1}{2} (c_{kk}^{ll_2}+c_{ii}^{ll_2})^2 (c_{ik}^{l_3 l})^2 \\
\quad + \sum_{l_2<l} \sum_{k_2<k} (c_{kk}^{ll_2}+c_{ii}^{ll_2})^2 (c_{kk_2}^{ll})^2
+ \sum_{l_2<l} \sum_{k_2<k} \sum_{l_2'<l} 4 (c_{kk}^{ll_2}+c_{ii}^{ll_2})^2 (c_{kk_2}^{ll_2'})^2 \\
\quad + \sum_{l_3<l_3',\, l_3,l_3'\ge l} \frac{1}{4} (c_{ik}^{l_3 l})^2 (c_{ik}^{l_3' l})^2
+ \sum_{l_3 \ge l} \sum_{k_2<k} \frac{1}{2} (c_{ik}^{l_3 l})^2 (c_{kk_2}^{ll})^2 \\
\quad + \sum_{l_3 \ge l} \sum_{k_2<k, l_2<l} 2 (c_{ik}^{l_3 l})^2 (c_{kk_2}^{ll_2})^2
+ \sum_{k_2<k} \sum_{k_2'<k, k_2'<k_2} (c_{kk_2}^{ll})^2 (c_{kk_2'}^{ll})^2 \\
\quad + \sum_{k_2<k} \sum_{k_2'<k, l_2<l} 4 (c_{kk_2}^{ll})^2 (c_{kk_2'}^{ll_2})^2
+ \sum_{\substack{k_2<k,k_2'<k \\ l_2<l, l_2'<l, (k_2',l_2')<(k_2,l_2)}} 16 (c_{kk_2}^{ll_2})^2 (c_{kk_2'}^{ll_2'})^2
\Bigg], & k<i.
\end{cases}
\end{align*}
Note that each term above consists of at most $O(n^2 L^2)$ terms, each of which are bounded by $\max_{i,k,l_1,l_2} ( c_{ik}^{l_1l_2})^4$.  
We now sum this over $t$; i.e., over $k \leq i$ and $l \leq L$, which has at most $O(n^2 L)$ terms, giving
\begin{align*}
    \sum_{t} \mathbb{E} B_t^4 \lesssim n^4 L^3 \sigma^4 \max_{i,k,l_1,l_2} ( c_{ik}^{l_1l_2})^4.
\end{align*}
We recall that from the definition of the coefficients $c_{ik}^{l_1l_2}$ we have that
\begin{align*}
    \big| c_{ik}^{l_1l_2} \big|  &= \bigg| e_i\t \bm{U} \bm{R}^{(l_1)} \R \R \bm{R}^{(l_2)} \bm{U}\t e_k \bigg| \lesssim \| \U \|_{2,\infty}^2 \frac{\kappa^2 r}{\lambda^2 L^2}.
\end{align*}
We also note that from our previous calculation $\sigma^2 \sum_{t} \mathbb{E} B_t^2 \lesssim s^2$.  We further recall that by the definition of $A_t$, we have that $\sum_{t} A_t^4 \lesssim n^2 L \max_{i,k,l_1,l_2} (c_{ik}^{l_1l_2})^4$.  Thus, combining these observations with \eqref{newyears}, we have that 
\begin{align*}
    \sum_{t} \mathbb{E} Y_t^4 &\lesssim \sigma^8 n^2 L \max_{i,k,l_1,l_2} (c_{ik}^{l_1l_2})^4 + \sigma^6 \max_{i,k,l_1,l_2} (c_{ik}^{l_1l_2})^2 \sum_{t} \mathbb{E} B_t^2 + n^4 L^3 \sigma^8 \max_{i,k,l_1,l_2} (c_{ik}^{l_1l_2})^4 \\
    &\lesssim \sigma^8 n^2 L \| \U\|_{2,\infty}^8 \frac{\kappa^8 r^4}{\lambda^8 L^8} + \sigma^4 \| \U\|_{2,\infty}^4 \frac{\kappa^4 r^2}{\lambda^4 L^4} s^2 + n^4 L^4 \sigma^8 \| \U\|_{2,\infty}^8 \frac{\kappa^8 r^4}{\lambda^8 L^8} \\
    &\lesssim s^2 \| \U\|_{2,\infty}^4 \frac{\sigma^4 \kappa^4 r^2}{\lambda^4 L^4} + \|\U\|_{2,\infty}^8 \frac{\sigma^8 n^4 r^4 \kappa^8 }{\lambda^8 L^4} \\
    &\lesssim s^4 \bigg( \frac{L^2 \lambda^4}{n r \sigma^4} \| \U\|_{2,\infty}^4 \frac{\sigma^4 \kappa^4 r^2}{\lambda^4 L^4} + \frac{L^4 \lambda^8}{n^2 r^2 \sigma^8 }\|\U\|_{2,\infty}^8 \frac{\sigma^8 n^4 r^4 \kappa^8 }{\lambda^8 L^4} \bigg) \\
    &\asymp s^4 \bigg( \|\U\|_{2,\infty}^4 \frac{\kappa^4 r}{n L^4} + \| \U\|_{2,\infty}^8 n^2 r^2 \kappa^8 \bigg).
\end{align*}
Thus, we see that the term in the parentheses is $o(1)$ whenever \eqref{variancelowerbound} holds.
This completes the proof.\end{proof}

\subsubsection{Proof of \cref{lem:extraterm}} \label{sec:extratermproof}
\begin{proof}
First, we note that
\begin{align*}
    \| \U\t \sum_l \nl \sl \U (\R) \|_F \leq \frac{\sqrt{r}}{L\lambda^2} \| \U\t \sum_l \nl \sl \|.
\end{align*}
We now bound the quantity $\U\t \sum_l \nl \sl$.  The proof is markedly similar to the proof of \cref{lem:initialconcentration}.  Modifying the proof, by taking a net over the set of vectors of dimension $r$ we can show that
\begin{align*}
    \| \U\t \sum_l \nl \sl \| \leq C s \sigma \sqrt{L} \lambda_{\max}
\end{align*}
with probability at least $1 - \exp( c r - s^2)$.  Therefore, taking $s \asymp \sqrt{r} + t$ shows that 
\begin{align*}
      \| \U\t \sum_l \nl \sl \| \leq C \sigma (\sqrt{r} + t) \sqrt{L} \lambda_{\max}
\end{align*}
with probability at least $1 - \exp( - c t^2)$. 
\end{proof}

\subsection{Proof of \cref{thm:minimaxoptimalconfidenceinterval}} \label{sec:confidenceintervalproof}

We now provide the following result, which provides consistent estimators of the asymptotic mean and variance.
\begin{lemma}\label{lem:parameterconsistency}
Under the conditions of \cref{thm:normality}, with probability at least $1 - \exp( - c n)$ it holds that
\begin{align*}
 \frac{\frac{ \sigma^2  n}{2}  \bigg| {\sf Tr} \bigg( \R \bigg) - {\sf Tr} \bigg( \rhatinv \bigg) \bigg|}{\sigma^2 \sqrt{n/2} \| \R \|_F} &= o(1); \\
 \bigg| 1 - \frac{\| \rhatinv \|_F}{\|\R \|_F} \bigg| &=o(1).
\end{align*}
\end{lemma}
\begin{proof}
    See \cref{sec:parameterconsistencyproof}.  
\end{proof}

We are now prepared to prove \cref{thm:minimaxoptimalconfidenceinterval}.
\begin{proof}[Proof of \cref{thm:minimaxoptimalconfidenceinterval}]
First, we note that \cref{thm:normality} implies that
\begin{align*}
    \frac{\|\sin\Theta(\bm{\hat U},\bm{U})\|_F^2 - \frac{\sigma^2 n}{2} {\sf Tr}\big( \R \big) }{\sigma^2 \sqrt{n/2} \| \R \|_F} \xrightarrow{d} \mathcal{N}(0,1).
\end{align*}
Thus, it holds that
\begin{align*}
    \liminf_{n\to\infty} \mathbb{P}\bigg\{ \|\sin\Theta(\bm{\hat U},\bm{U})\|_F^2 \in \frac{\sigma^2 n}{2} {\sf Tr}\big( \R \big) \pm z_{\alpha/2}\sigma^2 \sqrt{n/2} \| \R \|_F \bigg\} \geq 1 - \alpha.
\end{align*} 
We also have that the interval above is of length of order at most $\frac{\sigma^2\sqrt{nr}}{L \lambda^2} \lesssim \frac{\sigma^2 nr}{L \lambda^2} \times \frac{1}{\sqrt{nr}}$, uniformly over all $\mathcal{P}(\lambda)$ where $\lambda$ satisfies $\lambda \gg \frac{n r^{2} \kappa^4}{\sqrt{L}}$. 
Define 
\begin{align*}
\mu := \frac{\sigma^2 n}{2} {\sf Tr} \big( \R \big); &\qquad 
    \hat \mu = \frac{\sigma^2 n}{2} {\sf Tr} \big( \rhatinv\big); \\
    \nu = \sigma^2 \sqrt{n/2} \|\R\|_F; &\qquad \hat \nu = \sigma^2 \sqrt{n/2} \| \rhatinv \|_F.
\end{align*}
Define
\begin{align*}
    Z_n := \frac{\|\sin\Theta(\bm{\hat U},\U)\|_F^2 - \mu}{\nu}; \qquad 
    a_n = \frac{\hat \mu - \mu}{\nu}; \qquad 
    b_n = \frac{\hat \nu}{\nu}.
\end{align*}
Then we have that
\begin{align*}
    \mathbb{P}\bigg\{ \| \sin\Theta(\bm{\hat U},\U)\|_F^2 \in \hat{{\sf CI}_{\alpha}} \bigg\} = \mathbb{P} \big\{ | Z_n - a_n| \leq z_{\alpha/2} b_n \big\}.
\end{align*}
Slutsky's theorem  and \cref{lem:parameterconsistency} implies that $Z_n - a_n \xrightarrow{d} \mathcal{N}(0,1)$ and $z_{\alpha/2} b_n \xrightarrow{p} z_{\alpha/2}$.  Furthermore, we have that $|a_n|= o(1), |b_n-1| = o(1)$.  Then we have that
\begin{align*}
    \liminf \mathbb{P}\big\{ | Z_n - a_n| \leq z_{\alpha/2} b_n \big\} \geq \liminf \mathbb{P} \big\{ |Z_n| \leq z_{\alpha/2}(1 - o(1)) - o(1) \big\} \to 1 - \alpha.
\end{align*}
Thus we have the desired coverage.  The bound on the length follows immediately from \cref{lem:parameterconsistency}.   
\end{proof}

\subsubsection{Proof of \cref{lem:parameterconsistency}} \label{sec:parameterconsistencyproof}
\begin{proof}[Proof of \cref{lem:parameterconsistency}]
From \cref{lem:rlsquared}, we have that 
\begin{align*}
    \bigg| \frac{\sigma^2 n}{2} {\sf Tr}\bigg( \R \bigg) - \frac{\sigma^2 n}{2} {\sf Tr} \bigg( \rhatinv \bigg)\bigg| &\leq \sqrt{r}  \frac{\sigma^2 n}{2} \bigg\| \bigg( \sum_l \big( \bm{\hat R}^{(l)} \big)^2 \bigg)\inv - \O\t \bigg( \sum_l \big( \rl \big)^2 \bigg)\inv \O \bigg\|_F \\
    &\lesssim \err \frac{\kappa^2 \sigma^2 n\sqrt{r}}{\lambda^2 L} + \err \frac{\sigma^4 n^2 r^{2}}{\lambda^4 L^{3/2}} + \frac{\sigma^4 n r^{2}}{\lambda^4 L^{3/2}}\\
    &\quad + \frac{\sigma^4 n^2 r^{2}}{\lambda^4 L^2} + \frac{\sigma^4 r^{2}n^{3/2}}{\lambda^4 L^{3/2}} + \err \frac{\sigma^3 r^{3/2} n^2 \kappa}{\lambda^3 L^2} + \frac{\sigma^3   n^{3/2} r \kappa}{L^{3/2} \lambda^3} \\
    &\asymp \frac{\sigma^3 \kappa^3 n^{3/2} r}{\lambda^3 L^{3/2}} + \frac{\sigma^5 \kappa n^{5/2} r^{5/2}}{\lambda^5 L^2}+ \frac{\sigma^4 n r^2}{\lambda^4 L^{3/2}} \\
    &\quad + \frac{\sigma^4 n^2 r^2}{\lambda^4 L^2} + \frac{\sigma^4 r^2 n^{3/2}}{\lambda^4 L^{3/2}} + \frac{\sigma^4 \kappa^2 n^{5/2} r^2}{\lambda^4 L^{5/2}} + \frac{\sigma^3 n^{3/2} r \kappa}{L^{3/2} \lambda^3}. \numberthis \label{12926}
\end{align*}
We will show that the right hand side of \eqref{12926} is $o\big( \sigma^2 n \| \R \|_F \big).$ Recall that $  \| \R \|_F^2 \gtrsim \frac{r }{L^2 \lambda^4}$. Thus, it suffices to show that multiplying $\eqref{12926}$ by $\frac{\lambda^2 L}{\sigma^2 \sqrt{nr}}$ yields a term that is $o(1)$.  

We have
\begin{align*}
 \frac{\lambda^2 L}{\sigma^2 \sqrt{nr}} &\bigg\{ \frac{\sigma^3 \kappa^3 n^{3/2} r}{\lambda^3 L^{3/2}} + \frac{\sigma^5 \kappa n^{5/2} r^{5/2}}{\lambda^5 L^2}+ \frac{\sigma^4 n r^2}{\lambda^4 L^{3/2}}  \frac{\sigma^4 n^2 r^2}{\lambda^4 L^2} + \frac{\sigma^4 r^2 n^{3/2}}{\lambda^4 L^{3/2}} + \frac{\sigma^4 \kappa^2 n^{5/2} r^2}{\lambda^4 L^{5/2}} + \frac{\sigma^2 n^{3/2} r \kappa}{L^{3/2} \lambda^3}\bigg\} \\
    &\qquad \asymp \frac{\sigma\kappa^3 n \sqrt{r}}{\lambda \sqrt{L}} + \frac{\sigma^3 \kappa n^2 r^2}{\lambda^3 L} + \frac{\sigma^2 \sqrt{n} r^{3/2}}{\lambda^2 \sqrt{L}} + \frac{\sigma^2 n^{3/2} r^{3/2}}{\lambda^2 L} + \frac{\sigma^2 r^{3/2} n}{\lambda^2 \sqrt{L}} + \frac{\sigma^2 \kappa^2 n^2 r^2}{\lambda^2 L^{3/2}} + \frac{\sigma n \sqrt{r} \kappa}{\sqrt{L} \lambda} = o(1),
\end{align*}
where the final bound is $o(1)$ since $\lambda/\sigma \gg \frac{n r^2 \kappa^4}{\sqrt{L}}$ and $L \lesssim n$.    

Next, by \eqref{rhatinvbd},
\begin{align*}
    \big| \| \R \|_F - \| \rhatinv \|_F \big| &\leq    \bigg\| \bigg( \sum_l \big( \bm{\hat R}^{(l)} \big)^2 \bigg)\inv - \O\t \bigg( \sum_l \big( \rl \big)^2 \bigg)\inv \O \bigg\|_F\\ &\lesssim  \err \frac{\kappa^2}{\lambda^2 L} + \err \frac{\sigma^2 n r^{3/2}}{\lambda^4 L^{3/2}} + \frac{\sigma^2 r^{3/2}}{\lambda^4 L^{3/2}}\\
    &\quad + \frac{\sigma^2 n r^{3/2}}{\lambda^4 L^2} + \frac{\sigma^2 r^{3/2} \sqrt{n}}{\lambda^4 L^{3/2}} + \err \frac{\sigma r \sqrt{n} \kappa}{\lambda^3 L^2} + \frac{\sigma \sqrt{nr} \kappa}{L^{3/2} \lambda^3}.
    \end{align*}
    Thus,
    \begin{align*}
        \bigg|  1 - \frac{\| \rhatinv \|_F}{\| \R \|_F} \bigg| &\lesssim \frac{L \lambda^2}{\sqrt{r}} \bigg( \err \frac{\kappa^2}{\lambda^2 L} + \err \frac{\sigma^2 n r^{3/2}}{\lambda^4 L^{3/2}} + \frac{\sigma^2 r^{3/2}}{\lambda^4 L^{3/2}}\\
    &\quad + \frac{\sigma^2 n r^{3/2}}{\lambda^4 L^2} + \frac{\sigma^2 r^{3/2} \sqrt{n}}{\lambda^4 L^{3/2}} + \err \frac{\sigma r \sqrt{n} \kappa}{\lambda^3 L^2} + \frac{\sigma \sqrt{nr} \kappa}{L^{3/2} \lambda^3}\bigg) \\
    &\asymp \frac{\sigma \kappa^3 \sqrt{n}}{\lambda \sqrt{L}} + \frac{\sigma^3 \kappa n^{3/2} r^{3/2}}{\lambda^3 \sqrt{L}} + \frac{\sigma^2 r}{\lambda^2 \sqrt{L}} + \frac{\sigma^2 n r}{\lambda^2 L} + \frac{\sigma^2 r \sqrt{n}}{\lambda^2 \sqrt{L}} + \frac{\sigma^2 \kappa^2 r n}{\lambda^2 L^{3/2}} + \frac{\sigma \sqrt{n} \kappa}{\sqrt{L} \lambda} \\
    &= o(1),
    \end{align*}
    where the final bound holds since $\lambda/\sigma \gg \frac{n r^2 \kappa^4}{\sqrt{L}}$. 
    This completes the proof.
\end{proof}

\section{Proofs of Lower Bounds}

In this section we prove all of our lower bounds.  First, in \cref{sec:minimaxlowerboundproof} we prove \cref{thm:minimax}.   In \cref{sec:compgapproof} we prove the computational lower bound of \cref{thm:computational}, and finally in \cref{sec:ci_minimaxlowerboundproof} we prove \cref{thm:ci_minimaxlowerbound} and \cref{thm:ci_minimax2} in tandem.

\subsection{Proof of \cref{thm:minimax}}
\label{sec:minimaxlowerboundproof}

\begin{proof}[Proof of \cref{thm:minimax}]
The proof is a standard Fano's inequality argument.  Consider the metric space $\mathcal{B}_{n-r,r} = \{ \bm{U}: \bm{U} \in \mathbb{O}_{n-r,r}\}$ with the metric $\| \bm{U}_1\bm{U}_1\t - \bm{U}_2 \bm{U}_2\t\|_F$. Let $M(a)$ denote the packing number of this metric space. By Lemma 5 of \citet{koltchinskii_optimal_2015}, the packing number $M(\sqrt{r}\eps)$ of this set for $r \leq n - 2r$ satisfies
\begin{align*}
    \bigg( \frac{c}{\eps} \bigg)^{r(n-r)} \leq M( \sqrt{r}\eps) \leq \bigg( \frac{C}{\eps} \bigg)^{r(n-r)}.
\end{align*}
Since $r \leq c \sqrt{n}, 3r \leq n$ for $n$ sufficiently large.  
By taking $\eps = c/2$, we can find a subset $\mathcal{U} \subset \mathcal{B}_{n-r,r}$ with $|\mathcal{U}|\geq 2^{r(n- 2r)}$ such that for any $\bm{U}_i \neq \bm{U}_j$ in $\mathcal{U}$ it holds that
\begin{align*}
    \| \U_i \U_i\t - \U_j \U_j\t \|_F \geq \frac{c}{2} \sqrt{r}.
\end{align*}
For each $\bm{U}_i \in \mathcal{U}$, define $\bm{\tilde U}_i$ via
\begin{align*}
    \bm{\tilde U}_i = \begin{pmatrix}
        \sqrt{1 - \delta}\bm{I}_r \\ \sqrt{\delta} \bm{U}_i 
    \end{pmatrix},
\end{align*}
which is an orthonormal matrix.  We thus have that if $\bm{\tilde U}_i \neq \bm{\tilde U}_j$, then
\begin{align*}
    \| \bm{\tilde U}_i \bm{\tilde U}_i\t - \bm{\tilde U}_j \bm{\tilde U}_j\t \|_F &\geq \sqrt{2 \delta(1 - \delta)} \| \bm{U}_i - \bm{U}_j \|_F \geq \sqrt{2 \delta(1- \delta)} \inf_{\mathcal{O}: \mathcal{OO}\t = \bm{I}_r} \| \| \bm{U}_i - \bm{U}_j \mathcal{O}\|_F \\
    &\geq \sqrt{\delta(1 - \delta)} \| \bm{U}_i \bm{U}_i\t - \bm{U}_j \bm{U}_j\t \|_F \\
    &\geq \frac{c}{2} \sqrt{r \delta(1- \delta)}.
\end{align*}
We now let $\bm{S}^{(l)}_i = \lambda \bm{\tilde U}_i \bm{\tilde U}_i\t$ for all $l$.  Then under ${\sf GOE}$ noise with variance $\sigma^2$ it holds that
\begin{align*}
    KL\bigg( \big\{ \bm{S}^{(l)}_i \big\}, \big\{ \bm{S}^{(l)}_j \big\} \bigg) &= \frac{L}{2 \sigma^2} \| \lambda \bm{\tilde U}_i \bm{\tilde U}_i\t - \lambda \bm{\tilde U}_j \bm{\tilde U}_j\t \|_F^2 \\
    &= \frac{L \lambda^2}{2\sigma^2} \| \bm{\tilde U}_i \bm{\tilde U}_i\t - \bm{\tilde U}_j \bm{\tilde U}_j\t \|_F^2 \\
    &\leq C\frac{ \lambda^2 L}{\sigma^2} \| \bm{\tilde U}_i - \bm{\tilde U}_j \|_F^2 \\
    &\leq C \frac{\lambda^2 L \delta r}{\sigma^2}.
\end{align*}
Thus, by the generalized Fano's lemma, it holds that
\begin{align*}
    \inf_{\bm{\hat U}} \sup_{\bm{U}\in \{\bm{\tilde U}_i\}_{i=1}^{m} } \mathbb{E} \| \bm{\hat U} \bm{\hat U}\t - \bm{UU}\t \|_F^2 \geq c \delta ( 1 - \delta) r \bigg( 1 - \frac{C \frac{\lambda^2}{\sigma^2} L\delta r + \log(2)}{r(n- 2r) \log(2)} \bigg).
\end{align*}
Since $\| \sin\Theta(\bm{U},\bm{V})\|_F^2 = \frac{1}{2} \| \bm{UU}\t - \bm{VV}\t \|_F^2$, by taking $\delta = c_0 \big( \frac{\sigma^2 n}{\lambda^2 L} \wedge 1\big)$ for some sufficiently small constant $c_0$, we obtain the result.  
\end{proof}

\subsection{Proof of \cref{thm:computational}} \label{sec:compgapproof}
\begin{proof}
We follow a similar proof to \citet{lyu_optimal_2023}.  First, by section A.2 of \citet{kunisky_notes_2022}, we can without loss of generality assume that $\nl$ consists of IID $\mathcal{N}(0,\sigma^2)$ noise and  $\sigma^2 = 1$ (by dividing through by $\sigma$).  

Following \citet{lyu_optimal_2023}, we let the prior distribution for $(\bm{M},\bm{\eps})$ be denoted as $\Pi$. 
By Theorem 2.6 of \citet{kunisky_notes_2022}, it holds that
\begin{align*}
    \| L_n^{\leq D} \|^2 = 1 + \mathbb{E}_{\Pi} \sum_{k=1}^{\lfloor D/2 \rfloor } \frac{1}{(2k)!} \langle \bm{\eps}^{(1)}, \bm{\eps}^{(2)} \rangle^{2k} \langle \bm{M}^{(1)}, \bm{M}^{(2)} \rangle^{2k},
\end{align*}
where $(\bm{M}^{(1)},\bm{\eps}^{(1)})$ is independent from $(\bm{M}^{(2)},\bm{\eps}^{(2)})$.  We note that equation (44) of \citet{lyu_optimal_2023} continues to apply in our setting, yielding
\begin{align*}
    \mathbb{E} \langle \bm{\eps}^{(1)}, \bm{\eps}^{(2)} \rangle^{2k} = {L +  k -1 \choose k}.
\end{align*}
We now note a slight difference from their proof: we have that
\begin{align*}
    \langle \bm{M}^{(1)}, \bm{M}^{(2)} \rangle &= \lambda^2 \langle \bm{u}_1 \bm{u}_1\t, \bm{u}_2 \bm{u}_2\t \rangle = \frac{\lambda^2}{n^2} \bigg( \sum_{i=1}^{n} U_i^{(1)} U_i^{(2)} \bigg)^2,
\end{align*}
where $U_i^{(1)}$ and $U_i^{(2)}$ are two independent collections of Rademachers.  However, we note that the product of two Rademacher random variables is again Rademacher, and hence we have that
\begin{align*}
    \mathbb{E} \bigg( \sum_{i} U_i \bigg)^{4k} \leq n^{2k} (4k - 1)!!.
\end{align*}
Thus,
\begin{align*}
    \mathbb{E} \langle \bm{M}^{(1)}, \bm{M}^{(2)} \rangle^{2k} &= \frac{\lambda^{4k}}{n^{4k}} \mathbb{E} \bigg( \sum_{i=1}^{n} U_i \bigg)^{4k} \leq \frac{\lambda^{4k}}{n^{4k}} n^{2k} (4k- 1)!! = \frac{\lambda^{4k}}{n^{2k}}(4k-1)!!.
\end{align*}
Thus, we have that
\begin{align*}
    \| L_n^{\leq D} \|^2 = 1+ \sum_{k=1}^{\lfloor D/2 \rfloor } \frac{(4k-1)!!}{(2k)!} {L +  k -1 \choose k} \frac{\lambda^{4k}}{n^{2k}} \leq  1+ \sum_{k=1}^{\lfloor D/2 \rfloor }  {L +  k -1 \choose k} \frac{4^k \lambda^{4k}}{n^{2k}} =: 1+ \sum_{k=1}^{\lfloor D/2 \rfloor } T_k.
\end{align*}
We have that
\begin{align*}
    \frac{T_{k+1}}{T_k} = \frac{4\lambda^4}{n^2} \frac{{L+1 \choose k+1}}{{L + k - 1\choose k}} = \frac{4 \lambda^4}{n^2} \frac{L+k}{k+1} \lesssim \frac{\lambda^4 L}{n^2} \leq \frac{1}{2},
\end{align*}
provided that $\lambda^2 \lesssim \frac{n}{\sqrt{L}}$, which follows from the assumption that $\lambda = o\big( \frac{\sqrt{n}}{L^{1/4}}\big)$.  We thus have that
\begin{align*}
    \| L_n^{\leq D} \|^2 \leq 1 + O\bigg( \frac{\lambda^4 L}{n^2} \bigg) = 1 + o(1),
\end{align*}
as required.
\end{proof}

\subsection{Proof of \cref{thm:ci_minimaxlowerbound,thm:ci_minimax2}} \label{sec:ci_minimaxlowerboundproof}

\begin{proof}
We will prove both results in tandem. In particular, we will identify a construction of priors $\mathcal{P}_1^*$ and $\mathcal{P}_2^*$ where $\mathcal{P}_1^*$ and $\mathcal{P}_2^*$ lie in $\mathcal{\tilde P}(\lambda_1)$ and $\mathcal{\tilde P}(\lambda_2)$ respectively, and, furthermore, if $\lambda_1 = \lambda_2$, then both lie in $\mathcal{P}(\lambda_1)$.    Therefore, we focus on \cref{thm:ci_minimax2}.

For simplicity, let $\mathcal{P}_1 = \mathcal{P}(\lambda_1)$ and define $\mathcal{P}_2$ similarly.  
To prove \cref{thm:ci_minimax2} we need to prove the two inequalities
\begin{align}
 \mathcal{L}_{\alpha}^*(\mathcal{P}_1,\mathcal{P}_2) \geq\frac{\sigma^2 n r}{L\lambda_1^2} \frac{\sigma \sqrt{nr}}{\sqrt{L}\lambda_2}; 
   \label{inequality1} \\
   \mathcal{L}_{\alpha}^*(\mathcal{P}_1,\mathcal{P}_2) \geq  C \frac{\sigma^2 n r}{L\lambda_1^2} \frac{1}{\sqrt{nr}}. \label{inequality2}
\end{align}
We will prove both statements separately under a common framework.  Given a set $\mathcal{P}^*_1 \subset \mathcal{P}_1$ and an alternative hypothesis space $\mathcal{P}_2^* \subset \mathcal{P}_2 \subset \mathcal{P}_1$, we will show that for all $\U \in \mathcal{P}_2^*$ it holds that
\begin{align*}
    {\sf Tr}\bigg( \per\s \per^{*\top} \bm{U}\bm{U} \bigg) = \mu,
\end{align*}
for an appropriate choice of $\mu$, where $\U\s \in \mathcal{P}_1^*$ satisfies $\bm{U}^{*\top}\U_{\perp}\s= 0$. We further note that
\begin{align*}
    \| \pert \uhat \|_F^2 &= {\sf Tr} \big( \pert \uhat \uhat\t \per \big) = {\sf Tr} \bigg( \per \pert \uhat \uhat\t \bigg).
\end{align*}
Therefore, since $\|\sin\Theta(\U,\uhat)\|_F^2 = \| \pert \uhat \|_F^2$, any confidence interval for $\|\sin\Theta(\U,\uhat)\|_F^2$ is also a confidence interval for ${\sf Tr} \bigg( \per \pert \uhat \uhat\t \bigg)$.  Therefore, we have that 
\begin{align*}
    \mu = \| \sin\Theta(\bm{U}\s, \bm{U}_1) \|_F^2.
\end{align*}
Let $f_{\pi_{\mathcal{P}_1\s}}$ be the density function for the uniform prior over $\mathcal{P}_1\s$, with a similar definition for $\mathcal{P}_2\s$.  
By Lemma 1 of \citet{cai_confidence_2017} it holds that  
\begin{align}
   \mathcal{L}_{{\sf CI}_{\alpha}}(\mathcal{P}^*_1) &\geq \sup_{\theta \in \mathcal{P}^*_2 \cup \mathcal{P}^*_1} \mathbb{E}_{\theta} {\sf L}( {\sf CI}_{\alpha}) \geq | \mu | \bigg( 1 - 2\alpha - {\sf TV}\big( f_{\pi_{\mathcal{P}\s_1}}, f_{\pi_{\mathcal{P}\s_2}} \big) \bigg)_+. \label{caiguo}
\end{align}
Therefore, the result is proven if we can lower bound $\mu$ and select an appropriate $\mathcal{P}_1\s$ such that ${\sf TV}(f_{\pi_{\mathcal{P}_1\s}}, f_{\pi_{\mathcal{P}_2\s}}) \leq c$, where $c$ is a sufficiently small constant.   By the inequality 
\begin{align}
    {\sf TV}( f_1, f_2) \leq \frac{1}{2} \sqrt{\chi^2(f_1,f_2)}, \label{tv}
\end{align}
we can simply upper bound the $\chi^2$ divergence.

First we consider the inequality \eqref{inequality1} and give our construction. Since $r \geq 2$, set $r' := \lfloor r/2 \rfloor$, and set $r_0 = r - r'$.  Let
\begin{align*}
    \bm{W} = \begin{pmatrix}
        \bm{I}_{r_0} \\ \bm{0}_{(n- r_0 )\times r_0}
    \end{pmatrix} \in \mathbb{R}^{n \times r_0}.
\end{align*}
In addition, define
\begin{align*}
    \bm{V}\s := \begin{pmatrix}
        \bm{0}_{r_0 \times r'} \\ \bm{I}_{r'} \\ \bm{0}_{(n - r) \times r'}
    \end{pmatrix} \in \mathbb{R}^{n\times r'}.
\end{align*}
Then set
\begin{align*}
    \bm{U}\s := [\bm{W}, \bm{V}\s] \in \mathbb{R}^{n\times r}.
\end{align*}
Note that $\bm{U}\s$ is orthonormal.  Now, suppose that $\bm{Z} \in \mathbb{R}^{(n- r)\times r'}$ satisfies $\| \bm{Z}\|_F = 1$.  Set, for some parameter $\rho$ to be chosen momentarily, 
\begin{align*}
    \bm{V}_{\bm{Z}} := \begin{pmatrix}
        \bm{0}_{r_0 \times r'} \\ \bm{I}_{r'} \\ \rho \bm{Z}
    \end{pmatrix} ( \bm{I}_{r'} + \rho^2 \bm{Z}\t \bm{Z} )^{-1/2} \in \mathbb{R}^{n\times r'},
\end{align*}
and define $\bm{U}_{\bm{Z}} := [\bm{W}, \bm{V}_{\bm{Z}}]$, which is orthonormal by construction. Under the null $\mathcal{P}^*_1$ we define
\begin{align*}
    \bm{S}_*^{(l)} = \bm{S}_{*} = \lambda' \bm{WW}\t + \lambda_2 \bm{V}\s \bm{V}^{*\top},
\end{align*}
where $\lambda'$ is chosen such that $r_0 (\lambda')^2 + r' \lambda_2^2 = r \lambda_1^2$.  We note that in the case $\lambda_2 = \lambda_1$, then $\lambda' = \lambda_1$, and $\{\bm{S}_*^{(l)}\}_{l=1}^{L}$  all lie in the parameter space $\mathcal{P}(\lambda_1)$. 
Similarly, we set
\begin{align*}
    \bm{S}_{\bm{Z}}^{(l)} \equiv \bm{S}_{\bm{Z}} = \lambda' \bm{WW}\t + \lambda_2 \bm{V}_{\bm{Z}} \bm{V}_{\bm{Z}}\t.
\end{align*}
The matrices $\bm{S}_*$ and $\bm{S}_{\bm{Z}}$ have subspaces $\bm{U}\s$ and $\bm{U}_{\bm{Z}}$ respectively, and both matrices lie in our parameter space.  Moreover, since $\bm{W}$ is shared, it holds that
\begin{align*}
    \| \sin\Theta(\bm{U}_{\bm{Z}}, \bm{U}\s ) \|_F^2 = \| \sin\Theta(\bm{V}_{\bm{Z}}, \bm{V}\s) \|_F^2 = \tr\big( \rho^2 \bm{Z}\t \bm{Z}( \bm{I}_{r'} + \rho^2 \bm{Z}\t \bm{Z})\inv \big).
\end{align*}
Since $\|\bm{Z}\|_F= 1$, it is straightforward to demonstrate that $\frac{\rho^2}{1 + \rho^2} \leq  \| \sin\Theta(\bm{V}_{\bm{Z}},\bm{V}\s)\|_F^2 \leq \rho^2$.  Assuming that $\rho \ll 1$ yields that $\|\sin\Theta(\bm{U}_{\bm{Z}}, \bm{U}\s ) \|_F^2 \gtrsim \rho^2$.  

We let $\pi_{\mathcal{P}^*_2}$ be the uniform prior over matrices $\bm{Z}$ satisfying $\|\bm{Z}\|_F  = 1$.  We claim that
\begin{align*}
    \chi^2(f_{\pi_{\mathcal{P}_1\s}}, f_{\pi_{\mathcal{P}_2\s}}) + 1 = \mathbb{E}_{\bm{Z},\bm{Z}'} \exp\bigg\{ \frac{L \lambda_2^2}{\sigma^2} \langle \bm{V}_{\bm{Z}} \bm{V}_{{\bm{Z}}}\t - \bm{V}\s \bm{V}^{*\top}, \bm{V}_{\bm{Z}'}\bm{V}_{\bm{Z}'}\t - \bm{V}\s \bm{V}^{*\top} \rangle \bigg\}. \numberthis \label{chisquareclaim}
\end{align*}
Assuming this for the moment, we upper bound the right hand side above.  We have that
\begin{align*}
    \bm{V}_{\bm{Z}} \bm{V}_{\bm{Z}}\t &= \begin{pmatrix}
        \bm{0} & \bm{0} & \bm{0} \\
        \bm{0} & ( \bm{I}_{r'} + \rho^2 \bm{Z}\t \bm{Z})\inv & \rho( \bm{I}_{r'} + \rho^2 \bm{Z}\t \bm{Z})\inv  \bm{Z}\t\\
        \bm{0} & \rho \bm{Z}  ( \bm{I}_{r'} + \rho^2 \bm{Z}\t \bm{Z})\inv & \rho^2 \bm{Z} ( \bm{I}_{r'} + \rho^2 \bm{Z}\t \bm{Z})\inv \bm{Z}\t 
    \end{pmatrix}.
\end{align*}
Thus,
\begin{align*}
     \bm{V}_{\bm{Z}}& \bm{V}_{\bm{Z}}\t - \bm{V}\s \bm{V}^{*\top}\\
     &= \begin{pmatrix}
        \bm{0} & \bm{0} & \bm{0} \\
        \bm{0} & ( \bm{I}_{r'} + \rho^2 \bm{Z}\t \bm{Z})\inv - \bm{I}_{r'} & \rho( \bm{I}_{r'} + \rho^2 \bm{Z}\t \bm{Z})\inv  \bm{Z}\t\\
        \bm{0} & \rho \bm{Z}  ( \bm{I}_{r'} + \rho^2 \bm{Z}\t \bm{Z})\inv & \rho^2 \bm{Z} ( \bm{I}_{r'} + \rho^2 \bm{Z}\t \bm{Z})\inv \bm{Z}\t 
    \end{pmatrix} \\
    &= \begin{pmatrix}
        \bm{0} & \bm{0} & \bm{0} \\
        \bm{0} & \bm{0} & \rho( \bm{I}_{r'} + \rho^2 \bm{Z}\t \bm{Z})\inv  \bm{Z}\t\\
        \bm{0} & \rho \bm{Z}  ( \bm{I}_{r'} + \rho^2 \bm{Z}\t \bm{Z})\inv & \bm{0} 
    \end{pmatrix}  + \begin{pmatrix}
        \bm{0} & \bm{0} & \bm{0} \\
        \bm{0} & ( \bm{I}_{r'} + \rho^2 \bm{Z}\t \bm{Z})\inv - \bm{I}_{r'} & \bm{0} \\
        \bm{0} & \bm{0} & \rho^2 \bm{Z} ( \bm{I}_{r'} + \rho^2 \bm{Z}\t \bm{Z})\inv \bm{Z}\t 
    \end{pmatrix} \\
    &:= \bm{X}_{\bm{Z}}^{(1)} + \bm{X}_{\bm{Z}}^{(2)}.
\end{align*}
It therefore holds that
\begin{align*}
   \langle \bm{V}_{\bm{Z}} \bm{V}_{{\bm{Z}}}\t - \bm{V}\s \bm{V}^{*\top}, \bm{V}_{\bm{Z}'}\bm{V}_{\bm{Z}'}\t - \bm{V}\s \bm{V}^{*\top} \rangle &= \langle \bm{X}_{\bm{Z}}^{(1)}, \bm{X}_{\bm{Z}'}^{(1)} \rangle + \langle \bm{X}_{\bm{Z}}^{(2)}, \bm{X}_{\bm{Z}'}^{(2)} \rangle + \langle \bm{X}_{\bm{Z}}^{(1)},\bm{X}_{\bm{Z}'}^{(2)} \rangle  + \langle \bm{X}_{\bm{Z}'}^{(1)},\bm{X}_{\bm{Z}}^{(2)} \rangle. 
\end{align*}
We analyze each term separately.  First, we have that
\begin{align*}
    \langle \bm{X}_{\bm{Z}}^{(1)}, \bm{X}_{\bm{Z}'}^{(1)} \rangle &= 2 \rho^2 \tr\bigg[\bm{Z}\t \big( \bm{I}_{r'} + \rho^2 \bm{Z}\t \bm{Z} \big)\inv  \big( \bm{I}_{r'} + \rho^2 \bm{Z}^{'\top} \bm{Z}' \big)\inv \bm{Z}^{'\top} \bigg].
\end{align*}
Since $\rho \ll 1$ and $\|\bm{Z}\t \bm{Z}\|_F \leq 1$, we may write the Neumann series
\begin{align*}
    \big( \bm{I}_{r'} + \rho^2 \bm{Z}\t \bm{Z} \big)\inv &= \sum_{k=0}^{\infty}(-1)^k \rho^{2k} (\bm{Z}\t \bm{Z})^{k}.
\end{align*}
Plugging this series in yields
\begin{align*}
    \langle \bm{X}_{\bm{Z}}^{(1)}, \bm{X}_{\bm{Z}'}^{(1)} \rangle &= 2 \rho^2 \langle \bm{Z}, \bm{Z}' \rangle + R,
\end{align*}
where $|R| \leq C \rho^4$.   Using a similar argument for $\bm{X}_{\bm{Z}}^{(2)}$ shows that the remaining terms are all bounded by $C \rho^4$.  As a result, \eqref{chisquareclaim} yields
\begin{align*}
    \mathbb{E}_{\bm{Z},\bm{Z}'} \exp\bigg\{ \frac{L \lambda_2^2}{\sigma^2} \langle \bm{V}_{\bm{Z}} \bm{V}_{{\bm{Z}}}\t - \bm{V}\s \bm{V}^{*\top}, \bm{V}_{\bm{Z}'}\bm{V}_{\bm{Z}'}\t - \bm{V}\s \bm{V}^{*\top} \rangle \bigg\} &\leq \exp\bigg\{ \frac{C L \lambda_2^2 \rho^4}{\sigma^2} \bigg\}\mathbb{E}_{\bm{Z},\bm{Z'}} \exp\bigg\{ \frac{2 L \lambda_2^2 \rho^2}{\sigma^2} \langle\bm{Z},\bm{Z}' \rangle \bigg\}. 
\end{align*}
Next, we note that $\bm{Z}$ and $\bm{Z}'$ are matrices drawn uniformly from $\|\bm{Z}\|_F = 1$; thus, the inner product $\langle \bm{Z}, \bm{Z}'\rangle$ is equivalent to the inner product for two independent uniform random variables of dimension $(n - r) r'$.  Since the inner product between two independent uniform random variables is subgaussian with $\psi_2$ norm at most $\frac{C}{\sqrt{(n-r)r'}}$, it holds that if $r' \asymp r$,
\begin{align*}
     \exp\bigg\{ \frac{C L \lambda_2^2 \rho^4}{\sigma^2} \bigg\}\mathbb{E}_{\bm{Z},\bm{Z'}} \exp\bigg\{ \frac{2 L \lambda_2^2 \rho^2}{\sigma^2} \langle\bm{Z},\bm{Z}' \rangle \bigg\} &\leq  \exp\bigg\{ \frac{C L \lambda_2^2 \rho^4}{\sigma^2} \bigg\}   \exp\bigg\{ C \frac{L^2 \lambda_2^4 \rho^4}{\sigma^4 nr} \bigg\}.  
\end{align*}
Now, take
\begin{align*}
    \rho^2 &= c_{\rho} \frac{\sigma^2 nr}{L \lambda_1^2}\frac{\sigma \sqrt{nr}}{\lambda_2 \sqrt{L}}.
\end{align*} 
Plugging this in yields
\begin{align*}
   \chi^2(\mathcal{P}_1\s, \mathcal{P}_2\s) + 1 &\leq 
    \exp\bigg\{ \frac{C c_\rho^2 L \lambda_2^2}{\sigma^2} \bigg( \frac{\sigma^4 n^2r^2 }{L^2 \lambda_1^4}\frac{\sigma^2 nr }{\lambda_2^2 L }\bigg) \bigg\} + C  c_{\rho}^2 \frac{L^2 \lambda_2^4 }{\sigma^4 nr} \bigg( \frac{\sigma^4 n^2r^2 }{L^2 \lambda_1^4}\frac{\sigma^2 nr }{\lambda_2^2 L }\bigg) \bigg\} \\
    &\leq \exp\bigg\{ C c_{\rho}^2 \frac{1}{C_0^4} + C c_{\rho}^2 \frac{1}{C_0^2} \bigg\},
\end{align*}
where the final bound follows since $\lambda_1/\sigma \geq C_0 \frac{nr}{\sqrt{L}}$ and $\lambda_2 \leq \lambda_1$. If $C_0$ is sufficiently large or $c_{\rho}$ is sufficiently small,  the above is bounded by $1 + \eps^2$ for some small $\eps$.

Similarly, if $\lambda_2 \geq \frac{nr}{\sqrt{L}}$, we then take $\rho^2 =c_{\rho} \frac{\sigma^2 \sqrt{nr}}{L\lambda_1^2}$ to yield
\begin{align*}
    \exp\bigg\{ \frac{C c_{\rho}^2 L \lambda_2^2 \sigma^4 nr}{\sigma^2 \lambda_1^4 L^2} \bigg\} \exp\bigg\{ C c_{\rho}^2 \frac{L^2 \lambda_2^4 \sigma^4 nr}{\lambda_1^4 L^2 nr \sigma^4}\bigg\} &= \exp\bigg\{ \frac{C c_{\rho}^2 \lambda_2^2 \sigma^2 nr}{\lambda_1^4 L}\bigg\} \exp\bigg\{ C c_{\rho}^2 \frac{\lambda_2^4}{\lambda_1^4} \bigg\}.
\end{align*}
This remains bounded whenever $\lambda_1 \geq C_0 \sigma \sqrt{nr/L}$, which it is by assumption (indeed, we assume that $\lambda_1/\sigma \gg \frac{nr}{\sqrt{L}}$).  This completes the proof of both inequalities provided we can justify \eqref{chisquareclaim}.

We now justify the $\chi^2$ calculation. Recall that  $f_{\pi_{\mathcal{P}_1\s}}$ denotes the joint law of
$\{\al\}_{l=1}^L$ under the null parameter $\bm{S}_\ast^{(l)}$, and for each $\bm{Z}$ let
$f_{\bm{Z}}$ denote the joint law under the alternative parameter $\bm{S}_{\bm{Z}}^{(l)}$. We have that
\begin{align*}
\bm{S}_{\bm{Z}}^{(l)} - \bm{S}_{*}^{(l)} &= 
=
\lambda_2 \big( \bm{V}_{\bm{Z}} \bm{V}_{\bm{Z}}\t - \bm{V}^* \bm{V}^{*\top} \big).
\end{align*}
Let $\pi$ be the distribution of $Z$.  Then we have that
\[
f_{\pi_{\mathcal{P}_2\s}}:=\int f_{\bm{Z}}\,d\pi(\bm{Z}).
\]
Hence
\[
\chi^2(f_{\pi_{\mathcal{P}_2\s}},f_{\pi_{\mathcal{P}_1\s}})+1
=
\int\int
\mathbb E_{f_{\pi_{\mathcal{P}_1\s}}}
\left[
\frac{ f_{\bm{Z}}}{ f_{\pi_{\mathcal{P}_1\s}}}
\frac{ f_{\bm{Z}'}}{f_{\pi_{\mathcal{P}_1\s}}}
\right]
d\pi(\bm{Z})d\pi(\bm{Z}').
\]
We now compute the inner expectation. For fixed $\bm{Z}$, since the noise is Gaussian,
\begin{align*}
\frac{f_{\bm{Z}}}{f_{\pi_{\mathcal{P}\s}}}
&=
\exp\left\{
\frac{1}{\sigma^2}
\sum_{l=1}^L
\left\langle \nl ,\lambda_2 \big( \bm{V}_{\bm{Z}} \bm{V}_{\bm{Z}}\t - \bm{V}^* \bm{V}^{*\top} \big)\right\rangle
-
\frac{1}{2\sigma^2}
\sum_{l=1}^L
\|\lambda_2 \big( \bm{V}_{\bm{Z}} \bm{V}_{\bm{Z}}\t - \bm{V}^* \bm{V}^{*\top} \big)\|_F^2
\right\} \\
&= 
\exp\left\{\frac{\lambda_2}{\sigma^2}
\sum_{l=1}^L
\left\langle \nl , \bm{V}_{\bm{Z}} \bm{V}_{\bm{Z}}\t - \bm{V}^* \bm{V}^{*\top} \right\rangle
-
\frac{L \lambda_2^2 }{2\sigma^2}
\| \big( \bm{V}_{\bm{Z}} \bm{V}_{\bm{Z}}\t - \bm{V}^* \bm{V}^{*\top} \big)\|_F^2
\right\}.
\end{align*}
Multiplying the two likelihood ratios yields
\begin{align*}
\frac{f_{\bm{Z}}}{f_{\pi_{\mathcal{P}\s}}}
\frac{f_{\bm{Z}'}}{f_{\pi_{\mathcal{P}\s}}}
&=
\exp\left\{
\frac{\lambda_2}{\sigma^2}
\sum_{l=1}^L
\left\langle \nl ,\bm{V}_{\bm{Z}} \bm{V}_{\bm{Z}}\t + \bm{V}_{\bm{Z}'} \bm{V}_{\bm{Z}'}\t - 2\bm{V}\s \bm{V}^{*\top} \right\rangle
\right\} \\
&\quad \times \exp\bigg\{ -
\frac{L\lambda_2^2}{2\sigma^2}
\left(
\|\bm{V}_{\bm{Z}'} \bm{V}_{\bm{Z}'}\t - \bm{V}\s \bm{V}^{*\top}\|_F^2+\|\bm{V}_{\bm{Z}'} \bm{V}_{\bm{Z}'}\t - \bm{V}\s \bm{V}^{*\top}\|_F^2
\right)\bigg\}.
\end{align*}
It is straightforward to demonstrate that for any matrix $\bm{M}$, 
\begin{align*}
    \mathbb{E} \bigg\{ \frac{\lambda_2}{\sigma^2} \langle \nl, \bm{M} \rangle \bigg\} &= \exp\bigg\{ \frac{\lambda_2^2}{2\sigma^2} \| \bm{M} \|_F^2 \bigg\}.
\end{align*}
Thus, by independence over $l$, we obtain 
\begin{align*}
    \mathbb{E}_{f_{\pi_{\mathcal{P}\s}}} \frac{f_{\bm{Z}}}{f_{\pi_{\mathcal{P}\s}}}
\frac{f_{\bm{Z}'}}{f_{\pi_{\mathcal{P}\s}}} &= \exp\bigg\{  \frac{\lambda_2^2 L}{2\sigma^2} \| \bm{V}_{\bm{Z}} \bm{V}_{\bm{Z}}\t + \bm{V}_{\bm{Z}'} \bm{V}_{\bm{Z}'}\t - 2\bm{V}\s \bm{V}^{*\top} \|_F^2 \bigg\} \\
&\quad \times \exp\bigg\{- \frac{L\lambda_2^2}{2\sigma^2} \bigg( \| \bm{V}_{\bm{Z}} \bm{V}_{\bm{Z}}\t - \bm{V}\s \bm{V}^{*\top} \|_F^2 + \| \bm{V}_{\bm{Z}} \bm{V}_{\bm{Z}}\t - \bm{V}\s \bm{V}^{*\top} \|_F^2 \bigg) \bigg\} \\
&= \exp\bigg\{ \frac{\lambda_2^2 L}{2\sigma^2} \langle \bm{V}_{\bm{Z}} \bm{V}_{\bm{Z}}\t - \bm{V}\s \bm{V}^{*\top}, \bm{V}_{\bm{Z}'} \bm{V}_{\bm{Z}'}\t - \bm{V}\s \bm{V}^{*\top} \rangle \bigg\}.
\end{align*}
This completes the proof.
\end{proof}

\bibliography{gosie}

\end{document}